\documentclass[12pt,leqno]{amsart}
\usepackage{amsmath,stmaryrd}
\usepackage{amssymb}
\usepackage{amsthm}
\usepackage{epsf}
\usepackage{graphicx}
\usepackage{esint} 
\usepackage{dsfont}
\usepackage[pagebackref]{hyperref} 
\hypersetup{pdfpagemode=FullScreen,  colorlinks=true} 
\usepackage{enumerate} 

\numberwithin{equation}{section}

\newcommand{\vp}{\varphi}
\newcommand{\dr}{\partial}

\DeclareMathOperator{\Jac}{Jac}

\DeclareMathOperator{\diver}{div}
\DeclareMathOperator{\sgn}{sgn}
\DeclareMathOperator{\dist}{dist}
\DeclareMathOperator{\supp}{supp}
\DeclareMathOperator{\diam}{diam}
\DeclareMathOperator{\diag}{diag}

\DeclareMathOperator{\Tr}{Tr}
\DeclareMathOperator{\Angle}{Angle}
\def\div{\mathop{\operatorname{div}}}
\DeclareMathOperator{\loc}{loc}

\newcommand{\1}{{\mathds 1}}
\newcommand{\ds}{\displaystyle}
\newcommand{\ms}{\medskip}

\newcommand{\R}{\mathbb R}
\newcommand{\N}{\mathbb N}

\renewcommand{\H}{\mathcal H}
\newcommand{\bp}{\noindent {\em Proof: }}
\newcommand{\ep}{\hfill $\square$ \medskip}
\newcommand{\sm}{\setminus}
\newcommand{\wt}{\widetilde}

\newcommand{\D}{\mathbb D}

\newcommand{\bb}{\mathfrak b}
\renewcommand{\aa}{\mathfrak a}

\newcommand{\A}{\mathcal A}
\newcommand{\C}{\mathcal C}
\newcommand{\G}{\mathcal G}
\newcommand{\B}{\mathcal B}
\newcommand{\cS}{\mathcal S}

\newcommand{\Rn}{\mathbb R^n}
\newcommand{\norm}[1]{\left\Vert#1\right\Vert}
\newcommand{\abs}[1]{\left\vert#1\right\vert}
\newcommand{\br}[1]{\left(#1\right)}
\newcommand{\set}[1]{\left\{#1\right\}}

\newcommand{\om}{\Omega}
\newcommand{\pom}{\partial\Omega}

\newcommand{\dint}{\int\!\!\!\!\!\int}
\def\Yint#1{\mathchoice
	{\YYint\displaystyle\textstyle{#1}}%
	{\YYint\textstyle\scriptstyle{#1}}%
	{\YYint\scriptstyle\scriptscriptstyle{#1}}%
	{\YYint\scriptscriptstyle\scriptscriptstyle{#1}}%
	\!\dint}
\def\YYint#1#2#3{{\setbox0=\hbox{$#1{#2#3}{\iint}$}
		\vcenter{\hbox{$#2#3$}}\kern-.51\wd0}}
\def\longdash{\mkern-1.5mu{-}\mkern-7.5mu{-}} 

\def\fiint{\Yint\longdash}

\theoremstyle{plain}
\newtheorem{theorem}[equation]{Theorem}
\newtheorem{lemma}[equation]{Lemma}
\newtheorem{corollary}[equation]{Corollary}
\newtheorem{proposition}[equation]{Proposition}

\newtheorem{definition}[equation]{Definition}

\theoremstyle{definition}

\theoremstyle{remark}
\newtheorem{remark}[equation]{Remark}

\begin{document}

\title{Green functions and smooth distances}

\author[Feneuil]{Joseph Feneuil}
\address{Joseph Feneuil. Mathematical Sciences Institute, Australian National University, Acton, ACT, Australia}
\email{joseph.feneuil@anu.edu.au}

\author[Li]{Linhan Li}
\address{Linhan Li. School of Mathematics, The University of Edinburgh, Edinburgh, UK }
\email{linhan.li@ed.ac.uk}

\author[Mayboroda]{Svitlana Mayboroda}
\address{Svitlana Mayboroda. School of Mathematics, University of Minnesota, Minneapolis, MN 55455, USA}
\email{svitlana@math.umn.edu}

\thanks{S. Mayboroda was partly supported by the NSF RAISE-TAQS grant DMS-1839077 and the Simons foundation grant 563916, SM. J. Feneuil was partially supported by the Simons fundation grant 601941, GD and the ERC grant ERC-2019-StG 853404 VAREG.}

\maketitle

\begin{abstract} 
In the present paper, we show that for an optimal class of elliptic operators with non-smooth coefficients on a 1-sided Chord-Arc domain, the boundary of the domain is uniformly rectifiable if and only if the Green function $G$ behaves like a distance function to the boundary, in the sense that $\abs{\frac{\nabla G(X)}{G(X)}-\frac{\nabla D(X)}{D(X)}}^2D(X) dX$ is the density of a Carleson measure, where $D$ is a regularized distance adapted to the boundary of the domain. The main ingredient in our proof is a corona
decomposition that is compatible with Tolsa’s $\alpha$-number of uniformly rectifiable sets. We believe that the method can be applied to many other problems at the intersection of PDE and geometric measure theory, and in particular, we are able to derive a generalization of the classical F. and M. Riesz theorem to the same class of elliptic operators as above.
\end{abstract}

\ms\noindent{\bf Key words:} Uniform rectifiability, Chord-Arc domains, elliptic operators with non-smooth coefficients, Green functions, regularized distance, Dahlberg-Kenig-Pipher condition.

\ms\noindent

\tableofcontents

\section{Introduction}
\label{S1}

\subsection{Motivation and predecessors}
We consider elliptic operators $L$ on a domain $\om\subset\R^n$. In recent years a finale of an enormous body of work brought a characterization of uniform rectifiability in terms of absolute continuity of harmonic measure (see \cite{AHMMT}, a sample of earlier articles: \cite{DJ}, \cite{HM1}, \cite{HM2}, \cite{HMU}, \cite{AHM3TV}, \cite{Azzam}, see also the related article \cite{NTV} which proves the David-Semmes conjecture in codimension 1 and is a key step for the converse established in \cite{AHM3TV}). It also became clear that this characterization has its restrictions, for it fails in the domains with lower dimensional boundaries and it requires, in all directions, restrictions on the coefficients -- see a discussion in \cite{DM2}. In these contexts, the Green function emerged as a more suitable object to define the relevant PDE properties. Already the work in \cite{A} and \cite{DM2} suggested a possibility of a Green function characterization of regularity of sets. However, factually, \cite{DM2} provided more than satisfactory ``free boundary" results and only weak ``direct" results (no norm control). The papers  \cite{DLM1}, and \cite{DLM2}, and \cite{DFMGinfty} aimed at the desired quantitative version of such ``direct" results but were restricted to either Lipschitz graphs or sets with lower dimensional boundaries.
The primary goal of the article is to show that if $L$ is reasonably well-behaved, and $\om$ provides some access to its boundary, then the boundary of $\om$ is reasonably regular (uniformly rectifiable) if and only if the Green function behaves like a distance to the boundary.

Let us discuss some predecessors of this work, including the aforementioned ones, in more details. In \cite{A} Theorem VI, it is shown that the affine deviation of the Green function for the Laplace operator is related to the linear deviation of the boundary of the domain.
In \cite{DM2}, G.~David and the third author of the paper show that for a class of  elliptic operators, the Green function can be well approximated by distances to planes, or by a smooth distance to $\pom$, if and only $\pom$ is uniformly rectifiable. The bounds on the Green function given in \cite{DM2} are weak, more precisely, they carry no norm control of the sets where the Green function is close to a distance. Later, stronger and quantitative estimates on the comparison of the Green function and some distance functions are obtained in \cite{DFMGinfty}, \cite{DLM1}, and \cite{DLM2}. 
In \cite{DLM1}, a quantitative comparison between the Green function and the distance function to the boundary is given for an optimal class of elliptic operators on the upper half-space. Moreover, the proximity of the Green function and the distance function is shown to be precisely controlled by the oscillation of the coefficients of the operator. Next, \cite{DLM2} extends the result of \cite{DLM1} to $\R^n\setminus\R^d$ with $d$ strictly less than $n$. But the methods employed in \cite{DLM1} and \cite{DLM2} seem to the authors difficult to be adapted to domains whose boundaries are rougher than Lipschitz graphs. 
In \cite{DFMGinfty}, a bound for the difference of the Green function and smooth distances is obtained for sets with uniformly rectifiable boundaries, but its proof, which might appear surprising, is radically dependent on the fact that the boundary is of codimension strictly larger than 1. Also, the class of operators considered in \cite{DFMGinfty} is not optimal. So the instant motivation for the present work is to obtain a strong estimate on the Green function for an optimal class of operators, similar to the one considered in \cite{DLM1}, in a ``classical" setting: on domains with uniformly rectifiable boundaries of codimension 1. 
The method employed here is completely different from \cite{DFMGinfty} or \cite{DLM1}, and has the potential to be applicable to many other problems at the intersection of PDE and geometric measure theory.

We should also mention that in \cite{HMT} and \cite{DLM1}, some Carleson measure estimates on the {\it second derivatives} of the Green function have been obtained, and that in \cite{A}, the second derivative of the Green function for the Laplace operator is linked to the regularity (uniform rectifiability) of the boundary of the domain. However, the result of \cite{A} is only for the Laplace operator, the class of elliptic operators considered in \cite{HMT} is more general but still
not optimal, and the estimates obtained in \cite{DLM1} are restricted to Lipschitz graph domains.
We think that our estimates might shed some light on proving an estimate on second derivatives of the Green function for an optimal class of elliptic operators on chord-arc domains. 

For the ``free boundary" direction, since the weak type property of the Green function considered in \cite{DM2} already implies uniform rectifiablity, one expects the strong estimate on the Green function that we consider in this paper to automatically give uniform rectifiability. However, linking the two conditions directly seems to be more subtle than it might appear, and we actually need to obtain uniform rectifiablity from scratch. We point out that our result also holds for bounded domains, and thus dispensing with the unboundedness assumption on the domain in \cite{DM2}.

All in all, this paper is a culmination of all of the aforementioned efforts, featuring a true equivalence (characterization) of geometry through PDEs, and an optimal class of operators.

\subsection{Statements of the main results.}
We take a domain $\Omega \subset \R^n$ whose boundary $\partial \Omega$ is $(n-1)$-Ahlfors regular (AR for shortness), which means that there exists a measure $\sigma$ supported on $\partial \Omega$ such that 
\begin{equation} \label{defADR}
C_\sigma^{-1}r^{n-1} \leq \sigma(B(x,r)) \leq C_\sigma r^{n-1} \qquad \text{ for } x\in \partial \Omega, \, r \in (0, \diam \Omega).
\end{equation}
The domain $\om$ can be bounded or unbounded. In the unbounded case, $\diam\om=\infty$.
In the rest of the paper, $\sigma$ will always be an Ahlfors regular measure on $\partial \Omega$. It is known that the Ahlfors regular measures are the ones that can be written as $d\sigma = w d\mathcal H^{n-1}|_{\partial \Omega}$, where $\mathcal H^{n-1}|_{\partial \Omega}$ is the $n-1$ dimensional Hausdorff measure on $\partial \Omega$, and $w$ is a weight in $L^\infty(\partial \Omega, \mathcal H^{n-1}|_{\partial \Omega})$ such that $C^{-1} \leq w \leq C$ for some constant $C>0$.

We shall impose more assumptions on our domain. For both the ``free boundary" and the ``direct" results, we will assume that $\Omega$ is a 1-sided Chord Arc Domain (see Definition \ref{defi:CAD}). For the ``direct" result, we will rely on the assumption that $\partial \Omega$ is uniformly rectifiable (see \cite{DS1,DS2} and Section \ref{SUR} below), and thus ultimately assuming that $\Omega$ is a (2-sided) Chord Arc Domain. The optimality of the assumptions on $\om$ is discussed in more details in the end of this subsection. Since the dimension $n-1$ plays an important role in our paper, and in order to lighten the notion, we shall write $d$ for $n-1$.

\ms

Without any more delay, let us introduce the regularized distance to a set $\partial \Omega$. The Euclidean distance to the boundary is denoted by
\begin{equation}
\delta(X):= \dist(X,\partial \Omega).
\end{equation} 
For $\beta >0$, we define 
\begin{equation} \label{IdefD}
D_\beta(X):= \left(\int_{\dr \Omega} |X-y|^{-d-\beta} d\sigma(y) \right)^{-1/\beta} \qquad \text{ for } X\in\Omega.
\end{equation}
The fact that the set $\dr \Omega$ is $d$-Ahlfors regular is enough to have that 
\begin{equation} \label{equivD}
C^{-1} \delta(X) \leq D_\beta(X) \leq C\delta(X) \quad \text{ for } X\in\om,
\end{equation}
where the constant depends on $C_\sigma$, $\beta$, and $n$. The proof is easy and can be found after Lemma 5.1 in \cite{DFM3}.

\medskip

The notion of Carleson measure will be central all over our paper. We say that a quantity $f$ defined on $\Omega$ satisfies the Carleson measure condition - or $f\in CM_\Omega(M)$ for short - if there exists $M$ such that for any $x\in \dr \Omega$ and $r<\diam(\Omega)$,
\begin{equation} \label{defCarleson}
\iint_{B(x,r) \cap \Omega} |f(X)|^2 \delta(X)^{-1} dX \leq M r^{n-1}.
\end{equation}

\ms

Our operators are in the form $L=-\div \A \nabla$ and defined on $\Omega$. We shall always assume that they are uniformly elliptic and bounded, that is, there exists $C_{\A}>1$ such that
\begin{equation} \label{defelliptic}
\A(X)\xi \cdot \xi \geq C_{\A}^{-1} |\xi|^2 \qquad \text{ for } X\in \Omega, \, \xi \in \R^n,
\end{equation}
and
\begin{equation} \label{defbounded}
|\A(X)\xi \cdot \zeta| \leq C_{\A} |\xi||\zeta| \qquad \text{ for } X\in \Omega, \, \xi,\zeta \in \R^n.
\end{equation}
A weak solution to $Lu=0$ in $E \subset \Omega$ lies in $W^{1,2}_{loc}(E)$ and is such that
\begin{equation} \label{defsol}
\int_\Omega \A\nabla u \cdot \nabla \varphi \, dX \qquad \text{ for } \varphi \in C^\infty_0(E). 
\end{equation}
If $\Omega$ has sufficient access to the boundary (and $\partial \Omega$ is $(n-1)$-Ahlfors regular), then for any ball $B$ centered on $\partial \Omega$ and any function $u$ in $W^{1,2}(B\cap \Omega)$, we have notion of trace for $u$ on $B\cap \partial \Omega$. It is well known that if $u \in W^{1,2}(B\cap \Omega)$ is such that $\Tr(u) = 0$ on $B\cap \partial \Omega$, and if $u$ is a weak solution to $Lu=0$ on $B\cap \Omega$ with $L$ satisfying \eqref{defelliptic} and \eqref{defbounded}, then $u$ is continuous $B\cap \Omega$ and can be continuously extended by 0 on $B\cap \partial \Omega$. 

In addition to \eqref{defelliptic} and \eqref{defbounded}, we assume that our operators satisfy a weaker variant of the Dahlberg-Kenig-Pipher condition. The Dahlberg-Kenig-Pipher (DKP) condition was introduced by Dahlberg and shown to be sufficient for the $L^p$ solvability of the Dirichlet problem for some $p>1$ by Kenig and Pipher (\cite{KePiDrift}). It was also shown to be essentially necessary in \cite{CFK,MM}. The DKP condition says that the coefficient matrix $\A$ satisfies 
\begin{equation}\label{DKP}
    \delta(\cdot)\sup_{B(\cdot,\,\delta(\cdot)/2)}\abs{\nabla\A}\in CM_\om(M) \qquad\text{for some }M<\infty.
\end{equation}

Our assumption, slightly weaker than the classical DKP, is as follows.
\begin{definition}[Weak DKP condition]
An elliptic operator $L=-\div \A \nabla$ is said to satisfy the weak DKP condition with constant $M$ on $\om$ if there exists a decomposition $\A = \B+ \C$ such that
\begin{equation} \label{Main1a}
|\delta\nabla \B| + |\C| \in CM_\Omega(M).
\end{equation}
\end{definition}
Obviously, this condition is weaker than \eqref{DKP}: it allows for small Carleson perturbations and carries no supremum over the Whitney cubes. Moreover, we show in Lemma \ref{LellipB=ellipA} that the weak DKP condition self improves.

We are now ready for the statement of our main result.
\begin{theorem} \label{Main1}
Let $\beta >0$, $\Omega \subset \Rn$ be a 1-sided Chord-Arc Domain, and $L=-\div \A \nabla$ be a uniformly elliptic operator -- i.e., that verifies \eqref{defelliptic} and \eqref{defbounded} -- which satisfies the weak DKP condition with constant $M$ on $\om$. We write $G^{X_0}$ for the Green function of $L$ with pole at $X_0$.
The following are equivalent:
\begin{enumerate}[(i)]
    \item $\Omega$ is a Chord-Arc Domain,
    \item $\partial \Omega$ is uniformly rectifiable,
    \item there exists $C\in (0,\infty)$ such that for any ball $B$ centered on the boundary, and for any positive weak solution $u$ to $Lu=0$ in $2B \cap \Omega$ for which $\Tr u = 0$ on $2B \cap \dr \Omega$, we have
    \begin{equation} \label{Main1b}
    \iint_{\Omega \cap B} \left| \frac{\nabla u}{u} - \frac{\nabla D_\beta}{D_\beta} \right|^2 D_\beta \, dX \leq C \sigma(B),
    \end{equation}
    \item there exists $C\in (0,\infty)$  such that for any $X_0 \in \Omega$ and for any ball $B$ centered on the boundary satisfying $X_0 \notin 2B$, we have
    \begin{equation} \label{Main2b}
    \iint_{\Omega \cap B} \left| \frac{\nabla G^{X_0}}{G^{X_0}} - \frac{\nabla D_\beta}{D_\beta} \right|^2 D_\beta \, dX \leq C \sigma(B),
    \end{equation}
    \item there exist $X_0\in \Omega$ and $C\in (0,\infty)$  such that for any ball $B$ centered on the boundary that satisfies $X_0\notin 2B$, we have \eqref{Main2b}.
\end{enumerate}
Moreover, the constants $C$ in \eqref{Main1b}--\eqref{Main2b} can be chosen to depend only on $C_\A$, $M$, the CAD constants of $\Omega$, $\beta$, and $n$. 
\end{theorem}

\begin{remark}
The bound \eqref{Main1b} is a local one, meaning for instance that the bound will hold with a constant $C$ independent of the $B$ and the solution $u$ as long as $\om$ is chord-arc locally in $2B$ (that is, we only need the existence of Harnack chains and of corkscrew points in $2B\cap\om$) and the uniformly elliptic operator $L$ satisfies the weak DKP condition in $2B$. 
\end{remark}

The equivalence $(i) \Longleftrightarrow (ii)$ is already well known, see Theorem 1.2 in \cite{AHMNT}. Moreover, $(iii) \implies (iv) \implies (v)$ is immediate. So we need only to prove $(ii)\implies(iii)$ and $(v)\implies(i)$ in Theorem \ref{Main1}. 

\medskip

When the domain is unbounded, we can use the Green function with pole at infinity instead of the Green function. 
The Green function with pole at infinity associated to $L$ is the unique (up to a multiplicative constant) positive weak solution to $Lu=0$ with zero trace. See for instance Lemma 6.5 in \cite{DEM} for the construction (\cite{DEM} treats a particular case but the same argument works as long as we have CFMS estimates, see Lemma \ref{LCFMS} below). So we have that:

\begin{corollary} \label{Main2}
Let $\beta$, $\Omega$ and $L$ be as in Theorem \ref{Main1}. If $\Omega$ is unbounded, the following are equivalent:
\begin{enumerate}[(a)]
\item $\Omega$ is a Chord-Arc Domain, 
\item $\partial \Omega$ is uniformly rectifiable,
\item there exists $C\in (0,\infty)$  such that for any ball $B$ centered on the boundary, we have
\begin{equation} \label{Main2d}
\iint_{\Omega \cap B} \left| \frac{\nabla G^{\infty}}{G^{\infty}} - \frac{\nabla D_\beta}{D_\beta} \right|^2 D_\beta \, dX \leq C \sigma(B),
\end{equation}
\end{enumerate}
\end{corollary}

\medskip

For our proof of the ``direct" result, we need the fact that, for the same operators, the $L$-elliptic measure is $A_\infty$-absolutely continuous with respect to $\sigma$.

\begin{theorem}\label{Main3}
Let $\Omega$ be a Chord-Arc Domain, and let $L=-\div \A \nabla$ be a uniformly elliptic operator -- i.e., that verifies \eqref{defelliptic} and \eqref{defbounded} -- which satisfies the weak DKP condition with constant $M$ on $\om$. 

Then the $L$-elliptic measure $\omega_L\in A_\infty(\sigma)$, i.e. there exists $C,\theta >0$ such that  given an arbitrary surface ball $\Delta=B\cap  \partial\Omega$, with $B=B(x,r)$, $x\in  \partial\Omega$, $0<r<\diam( \partial\Omega)$, and for every Borel set $F\subset\Delta$, we have that
\begin{equation}\label{defAinfty}
\frac{\sigma(F)}{\sigma(\Delta)} \leq C \left(\omega_L^{X_{\Delta}}(F)\right)^\theta,
\end{equation}
where $X_{\Delta}$ is a corkscrew point relative to $\Delta$ (see Definition \ref{def1.cork}).

The constants $C$ and $\theta$ - that are called the intrinsic constants in $\omega_L \in A_\infty(\sigma)$ - depend only on $C_\A$, $M$, the CAD constants of $\Omega$, and $n$. 
\end{theorem}

The above is known for operators satisfying the DKP condition \eqref{DKP} on Chord-Arc domains. In fact, it is shown in \cite{KePiDrift} that $\omega_L\in A_\infty(\sigma)$ for DKP operators on Lipschitz domains. But since the DKP condition is preserved in subdomains, and the Chord-Arc domains are well approximated by Lipschitz subdomains (\cite{DJ}), the $A_\infty$ property can be passed from Lipschitz subdomains to Chord-Arc domains (see \cite{JK}, or \cite{HMMTZ}). Moreover, combined with the stability of the $A_\infty$ property under Carleson perturbation of elliptic operators proved in \cite{CHMT}, it is also known for the elliptic operators $L=-\diver \A \nabla$ for which $\A = \B + \C$, where 
\begin{equation}\label{DKPhw}
   \sup_{B(\cdot,\,\delta(\cdot)/2)}\{\abs{\delta \nabla\B} + \abs{\C}\} \in CM_\om(M) \qquad\text{for some }M<\infty.
\end{equation}
However, to the best of our knowledge, the $A_\infty$-absolute continuity of the elliptic measure was not proved explicitly for elliptic operators satisfying the slightly weaker condition \eqref{Main1a}. 

We obtain Theorem \ref{Main3} as a consequence of the following result - which is another contribution of the article - and Theorem 1.1 in \cite{CHMT}.

\begin{theorem}\label{Main4}
Let $\om$ be a domain in $\R^n$ with uniformly rectifiable (UR) boundary of dimension $n-1$. Let $L$ be a uniformly elliptic operator which satisfies the weak DKP condition with constant $M$ on $\om$. Suppose that $u$ is a bounded solution of $Lu=0$ in $\om$. Then for any ball $B$ centered on the boundary, we have
\begin{equation}\label{Main4a}
    \iint_{\om\cap B}\abs{\nabla u(X)}^2\delta(X)dX\le C\norm{u}_{L^\infty(\om)}^2\sigma(B\cap\pom),
\end{equation}
where the constant $C$ depends only on $n$, $M$, and the UR constants of $\pom$. 
\end{theorem}

Notice that in this theorem, we completely dispense with the Harnack chain and corkscrew conditions (see Definitions \ref{def1.cork} and \ref{def1.hc}) for the domain. Previously, an analogous result was obtained for bounded {\it harmonic functions} in \cite{HMMDuke} (see also \cite{GMT} for the converse) and for DKP operators in \cite[Theorem 7.5]{HMMtrans}. The result for elliptic operators which satisfy the weak DKP condition is again not explicitly written anywhere. However, slightly changing the proofs of a combination of papers would give the result. For instance, we can adapt Theorem 1.32 in \cite{DFM3} to the codimension 1 case to prove Theorem \ref{Main3} in the flat case, then extending it to Lipschitz graph by using the change of variable in \cite{KePiDrift}, and finally proving \ref{Main3} to all complement of uniformly rectifiable domains by invoking Theorem 1.19 (iii) in \cite{HMMtrans}).

Here, we claim that we can directly demonstrate Theorem \ref{Main4} using a strategy similar to our proof of Theorem \ref{Main1}. In Section~\ref{SecPfofThm4}, we explain how to modify the proof of Theorem \ref{Main1} to obtain Theorem \ref{Main4}.
By \cite{CHMT} Theorem 1.1, assuming that $\om$ is 1-sided CAD, the estimate \eqref{Main4a} implies that $\omega_L\in A_\infty(\sigma)$, and therefore our Theorem \ref{Main3} follows from Theorem \ref{Main4}. 
 Note that the bound \eqref{defAinfty} is a characterization of $A_\infty$, see for instance Theorem 1.4.13 in \cite{KenigB}.

Let us discuss in more details our assumptions for Theorem \ref{Main1}.  In order to get the bound \eqref{Main1b}, we strongly require that the boundary $\partial \Omega$ is uniformly rectifiable and that the operator $L$ satisfies the weak DKP condition. We even conjecture that those conditions are necessary, that is, if $\partial \Omega$ is not uniformly rectifiable, then the bound \eqref{Main1b} holds for none of the weak DKP operators. 

The corkscrew condition and the Harnack chain condition (see Definitions \ref{def1.cork} and \ref{def1.hc}) are only needed for the proof of Lemma \ref{lemlogk} - where we used the comparison principle - and for the implication $(iii)\implies(i)$ in Theorem \ref{Main1}. However, since most of our intermediate results can be proved without those conditions and could be of interest in other situations where the Harnack chain is not assumed (like - for instance - to prove Theorem \ref{Main4}), {\bf we avoided to use the existence of Harnack chains and of corkscrew points in all the proofs and the intermediate results} except for Lemma \ref{lemlogk} and in Section \ref{Sconv}, even if it occasionally slightly complicated the arguments.

These observations naturally lead to the question about the optimality of our conditions on $\om$, and more precisely, whether we can obtain the estimate \eqref{Main1b} assuming only uniform rectifiablity. The answer is no, as we can construct a domain $\om$ which has uniformly rectifiable boundary but is only semi-uniform (see Definition \ref{defSUD}) where \eqref{Main1b} fails. More precisely, we prove in Section \ref{Scount} that:

\begin{proposition}
There exists a semi-uniform domain $\Omega$ and a positive harmonic function $G$ on $\Omega$ such that \eqref{Main1b} is false.
\end{proposition}

But of course, assuming $\Omega$ is a Chord-Arc Domain is not necessary for \eqref{Main1b} since \eqref{Main1b} obviously holds when $\Omega = \R^n \setminus \R^{n-1} = \R^n_+ \cup \R^n_-$, because we can apply Theorem \ref{Main1} to both $\Omega_+ = \R^n_+$ and $\Omega_- = \R^n_-$ and then sum.
\medskip

\subsection{Main steps of the proof of $(ii) \implies (iii)$}\label{Sproof}

In this section, we present the outline of the proof of $(ii) \implies (iii)$ in Theorem \ref{Main1}. More exactly, this subsection aims to link the results of all other sections of the paper in order to provide the proof.

The approach developed in this article is new and it is interesting by itself, because it gives an alternative to proofs that use projection and extrapolation of measures. Aside from Theorems \ref{Main1} and \ref{Main4}, we claim that our approach can be used to obtain a third proof of the main result from \cite{FenUR} and \cite{DM1}, which establishes the $A_\infty$-absolute continuity of the harmonic measure when $\Omega$ is the complement of a uniformly rectifiable set of low dimension and $L$ is a properly chosen degenerate operator.

\medskip

Let $\Omega$ and $L$ be as in the assumptions of Theorem \ref{Main1}, and let $\B$ and $\C$ denote the matrices in \eqref{Main1a}. Take $B := B(x_0,r)$ to be a ball centered at the boundary, that is $x_0 \in \partial \Omega$, and then a non-negative weak solution $u$ to $Lu=0$ in $2B \cap \Omega$ such that $\Tr(u) = 0$ on $B \cap \partial \Omega$.

\medskip

\noindent {\bf Step 1: From balls to dyadic cubes.} We construct a dyadic collection $\D_{\partial \Omega}$ of pseudo-cubes in $\partial \Omega$ in the beginning of Section \ref{SUR}, and a collection of Whitney regions $W_\Omega(Q),W_\Omega^*(Q)$ associated to $Q\in \D_{\partial \Omega}$ in the beginning Section \ref{SWhitney}. We claim that \eqref{Main1b} is implied by cubes 
\begin{equation} \label{Main1c}
I := \sum_{\begin{subarray}{c} Q \in \D_{\partial \Omega} \\ Q \subset Q_0 \end{subarray}} \iint_{W_\Omega(Q)} \left| \frac{\nabla u}{u} - \frac{\nabla D_\beta}{D_\beta} \right|^2 \delta \, dX \leq C \sigma(Q_0)
\end{equation}
for any cube $Q_0\in \D_{\partial \Omega}$ satisfying $Q_0 \subset \frac87B \cap \partial\Omega$ and $\ell(Q_0) \leq 2^{-8}r$. It follows from the definition of $W^*_{\Omega}(Q)$ \eqref{defWO*} that
\begin{equation} \label{Main1d}
W^*_{\Omega}(Q) \subset \frac74B \qquad \text{ for } Q \subset Q_0
\end{equation}
and $Q_0$ as above.

We take $\{Q_0^i\}\subset\D_{\pom}$ as the collection of dyadic cubes that intersects $B \cap \partial\Omega$ and such that $2^{-9}r < \ell(Q_0^i) \leq 2^{-8}r$. There is a uniformly bounded number of them, each of them satisfies $Q_0^i \subset \frac32B \cap \Omega$ and $\ell(Q_0^i) \leq 2^{-8}r$ and altogether, they verify
\[B  \cap \Omega \subset \{X\in B, \, \delta(X) > 2^{-9}r\} \bigcup \Bigl(\bigcup_{i} \bigcup_{Q\in \D_{\partial \Omega}(Q_0^i)} W_\om(Q) \Bigr).\]
The estimate \eqref{Main1b} follows by applying \eqref{Main1c} to each $Q_0^i$ - using \eqref{equivD} and \eqref{defADR} when needed - and \eqref{Main1e} below to $\{X\in B, \, \delta(X) > 2^{-9}r\}$.
\medskip

\noindent {\bf Step 2: Bound on a Whitney region.} In this step, we establish that if $E\subset \frac74B$ is such that $\diam(E) \leq K\delta(E)$, then
\begin{equation} \label{Main1e}
J_E:= \iint_{E} \left| \frac{\nabla u}{u} - \frac{\nabla D_\beta}{D_\beta} \right|^2 \delta \, dX \leq C_K \delta(E)^{n-1}.
\end{equation}
We could use Lemma \ref{lemflat} to prove this, but it would be like using a road roller to crack a nutshell, because it is actually easy. We first separate 
\[J_E \leq \iint_{E} \left| \frac{\nabla u}{u}\right|^2 \delta \, dX + \iint_{E} \left|\frac{\nabla D_\beta}{D_\beta} \right|^2 \delta \, dX := J^1_E + J^2_E.\]
We start with $J_E^2$. Observe that $|\nabla [D_\beta^{-\beta}]| \leq (d+\beta) D_{\beta+1}^{-\beta-1}$, so $|\frac{\nabla D_\beta}{D_\beta}| \lesssim \delta^{-1}$ by \eqref{equivD}. We deduce that $J_E^2 \lesssim |E| \delta(E)^{-1} \lesssim \delta(E)^{n-1}$ as desired. 
As for $J_E^1$, we construct 
\[E^*:= \{X\in \Omega, \dist(X,E) \leq \delta(E)/100\}\subset \frac{15}{8}B,\] and then the Harnack inequality (Lemma \ref{Harnack}) and the Caccioppoli inequality (Lemma \ref{Caccio}) yield that
\[J_E^1 \lesssim \delta(E) (\sup_{E^*} u)^{-2} \iint_{E} |\nabla u|^2 dX \lesssim \delta(E)^{-1} (\sup_{E^*} u)^{-2} \iint_{E^*} u^2 dX \leq \delta(E)^{-1}|E^*| \lesssim \delta(E)^{n-1}.\]
The bound \eqref{Main1e} follows.

\medskip

\noindent {\bf Step 3: Corona decomposition.} Let $Q_0$ as in Step 1. One can see that we cannot use \eqref{Main1e} to each $E=W_\Omega(Q)$ and still hope to get the bound \eqref{Main1c} for $I$. We have to use \eqref{Main1e} with parsimony. We first use a corona decomposition of $D_{\partial \Omega}(Q_0)$, and we let the stopping time region stops whenever $\alpha(Q)$ or the angle between the approximating planes are too big. We choose $0 < \epsilon_1 \ll \epsilon_0 \ll 1$ and Lemma \ref{Lcorona} provides a first partition of $\D_{\partial \Omega}$ into bad cubes $\B$ and good cubes $\G$ and then a partition of $\G$ into a collection of coherent regimes $\{\cS\}_{\cS \in \mathfrak S}$. 

Let $\B(Q_0):= \B \cap \D_{\partial \Omega}(Q_0)$ and then $\mathfrak S(Q_0) = \{\cS \cap \D_{\partial \Omega}(Q_0)\}$. 
Observe that $\mathfrak S(Q_0)$ contains the collection of $\cS \in \mathfrak S$ such that $Q(\cS) \subset Q_0$ and maybe {\bf one} extra element, in the case where $Q_0 \notin \B \cup \bigcup_{\cS \in \mathfrak S} Q(\cS)$, which is the intersection with $\D_{\partial \Omega}(Q_0)$ of the coherent regime $\cS \in \mathfrak S$ that contains $Q_0$. 
In any case $\mathfrak S(Q_0)$ is a collection of (stopping time) coherent regimes. In addition, Lemma \ref{Lcorona} shows that $\mathfrak S(Q_0)$ and $\B(Q_0)$ verifies the Carleson packing condition 
\begin{equation} \label{Main1f}
\sum_{Q\in \B(Q_0)} \sigma(Q) + \sum_{\cS \in \mathfrak S(Q_0)} \sigma(Q(\cS)) \leq C \sigma(Q_0).
\end{equation}
We use this corona decomposition to decompose the sum $I$ from \eqref{Main1c} as
\begin{multline} \label{Main1g} 
I  = \sum_{Q  \in \B(Q_0)} \iint_{W_\Omega(Q)} \left| \frac{\nabla u}{u} - \frac{\nabla D_\beta}{D_\beta} \right|^2 \delta \, dX + \sum_{\cS \in \mathfrak S(Q_0)}  \iint_{W_\Omega(\cS)} \left| \frac{\nabla u}{u} - \frac{\nabla D_\beta}{D_\beta} \right|^2 \delta \, dX \\ := I_1 + \sum_{\cS \in \mathfrak S(Q_0)} I_{\cS},
\end{multline}
where $W_\Omega(\cS) = \bigcup_{Q\in \cS} W_\Omega (Q)$. For each cube $Q\in \B(Q_0)$, the regions $W_\Omega(Q)$ are included in $\frac74B$ and verify $\diam(W_\Omega(Q)) \leq 8 \delta(W_\Omega(Q)) \leq 8\ell(Q)$, so we can use \eqref{Main1e} and we obtain 
\begin{equation} \label{Main1h}
I_1 \lesssim \sum_{Q\in \B(Q_0)} \ell(Q)^{n-1} \lesssim \sigma(Q_0).
\end{equation}
by \eqref{defADR} and \eqref{Main1f}. 

\medskip

\noindent {\bf Step 4: How to turn the estimation of $I_{\cS}$ into a problem on $\R^n_+$.} Now, we take $\cS$ in $\mathfrak S(Q_0)$, which is nice because $\partial \Omega$ is well approximated by a small Lipschitz graph $\Gamma_{\cS}$ around any dyadic cube of $\cS$ (see Subsection \ref{SSLipschitz} for the construction of $\Gamma_\cS$). For instance, fattened versions of our Whitney regions $W^*_\Omega(Q)$, $Q\in \cS$, are Whitney regions in $\R^n \setminus \Gamma_{\cS}$ (see Lemma \ref{LclaimPib}). More importantly, at any scale $Q\in \cS$, the local Wasserstein distance between $\sigma$ and the Hausdorff measure of $\Gamma_\cS$ is bounded by the local Wasserstein distance between $\sigma$ and the best approximating plane, which means that $\Gamma_\cS$ approximate $\partial \Omega$ better (or at least not much worse) than the best plane around any $Q\in \cS$. Up to our knowledge, it is the first time that such a property on $\Gamma_\cS$ is established. It morally means that $D_\beta(X)$ will be well approximated by 
\[D_{\beta,\cS}(X):= \left( \int_{\Gamma_\cS} |X-y|^{-d-\beta} d\mathcal H^{n-1}(y)  \right)^{-\frac1\beta}\]
whenver $X\in X_\Omega(\cS)$, and that the error can be bounded in terms of the Tolsa's $\alpha$-numbers. 

Nevertheless, what we truly want is the fact $\partial \Omega$ is well approximated by a plane - instead of a Lipschitz graph - from the standpoint of any $X \in W_\Omega(\cS)$, because in this case we can use Lemma \ref{lemflat}. 
Yet, despite this slight disagreement, $\Gamma_{\cS}$ is much better than a random uniformly rectifiable set, because $\Gamma_{\cS}$ is the image of a plane $P$ by a bi-Lipschitz map.  
So we construct a bi-Lipschitz map $\rho_{\cS} : \, \R^n \to \R^n$ that of course maps a plane to $\Gamma_{\cS}$, but which also provides an explicit map from any point $X$ in $W_\Omega(\cS)$ to a plane $\Lambda(\rho_{\cS}^{-1}(X))$ that well approximates $\Gamma_{\cS}$ - hence $\partial \Omega$ - from the viewpoint of $X$. So morally, we constructed $\rho_\cS$ such that we have a function 
\[X \mapsto \dist(X,\Lambda(\rho^{-1}_{\cS}(X)))\]
which, when $X\in W_\Omega(\cS)$, is a good approximation of 
\[D_{\beta,\cS}(X):= \left( \int_{\Gamma_\cS} |X-y|^{-d-\beta} d\mathcal H^{n-1}(y)  \right)^{-\frac1\beta}\]
in terms of the Tolsa's $\alpha$-numbers.

We combine the two approximations to prove (see Lemma \ref{LprDb}, which is a consequence of Corollary \ref{CestD} and our construction of $\rho_{\cS}$) that
\[\iint_{W_\Omega(Q)} \left| \frac{\nabla D_\beta(X)}{D_\beta(X)} - \frac{N_{\rho^{-1}_{\cS}(X)}(X)}{\dist(X,\Lambda(\rho^{-1}_{\cS}(X))} \right|^2\,  \delta(X) \, dX \leq C |\alpha_{\sigma,\beta}(Q)|^2 \sigma(Q) \qquad \text{ for } Q\in \cS,\]
where $Y \to N_{\rho_{\cS}^{-1}(X)}(Y)$ is the gradient of the distance to $\Lambda(\rho_{\cS}^{-1}(X))$.
And since the $\alpha_{\sigma,\beta}(Q)$ satisfies the Carleson packing condition, see Lemma \ref{LalphabetaCM}, we deduce that
\begin{equation} \label{Main1i}
I_{\cS} \leq 2I'_{\cS} + C \sum_{Q\in \cS} |\alpha_{\sigma,\beta}(Q)|^2  \sigma(Q) \leq 2 I'_{\cS} + C \sigma(Q(\cS))
\end{equation}
where 
\[ I'_{\cS}:= \iint_{W_\Omega(\cS)} \left| \frac{\nabla u}{u} - \frac{N_{\rho^{-1}_{\cS}(X)}(X)}{\dist(X,\Lambda(\rho^{-1}_{\cS}(X))} \right|^2 \delta \, dX.\]

We are left with $I'_{\cS}$. We make the change of variable $(p,t) = \rho^{-1}_{\cS}(X)$ in the integral defining $I'_{\cS}$ and we obtain that
\[\begin{split}
I'_{\cS} & = \iint_{\rho_{\cS}^{-1}(W_\Omega(\cS))} \left| \frac{(\nabla u) \circ \rho_{\cS} (p,t)}{u \circ \rho (p,t) } - \frac{N_{p,t}(\rho_{\cS}(p,t))}{\dist(\rho_{\cS}(p,t),\Lambda(p,t)} \right|^2 \delta \circ \rho_{\cS}(p,t) \det\br{\Jac(p,t)}\, dt\, dp \\
& \le 2 \iint_{\rho_{\cS}^{-1}(W_\Omega(\cS))} \left|\frac{\nabla \br{u \circ \rho_{\cS} (p,t)}}{u \circ \rho (p,t)} - \frac{\Jac(p,t) N_{p,t}(\rho_{\cS}(p,t))}{\dist(\rho_{\cS}(p,t),\Lambda(p,t))} \right|^2 \, |t| \, dt\, dp ,
\end{split}\]
where $\Jac(p,t)$ is the Jacobian matrix of $\rho_{\cS}$, which is close to the identity by Lemma \ref{LestonJ}. We have also used that $\delta \circ \rho_{\cS}(p,t) \approx t$ since $\delta(X) \approx \dist(X,\Gamma_{\cS})$ on $W_\Omega(\cS)$ and the bi-Lipschitz map $\rho_{\cS}^{-1}$ preserves this equivalence. 
Even if the term
\[\frac{\Jac(p,t) N_{p,t}(\rho_{\cS}(p,t))}{\dist(\rho_{\cS}(p,t),\Lambda(p,t))}\]
looks bad, all the quantities inside are constructed by hand, and of course, we made them so that they are close to the quotient $\frac{\nabla d_P}{d_P}$, where $d_P$ is the distance to a plane that depends only on $\cS$. With our change of variable, we even made it so that $P = \R^{n-1} \times \{0\}$, that is $\frac{\nabla d_P}{d_P} = \frac{\nabla t}{t}$. Long story short, Lemma \ref{LprNpt} gives that
\begin{equation} \label{Main1j}
I'_{\cS} \leq 4I''_{\cS} + \sigma(Q(\cS))
\end{equation}
where 
\[ I''_{\cS}:= \iint_{\rho^{-1}(W_\Omega(\cS))} \left| \frac{\nabla v}{v} - \frac{\nabla t}{t}\right|^2 |t| \, dt\, dp\,, \qquad v=u\circ\rho_\cS.\]

\medskip

\noindent {\bf Step 5: Conclusion on $I_{\cS}$ using the flat case.} 
It is easy to see from the definition that Chord-Arc Domains are preserved by bi-Lipschitz change of variable, and the new CAD constants depends only on the old ones and the Lipschitz constants of the bi-Lipschitz map. Since the bi-Lipschitz constants of $\rho_{\cS}$ is less than 2 (and so uniform in $\cS$), we deduce that $\rho^{-1}_{\cS}(\Omega)$ is a Chord-Arc Domain with CAD constants that depends only on the CAD constants of $\Omega$.

Then, in Section \ref{SWhitney}, we constructed a cut-off function $\Psi_{\cS}$ adapted to $W_\Omega(\cS)$. We have shown in Lemma \ref{LprWS} that $\Psi_{\cS}$ is 1 on $W_\Omega(\cS)$ and supported in $W^*_\Omega(\cS)$, on which we still have $\delta(X) \approx \dist(X,\Gamma_{\cS})$. In Lemma \ref{LprPsiS}, we proved that the support of $\nabla \Psi_{\cS}$ is small, in the sense that implies $\1_{\supp \nabla \Psi_{\cS}} \in CM_\Omega$. What is important is that those two properties are preserved by bi-Lipschitz change of variable, and thus $\Psi_{\cS} \circ \rho_{\cS}$ is as in Definition \ref{defcutoffboth}.

We want the support of $\Psi_{\cS} \circ \rho_{\cS}$ to be included in a ball $B_{\cS}$ such that $2B_{\cS}$ is a subset of our initial ball $B$, and such that the radius of $B_{\cS}$ is smaller than $C\ell(Q(\cS))$. But those facts are an easy consequence of \eqref{Main1d} and the definition of $W^*_\Omega(\cS)$ (and the fact that $\rho_{\cS}$ is bi-Lipschitz with the Lipschitz constant close to 1).

We also want $u\circ \rho_{\cS}$ to be a solution of $L_{\cS}(u\circ \rho_{\cS}) = 0$ for a weak-DKP operator $L_{\cS}$. The operator $L_{\cS}$ is not exactly weak-DKP everywhere in $\rho^{-1}_{\cS}(\Omega)$, but it is the case on the support of $\Psi_{\cS}$ (see Lemma \ref{LprAS}), which is one condition that we need for applying Lemma \ref{lemflat}.

To apply Lemma \ref{lemflat} - or more precisely for Lemma \ref{lemlogk}, where one term from the integration by parts argument is treated - we need to show that $\omega_{L^*} \in A_\infty(\sigma)$. This is a consequence of Theorem \ref{Main3}. Indeed, since the adjoint operator $L^*$ is also a weak DKP operator on $\om$, Theorem \ref{Main3} asserts that $\omega_{L^*}\in A_\infty(\sigma)$, where $\sigma$ is an Ahlfors regular measure on $\pom$. A direct computation shows that for any set $E\subset\pom$ and any $X_0\in\om$,
\[
\omega_{L^*_\cS}^{\rho_\cS^{-1}(X_0)}\br{\rho_{\cS}^{-1}(E)}=\omega_{L^*}^{X_0}(E),
\]
and since the mapping $\rho_\cS$ is bi-Lipschitz,  $\omega_{L^*}\in A_\infty(\sigma)$ implies that the $L^*_\cS$-elliptic measure $\omega_{L^*_\cS}\in A_\infty(\wt\sigma)$, where $\wt\sigma$ is an Ahlfors regular measure on the boundary of $\rho_{\cS}^{-1}(\om)$. Moreover, the intrinsic constants in $\omega_{L^*_\cS}\in A_\infty(\wt\sigma)$ depend only on the intrinsic constants in $\omega_{L^*}\in A_\infty(\sigma)$ because the  the bi-Lipschitz constants of $\rho_\cS$ are bounded uniformly in $\cS$.

All those verification made sure that we can apply Lemma \ref{lemflat}, which entails that
\begin{equation} \label{Main1k}
I''_{\cS} \leq  \iint_{\rho^{-1}(W_\Omega(\cS))} |t| \left| \frac{\nabla (u\circ\rho_\cS)}{u\circ\rho_\cS} - \frac{\nabla t}{t}\right|^2 (\Psi_{\cS} \circ \rho_{\cS})^2 \, dt\, dp \lesssim \ell(Q(\cS))^{n-1} \lesssim \sigma(Q(\cS)).
\end{equation}

\medskip

\noindent {\bf Step 6: Gathering of the estimates.}
We let the reader check that \eqref{Main1f}--\eqref{Main1k} implies \eqref{Main1c}, and enjoy the end of the sketch of the proof!

\subsection{Organisation of the paper}

In Section \ref{SMisc}, we present the exact statement on our assumptions on $\Omega$, and we give the elliptic theory that will be needed in Section \ref{Sflat}. 

Sections \ref{SUR} to \ref{Sflat} proved the implication $(ii) \implies (iii)$ in Theorem \ref{Main1}.
Section \ref{SUR} introduces the reader to the uniform rectifiability and present the corona decomposition that will be needed. 
The corona decomposition gives a collection of (stopping time) coherent regimes $\{\cS\}_{\mathfrak S}$. From Subsection \ref{SSLipschitz} to Section \ref{Srho}, $\cS \in \mathfrak S$ is fixed. 
We construct in Subsection \ref{SSLipschitz} a set $\Gamma_{\cS}$ which is the graph of a Lipschitz function with small Lipschitz constant. 

Section \ref{SWhitney} associate a ``Whitney'' region $W_\Omega(\cS)$ to the coherent regime $\cS$ so that from the stand point of each point of $W_\Omega(\cS)$, $\Gamma_{\cS}$ and $\partial \Omega$ are well approximated by the same planes.

In Section \ref{SDbeta}, we are applying the result from Section \ref{SWhitney} to compare $D_\beta$ with the distance to a plane that approximate $\Gamma_{\cS}$. 

Section \ref{Srho} construct a bi-Lipschitz change of variable $\rho_{\cS}$ that flattens $\Gamma_\cS$, and we use the results from Sections \ref{SWhitney} and \ref{SDbeta} in order to estimate the difference
\[\frac{\nabla[D_\beta \circ \rho_{\cS}]}{D_\beta \circ \rho_{\cS}} - \frac{\nabla t}{t}\]
in terms of Carleson measure. Sections \ref{SUR} to \ref{Srho} are our arguments for the geometric side of the problem, in particular, the solutions to $Lu=0$ are barely mentionned (just to explain the effect of $\rho_{\cS}$ on $L$). 

Section \ref{Sflat} can be read independently and will contain our argument for the PDE side of the problem. Morally speaking, it proves Theorem \ref{Main1} $(ii)\implies (iii)$ when $\Omega = \R^n_+$.

Section \ref{SecPfofThm4} presents a sketch of proof of Theorem \ref{Main4}. The strategy is similar to our proof of Theorem \ref{Main1}, and in particular, many of the constructions and notations from Section~\ref{SUR} to Section~\ref{Sflat} are adopted in Section~\ref{SecPfofThm4}. But since we do not need to deal with the regularized distance $D_\beta$, the proof is much shorter.

 Section \ref{Sconv} tackles the converse implication, proving $(v) \implies (i)$ in Theorem \ref{Main1}. The proof adapts an argument of \cite{DM2}, which states that if $G$ is sufficiently close to $D_\beta$, then $\partial \Omega$ is uniformly rectifiable. As mentioned earlier, we unfortunately did not succeed to link our strong estimate \eqref{Main2b} directly to the weak ones assumed in \cite{DM2}, which explains why we needed to rewrite the argument.

We finish with Section \ref{Scount}, where we construct a semi-uniform domain and a positive harmonic solution on it for which our estimate \eqref{Main1b} is false. 

\section{Miscellaneous} 

\label{SMisc}

\subsection{Self improvement of the Carleson condition on $\A$}

\begin{lemma} \label{LellipB=ellipA}
Let $\A$ be a uniformly elliptic matrix on a domain $\Omega$, i.e. a matrix function that satisfies \eqref{defelliptic} and \eqref{defbounded} with constant $C_\A$. 
Assume that $\A$ can be decomposed as $\A = \B + \C$ where
\begin{equation} \label{elBelA1}
|\delta \nabla \B| + |\C| \in CM_{\Omega}(M).
\end{equation}
Then there exists another decomposition $\A = \wt \B + \wt \C$ such that 
\begin{equation} \label{elBelA2}
|\delta \nabla \wt\B| + |\wt\C| \in CM_{\Omega}(CM)
\end{equation}
with a constant $C>0$ that depends only on $n$, and $\wt\B$ satisfies \eqref{defelliptic} and \eqref{defbounded} with the same constant $C_{\A}$ as $\A$. In addition, 
\begin{equation}\label{eqwtB}
    |\delta \nabla \wt \B| \leq CC_{\A}.
\end{equation}
\end{lemma}

\bp
Let $\A = \B + \C$ as in the assumption of the lemma. Let $\theta \in C^\infty_0(\R^n)$ be a nonnegative function such that $\supp \theta \subset B(0,\frac1{10})$, $\iint_{\R^n} \theta(X) dX = 1$. Construct $\theta_X(Y) := \delta(X)^{-n} \theta\big(\frac{Y-X}{\delta(X)}\big)$ and then
\begin{equation} \label{defwtB}
\wt{\B}(X) := \iint_{\R^n} \A(Y)  \, \theta_X(Y)  \, dY \quad \text{ and }  \quad \wt \C := \A - \wt \B.
\end{equation}
 We see that $\wt \B$ is an average of $\A$, so $\wt B$ verifies \eqref{defelliptic} and \eqref{defbounded} with the same constant as $\A$. 
 So it remains to prove \eqref{elBelA2} and \eqref{eqwtB}. Observe that
\begin{multline*}
\nabla_X \theta_X(Y) = - n \delta(X)^{-n-1} \nabla \delta(X) \theta\Big(\frac{Y-X}{\delta(X)}\Big) - \delta(X)^{-n-1} (\nabla \theta)\Big(\frac{Y-X}{\delta(X)}\Big) \\ - \delta(X)^{-n-2} \nabla \delta(X)  (Y-X) \cdot (\nabla \theta)\Big(\frac{Y-X}{\delta(X)}\Big).
\end{multline*}
Let $\Theta(Z)$ denote $Z\theta(Z)$, then $\div\Theta(Z)=n\theta(Z)+Z\cdot\nabla\theta(Z)$. So 
\[
 \delta(X) \nabla_X \theta_X(Y)  = - \delta(X)^{-n}  (\nabla \theta)\Big(\frac{Y-X}{\delta(X)}\Big) - \delta(X)^{-n} \nabla \delta(X) (\div\Theta)\Big(\frac{Y-X}{\delta(X)}\Big). 
\]
From here, one easily sees that $\abs{\delta(X) \nabla_X \theta_X(Y)}$ is bounded by $C\delta(X)^{-n}$ uniformly in $X$ and $Y$, and thus 
 \[|\delta(X)\nabla \wt B(X)| \lesssim \fiint_{B(X,\delta(X)/2)} |\A(Y)| dY \leq C_{\A}\,,\]
which proves \eqref{eqwtB}. Set $\Theta_X(Y) = \delta(X)^{-n} \Theta\Big(\frac{Y-X}{\delta(X)}\Big) $. Then
\[
 \delta(X) \nabla_X \theta_X(Y)  = - \delta(X) \nabla_Y \theta_X(Y) -  \delta(X) \nabla \delta(X)\diver_Y\Theta_X(Y).
\]
As a consequence, 
\[\begin{split}
\delta(X) \nabla \wt \B(X) & = \iint_{\R^n} (\B + \C)(Y)  \, \delta(X) \nabla_X \theta_X(Y)  \, dY \\
& = \delta(X)  \iint_{\R^n} \nabla \B(Y)  \, \theta_X(Y)  \, dY + \delta(X) \nabla \delta(X)  \iint_{\R^n} \nabla \B(Y)  \cdot \Theta_X(Y)  \, dY \\
& \qquad  +   \iint_{\R^n} \C(Y)  \, [\delta(X) \nabla_X \theta_X(Y)]  \, dY.
\end{split}\]
We deduce that 
\[|\delta(X) \nabla \wt \B(X)| \lesssim \fiint_{B(X,\delta(X)/10)} (\delta |\nabla \B(Y)| + |\C(Y)|) \, dY,\]
and so the fact that $|\delta\nabla \B| + |\C| \in CM_\Omega(M)$ is transmitted to $\delta \nabla \wt \B$, i.e. $\delta \nabla \wt \B \in CM_\Omega(CM)$.   

As for $\wt \C$, since $\iint \theta_X(Y) dY = 1$, we have
\begin{equation}\label{wtCsplit}
    \begin{split}
|\wt \C(X)| & = \left|\iint_{\R^n} (\A(Y) - \A(X)) \theta_X(Y)  \, dY \right| \\
& \leq \iint_{\R^n} (|\B(Y) - \B(X)| + |\C(Y)| + |\C(X)|) \theta_X(Y)  \, dY \\
& \lesssim |\C(X)| + \fiint_{B(X,\delta(X)/10)} (|\B(Y) - \B(X)| + |\C(Y)|) \, dY. 
\end{split}
\end{equation}
By Fubini's theorem, to show that $\abs{\wt \C}\in CM_{\om}(CM)$, it suffices to show that for any ball $B$ centered on the boundary,
\[
\iint_{B\cap\om}\fiint_{B(Z,\delta(Z)/4)}\abs{\wt \C(X)}^2dX\frac{dZ}{\delta(Z)}\le CM\sigma(B\cap\pom).
\]
From this one sees that the terms on the right-hand side of \eqref{wtCsplit} that involves $\C$  can be easily controlled using $\abs{\C}\in CM_\om(M)$. So by the Cauchy-Schwarz inequality, it suffices to control
\begin{equation}\label{BLHS}
    \iint_{Z\in B\cap\om}\fiint_{X\in B(Z,\delta(Z)/4)}\fiint_{Y\in B(X,\delta(X)/10)}\abs{\B(Y)-\B(X)}^2dYdX\frac{dZ}{\delta(Z)}.
\end{equation}
Notice that for all $X\in B(Z,\delta(Z)/4)$, $B(X,\delta(X)/10)\subset B(Z,\delta(Z)/2)$, and thus
\[\fiint_{Y\in B(X,\delta(X)/10)}\abs{\B(Y)-\B(X)}^2dY\lesssim \fiint_{Y\in B(Z,\delta(Z)/2)}\abs{\B(Y)-\B(X)}^2dY.\]
Therefore, 
\begin{multline*}
   \eqref{BLHS}\lesssim \iint_{Z\in B\cap\om}\fiint_{X\in B(Z,\delta(Z)/4)}\fiint_{Y\in B(Z,\delta(Z)/2)}\abs{\B(Y)-\B(X)}^2dYdX\frac{dZ}{\delta(Z)}\\
   \lesssim\iint_{Z\in B\cap\om}\fiint_{X\in B(Z,\delta(Z)/2)}\abs{\nabla\B(X)}^2dX\delta(Z)dZ\\
   \lesssim\iint_{X\in 2B\cap\om}\abs{\nabla\B(X)}^2\delta(X)dX\le CM\sigma(B\cap\pom)
\end{multline*}
by the Poincar\'e inequality, Fubini's theorem, and $\abs{\delta\nabla\B}\in CM_\om(M)$.  
So again, the Carleson bound on $|\delta\nabla \B| + |\C|$ is given to $\wt \C$ as well. The lemma follows.
\ep

\subsection{Definition of Chord-Arc Domains.}
\label{SSCAD}

\begin{definition}[{\bf Corkscrew condition}, \cite{JK}]\label{def1.cork}
We say that a domain $\Omega\subset \Rn$
satisfies the {\it corkscrew condition} with constant $c\in(0,1)$ if 
for every surface ball $\Delta:=\Delta(x,r),$ with $x\in \partial\Omega$ and
$0<r<\diam(\Omega)$, there is a ball
$B(X_\Delta,cr)\subset B(x,r)\cap\Omega$.  The point $X_\Delta\subset \Omega$ is called
a {\bf corkscrew point} relative to $\Delta$ (or, for $x$ at scale $r$).
\end{definition}

\begin{definition}[{\bf Harnack Chain condition}, \cite{JK}]\label{def1.hc}
We say that $\Omega$ satisfies the {\it Harnack Chain condition} with constants $M$, $C>1$ if for every $\rho >0,\, \Lambda\geq 1$, and every pair of points
$X,X' \in \Omega$ with $\delta(X),\,\delta(X') \geq\rho$ and $|X-X'|<\Lambda\,\rho$, there is a chain of
open balls
$B_1,\dots,B_N \subset \Omega$, $N\leq M(1+\log\Lambda)$,
with $X\in B_1,\, X'\in B_N,$ $B_k\cap B_{k+1}\neq \emptyset$
and $C^{-1}\diam (B_k) \leq \dist (B_k,\partial\Omega)\leq C\diam (B_k).$  The chain of balls is called
a {\it Harnack Chain}.
\end{definition}

\begin{definition}[\bf 1-sided NTA and NTA]\label{def1.1nta}
We say that a domain $\Omega$ is a {\it 1-sided NTA  domain} with constants $c,C,M$ if it satisfies the corkscrew condition with constant $c$ and Harnack Chain condition with constant $M, C$.
Furthermore, we say that $\Omega$ is an {\it NTA  domain} if it is a 1-sided NTA domain and if, in addition, $\Omega_{\rm ext}:= \Rn\setminus \overline{\Omega}$ also satisfies the corkscrew condition.
\end{definition}

\begin{definition}[\bf 1-sided CAD and CAD]\label{defi:CAD}
A \emph{1-sided chord-arc domain} (1-sided CAD) is a 1-sided NTA domain with AR boundary. The 1-sided NTA constants and the AR constant are called the 1-sided CAD constants. 
A \emph{chord-arc domain} (CAD, or 2-sided CAD) is an NTA domain with AR boundary. The 1-sided NTA constants, the corkscrew constant for $\om_{\rm ext}$, and the AR constant are called the CAD constants.
\end{definition}

Uniform rectifiability (UR) is a quantitative version of rectifiability.

\begin{definition}[\bf UR]
We say that $E$ is uniformly rectifiable if $E$ has big pieces of Lipschitz images, that is, if $E$ is $(n-1)$-Ahlfors regular \eqref{defADR}, and there exists $\theta, M>0$ such that, for each $x\in \Gamma$ and $r>0$, there is a Lipschitz mapping $\rho$ from the ball $B(0,r) \subset \R^d$ into $\R^n$ such that $\rho$ has Lipschitz norm $\leq M$ and 
\[\sigma(\Gamma \cap B(x,r) \cap \rho(B_{\R^d}(0,r))) \geq \theta r^d.\]
\end{definition}

However, we shall not use the above definition. What we do require is the characterization of UR by Tolsa's $\alpha$-numbers (\cite{Tolsa09} ), as well as a modification of the corona decomposition of uniformly rectifiable sets constructed in \cite{DS1}. See Section~\ref{SUR} for details. 
We shall also need the following result.
\begin{lemma}\label{lem.UReqv}
Suppose that $\om\subset\Rn$ is 1-sided chord-arc domain. Then the following are equivalent:
\begin{enumerate}
    \item $\pom$ is uniformly rectifiable.
    \item $\om_{\rm ext}$ satisfies the corkscrew condition, and hence, $\om$ is a chord-arc domain.
\end{enumerate}
\end{lemma}
That (1) implies (2) was proved in \cite{AHMNT}. That (2) implies (1) can be proved via the $A_\infty$ of harmonic measure (see \cite{AHMNT} Theorem 1.2), or directly as in \cite{DJ}.

\subsection{Preliminary PDE estimates}

\begin{lemma}[The Caccioppoli inequality] \label{Caccio}
Let $L=-\diver A\nabla$ be a uniformly elliptic operator and $u\in W^{1,2}(2B)$ be a solution of $Lu=0$ in $2B$, where $B$ is a ball with radius $r$. Then there exists $C$ depending only on $n$ and the ellipticity constant of $L$ such that 
\[
\fint_{B}\abs{\nabla u(X)}^2dX\le \frac{C}{r^2}\fint_{2B}\abs{u(X)}^2dX.
\]
\end{lemma}

\begin{lemma}[The Harnack inequality] \label{Harnack}
Let $L$ be as in Lemma \ref{Caccio} and let $u$ be a nonnegative solution of $Lu=0$ in $2B\subset\om$. Then there exists constant $C\ge 1$ depending only on $n$ and the ellipticity constant of $L$ such that 
\[
\sup_B u \le C\inf_B u.
\]
\end{lemma}
Write $L^*$ for the transpose of $L$ defined by 
$L^*=-\diver A^\top\nabla$, where $A^\top$ denotes the transpose matrix of $A$. Associated with $L$ and $L^*$ one can respectively construct the elliptic measures $\{\omega_L^X\}_{X\in \Omega}$ and $\{\omega_{L^*}^X\}_{X\in\Omega}$, and the Green functions $G_L$ and $G_{L^*}$ on domains with Ahlfors regular boundaries (cf. \cite{KenigB}, \cite{HMT}).
\begin{lemma}[The Green function]\label{lemagreen}
Suppose that $\Omega\subset\R^{n}$ is an open set such that $\partial\Omega$ is Ahlfors regular. Given an elliptic operator $L$, there exists a unique Green function $G_L(X,Y): \Omega \times \Omega \setminus  \diag(\Omega) \to \R$ 
with the following properties: 
$G_L(\cdot,Y)\in W^{1,2}_{\rm loc}(\Omega\setminus \{Y\})\cap C(\overline{\Omega}\setminus\{Y\})$, 
$G_L(\cdot,Y)\big|_{\pom}\equiv 0$ for any $Y\in\Omega$, 
and $L G_L(\cdot,Y)=\delta_Y$ in the weak sense in $\Omega$, that is,
\begin{equation*}\label{Greendef}
    \iint_\Omega A(X)\,\nabla_X G_{L}(X,Y) \cdot\nabla\vp(X)\, dX=\vp(Y), \qquad\text{for any }\vp \in C_c^\infty(\Omega).
\end{equation*}
In particular, $G_L(\cdot,Y)$ is a weak solution to $L G_L(\cdot,Y)=0$ in $\Omega\setminus\{Y\}$. Moreover,
\begin{equation}\label{eq2.green}
G_L(X,Y) \leq C\,|X-Y|^{2-n}\quad\text{for }X,Y\in\Omega\,,
\end{equation}
\begin{equation*}\label{eq2.green2}
c_\theta\,|X-Y|^{2-n}\leq G_L(X,Y)\,,\quad \text{if } \,\,\,|X-Y|\leq \theta\, \dist(X,\pom)\,, \,\, \theta \in (0,1)\,,
\end{equation*}
\begin{equation*}
\label{eq2.green3}
G_L(X,Y)\geq 0\,,\quad G_L(X,Y)=G_{L^*}(Y,X), \qquad \text{for all } X,Y\in\Omega\,,\, 
X\neq Y.
\end{equation*}
\end{lemma}

The following lemma will be referred to as the CFMS estimates (cf. \cite{CFMS}, \cite{KenigB} for NTA domains, and \cite{HMT2} or \cite{DFMprelim2} for 1-sided CAD).

\begin{lemma}[The CFMS estimates] \label{LCFMS}
Let $\Omega$ be a 1-sided CAD domain. Let $L$ be an elliptic operator satisfying \eqref{defelliptic} and \eqref{defbounded}. There exist $C$ depending only on $n$, $C_\A$, and the $1$-sided CAD constants, such that for any $B:=B(x,r)$, with $x\in \pom$, $0<r<\diam(\pom)$ and $\Delta:=\Delta(x,r)$, we have the following properties. 
\begin{enumerate}
\item The elliptic measure is non-degenerate, that is
\[C^{-1} \leq \omega_L^{X_\Delta}(\Delta) \leq C.\]
\item For
$X\in \Omega\setminus 2\,B$ we have
\begin{equation}\label{eq.CFMS}
\frac1C\omega_L^X(\Delta)
\leq
r^{n-1}G_L(X,X_\Delta) \leq C \omega_L^X(\Delta).
\end{equation} 

\item If $0\leq u, v\in W^{1,2}_{\rm loc}(4\,B\cap\Omega)\cap C(\overline{4\,B\cap\Omega})$ are two nontrivial weak solutions of $Lu=Lv=0$ in $4\,B\cap\Omega$ such that $u=v=0$ in $4\,\Delta$, then
$$
C^{-1}\frac{u(X_\Delta)}{v(X_\Delta)}\le \frac{u(X)}{v(X)}\le C\frac{u(X_\Delta)}{v(X_\Delta)},\qquad\text{for all }X\in B\cap\Omega.
$$
\end{enumerate}

\end{lemma}

\section{Characterization of the uniform rectifiability}

\label{SUR}

In all this section, we assume that $\partial \Omega$ is uniformly rectifiable, and we plan to prove a corona decomposition of the uniformly rectifiable set which is ``Tolsa's $\alpha$-number compatible''. 

Instead of a long explanation of the section, which will not be helpful anyway to any reader who is not already fully familiar with the corona decomposition (C3) in \cite{DS1} and the Tolsa $\alpha$-number (see \cite{Tolsa09}), we shall only state below the results proved in the section (the definition of all the notions and notation will be ultimately given in the section below).

\begin{lemma} \label{LcoronaG}
Let $\partial \Omega$ be a uniformly rectifiable set. Given any positive constants $0 <\epsilon_1 < \epsilon_0 < 1$, there exists a disjoint decomposition $\D_{\pom} = \G \cup \B$ such that
\begin{enumerate}[(i)]
\item The ``good'' cubes $Q\in \G$ are such that $\alpha_\sigma(Q) \leq \epsilon_1$ and
\begin{equation}
\sup_{y \in 999\Delta_Q} \dist(y,P_Q) + \sup_{p\in P_Q \cap 999B_Q} \dist(p,\partial \Omega) \leq \epsilon_1 \ell(Q).
\end{equation}
\item The collection $\G$ of ``good'' cubes can be further subdivided into a disjoint family $\G = \ds \bigcup_{\cS \in \mathfrak S} \cS$ of {\bf coherent} regimes such that for any $\cS\in \mathfrak S$, there a hyperplane $P:=P_{\cS}$ and a $2\epsilon_0$-Lipschitz function $\mathfrak b_\cS:= \bb$ on $P$ such that  
\begin{equation} \label{newstatement}
    \int_{P\cap \Pi(2B_Q)} \dist(\mathfrak b(p),P_Q) \, dp \leq C \ell(Q) \sigma(Q)  \alpha_\sigma(Q) \qquad \text{ for } Q\in \cS,
\end{equation}
where $C$ depends only on $n$.
\item The cubes in $\B$ (the ``bad'' cubes) and the maximal cubes $Q(\cS)$ satisfies the Carleson packing condition
 \begin{equation}
\sum_{\begin{subarray}{c} Q\in \B \\ Q \subset Q_0 \end{subarray}} \sigma(Q) + \sum_{\begin{subarray}{c} \cS\in \mathfrak S \\ Q(\cS) \subset Q_0 \end{subarray}} \sigma(Q(\cS)) \leq C_{\epsilon_0,\epsilon_1} \sigma(Q_0) \qquad \text{ for all } Q_0\in \D_{\partial \Omega}.
\end{equation}
\end{enumerate}
\end{lemma}

In the above lemma, $\sigma$ is the Ahlfors regular measure for $\partial \Omega$ given in \eqref{defADR}, $\mathbb D_{\partial \Omega}$ is a dyadic decomposition of $\partial \Omega$, $\Pi:=\Pi_{\cS}$ is the orthogonal projection on $P_\cS$, $P_Q$ is the best approximating plane of $\partial \Omega$ around $Q$, and $\alpha_\sigma$ is the Tolsa $\alpha$-number for $\sigma$. The novelty, which is not similar to any of the corona decompositions that the authors are aware of, is \eqref{newstatement}, which quantify the difference between $\partial \Omega$ and the approximating graph $\Gamma_\cS$ in terms of $\alpha$-number.

Corona decompositions are a useful and popular tool in the recent literature pertaining to uniformly rectifiable sets, see for instance \cite{DS1}, \cite{HMMDuke}, \cite{GMT}, \cite{AGMT}, \cite{BH}, \cite{AHMMT}, \cite{BHHLN}, \cite{MT}, \cite{CHM}, \cite{MPT} to cite only a few.

\subsection{Dyadic decomposition} \label{SSdyadic} 
We construct a dyadic system of pseudo-cubes on $\partial \Omega$. In the presence of the Ahlfors regularity property,  such construction appeared for instance  in \cite{David88}, \cite{David91}, \cite{DS1} or \cite{DS2}. We shall use the very nice construction of Christ \cite{Ch}, that allow to bypass the need of a measure on $\partial \Omega$. More exactly, one can check that the construction of the dyadic sets by Christ to not require a measure, and as such are independent on the measure on $\partial \Omega$.

There exist a universal constant $0<a_0<1$ and a collection $\mathbb{D}_{\partial \Omega} = \cup_{k \in \mathbb Z} \mathbb{D}_{k}$ of Borel subsets of $\partial \Omega$, with the following properties. We write 
\[\mathbb{D}_{k}:=\{Q_{j}^k\subset \D_{\partial \Omega}: j\in \mathfrak{I}_k\}, \]
where $\mathfrak{I}_k$ denotes some index set depending on $k$, but sometimes, to lighten the notation, we shall forget about the indices and just write $Q \in \mathbb{D}_k$ and refer to $Q$ as a cube (or pseudo-cube) of generation $k$. Such cubes enjoy the following properties:

\begin{enumerate}[(i)]

\item $\partial \Omega =\cup_{j} Q_{j}^k \,\,$ for any $k \in \mathbb Z$.

\item If $m > k$ then either $Q_{i}^{m}\subseteq Q_{j}^{k}$ or
$Q_{i}^{m}\cap  Q_{j}^{k}=\emptyset$.

\item $Q_i^m \cap Q_j^m=\emptyset$ if $i\neq j$. 

\item Each pseudo-cube $Q\in\mathbb{D}_k$ has a ``center'' $x_Q\in \D$ such that
\begin{equation}\label{cube-ball}
\Delta(x_Q,a_02^{-k})\subset Q \subset \Delta(x_Q,2^{-k}).
\end{equation}
\end{enumerate}

Let us make a few comments about these cubes. We decided to use a dyadic scaling (by opposition to a scaling where the ratio of the sizes between a pseudo-cube and its parent is, in average, $\epsilon < \frac12$) because it is convenient.
The price to pay for forcing a dyadic scaling is that if $Q \in \mathbb{D}_{k+\ell}$ and $R$ is the cube of $\mathbb{D}_k$ that contains $Q$ (it is unique by ($ii$), and it is called an ancestor of $Q$) is {\em not} necessarily strictly larger (as a set) than $Q$.
We also considered that the $\partial \Omega$ was unbounded, to avoid separating cases. If the boundary $\partial \Omega$ is bounded, then $\mathbb D_{\partial \Omega} := \bigcup_{k\leq k_0} \mathbb D_k$ where $k_0$ is such that $2^{k_0-1} \leq \diam(\Omega) \leq 2^{k_0-1}$, and we let the reader check that this variation doesn't change a single argument in the sequel.

If $\mu$ is any doubling measure on $\partial \Omega$ - that is if $\mu(2\Delta) \leq C_\mu \mu(\Delta)$ for any boundary ball $\Delta \subset \partial \Omega$ - then we have the following extra property:
\begin{enumerate}[(i)] \addtocounter{enumi}{4}
\item $\mu(\partial Q_i^k) = 0$ for all $i,k$.
\end{enumerate}
In our construction, (i) and (iii) forces the $Q_i^k$ to be neither open nor closed. But this last property (v) means that taking the interior or the closure of $Q_i^k$ instead of $Q_i^k$ would not matter, since the boundary amounts to nothing. 

Let us introduce some extra notation. When $E \subset \partial \Omega$ is a set, $\D_{\partial \Omega}(E)$ is the sub-collection of dyadic cubes that are contained in $E$. When $Q\in \mathbb D_{\partial \Omega}$, we write $k(Q)$ for the generation of $Q$ and $\ell(Q)$ for $2^{-k(Q)}$, which is roughly the diameter of $Q$ by \eqref{cube-ball}. We also use $B_Q \subset \R^n$ for the ball $B(x_Q,\ell(Q))$ and $\Delta_Q$ for the boundary ball $\Delta(x_Q,\ell(Q))$ that appears in \eqref{cube-ball}. For $\kappa \geq 1$, the dilatation $\kappa Q$ is
\begin{equation} \label{defkappaQ}
\kappa Q = \{x\in \partial \Omega, \, \dist(x,Q) \leq (\kappa - 1) \ell(Q)\},
\end{equation}
which means that $\kappa Q \subset \kappa \Delta_Q \subset (\kappa+1) Q$.

\medskip

The dyadic decomposition of $\partial \Omega$ will be the one which is the most used. However, we also use dyadic cubes for other sets, for instance to construct Whitney regions, and we use the same construction and notation as the one for $\partial \Omega$. In particular, we will use dyadic cubes in $\R^n$ and in a hyperplane $P$ that still satisfy \eqref{cube-ball} for the universal constant $a_0$ - i.e. the dyadic cubes are not real cubes - and the definition \eqref{defkappaQ} holds even in those contexts.

\subsection{Tolsa's $\alpha$ numbers}

Tolsa's $\alpha$ numbers estimate how far a measure is from a flat measure, using Wasserstein distances. We denote by $\Xi$ the set of affine $n-1$ planes in $\R^n$, and for each plane $P\in \Xi$, we write $\mu_P$ for the restriction to $P$ of the $(n-1)$-dimensional Hausdorff measure, that is $\mu_P$ is the Lebesgue measure on $P$. A flat measure is a measure $\mu$ that can be written $\mu = c\mu_P$ where $c>0$ and $P \in \Xi$, the set of flat measure is then denoted by $\mathcal F$.

\begin{definition}[local Wasserstein distance]
If $\mu$ and $\sigma$ are two $(n-1)$-Ahlfors regular measures on $\R^n$, and if $y \in \R^n$ and $s>0$, we define
\[\dist_{y,s}(\mu,\sigma) := s^{-n} \sup_{f\in Lip(y,s)} \left| \int f\, d\mu - \int f \, d\sigma \right|\]
where $Lip(y,s)$ is the set of $1$-Lipschitz functions that are supported in $B(y,s)$.

If $Q \in \D_{\partial \Omega}$, then we set $\dist_{Q}(\mu,\sigma) := \dist_{x_Q,10^3\ell(Q)}(\mu,\sigma)$ and $Lip(Q) := Lip(x_Q,10^3\ell(Q))$, where $x_Q$ is as in \eqref{cube-ball}. Moreover, if $\sigma$ is an Ahlfors regular measure on $\partial \Omega$, then we set
\[\alpha_\sigma(Q) := \inf_{\mu \in \mathcal F} \dist_Q(\mu,\sigma).\]
\end{definition}
Note that 
\begin{equation}\label{eqalphaQbdd}
    0\le\alpha_\sigma(Q)\le C \qquad\text{for all }Q\in\D_{\dr\Omega}
\end{equation} 
where $C<\infty$ depends only on the Ahlfors constants of $\mu$ and $\sigma$. 

The uniform rectifiability of $\partial \Omega$ is characterized by the fact that, for any $(n-1)$-Ahlfors regular measure $\sigma$ supported on $\partial \Omega$, and any $Q_0\in\D_{\pom}$, we have  
\begin{equation}\label{defUR}
\sum_{Q\in \D_{\partial \Omega}(Q_0)} \alpha_\sigma(Q)^2 \sigma(Q) \leq C \sigma(Q_0) \approx \ell(Q_0)^{n-1}
\end{equation}
and, for any $\epsilon>0$,
\begin{equation}\label{geomlem}
\sum_{\begin{subarray}{c} Q\in \D_{\partial \Omega}(Q_0)\\ \alpha_\sigma(Q) > \epsilon  \end{subarray}} \sigma(Q) \leq C_\epsilon \sigma(Q_0) \approx \ell(Q_0)^{n-1}.
\end{equation}
For a proof of these results, see Theorem 1.2 in \cite{Tolsa09}. 

It will be convenient to introduce the following notation. Given $Q\in \D_{\partial \Omega}$, the quantities $c_Q$, $P_Q$, and $\mu_Q$ are such that 
\begin{equation}\label{defmuQ}
\mu_Q = c_Q \mu_{P_Q} \quad \text{ and } \quad \dist_Q(\sigma,\mu_Q) \leq 2 \alpha_\sigma(Q),
\end{equation}
that is $\mu_Q$ is a flat measure which well approximates $\sigma$ (as long as $\alpha_\sigma(Q)$ is small). So it means that 
\begin{equation} \label{ftestinalpha}
\left| \int f\, d\sigma - \int f \, d\mu_Q \right| \leq 2(10^3\ell(Q))^{n} \alpha_\sigma(Q) \qquad \text{ for } f\in Lip(Q). 
\end{equation}

Let us finish the subsection with the following simple result.

\begin{lemma} \label{Lalphatobeta}
There exists $C>1$ depending only on $C_\sigma$ and $n$ such that if $Q \in \D_{\partial \Omega}$ and $\epsilon\in (0,C^{-n})$ verify $\alpha_\sigma(Q) \leq \epsilon$, then 
\begin{equation} \label{Lab1}
\sup_{y \in 999\Delta_Q} \dist(y,P_Q) + \sup_{p\in P_Q \cap 999B_Q} \dist(p,\partial \Omega) \leq C \epsilon^{1/n} \ell(Q).
\end{equation}
\end{lemma}

\bp Assume that $\alpha_\sigma(Q) \leq \epsilon =  8000^{-n} C_\sigma^{-1} \eta ^{n}$ with $\eta \in (0,1)$. 
For a given point $y\in 999\Delta_Q$, we set the function $f_1(z):= \max\{0,\eta \ell(Q) - |y-z|\} \in Lip(Q)$. Observe that 
\[\int f_1 d\sigma \geq \frac{\eta\ell(Q)}{2} \sigma\br{\frac{\eta}2 \Delta_Q} \geq C_\sigma^{-1} \Big(\frac{\eta \ell(Q)}2\Big)^n.\]
and thanks to \eqref{ftestinalpha}
\[\begin{split} 8000^{-n} C_\sigma^{-1} \eta^{n} & > \alpha_\sigma(Q) \geq \frac{(1000\ell(Q))^{-n}}2 \dist_Q(\sigma,\mu_Q) \\ 
& \geq (2000\ell(Q))^{-n} \left| \int f_1 \, d\sigma - \int f_1 \, d\mu_Q \right|.
\end{split}\]
By combining the two inequalities above, we have
\[\left| \int f_1 \, d\sigma - \int f_1 \, d\mu_Q \right| \leq C_\sigma^{-1} \Big(\frac{\eta \ell(Q)}4\Big)^n \leq \frac12 \int f_1 d\sigma.\]
So necessarily, the support of $f_1$ intersects the support of $\mu_Q$, that is $ \dist(y,P_Q) \leq \eta \ell(Q)$ and the first part of \eqref{Lab1} is proved. But notice also that the same computations force the constant $c_Q$ in the flat measure $\mu_Q = c_Q \mu_{P_Q}$ to be larger than $(2c_{n-1}C_\sigma)^{-1}$, where $c_n$ is the volume of the $n$-dimensional unit ball. We take now a point $p\in P_Q \cap 999B_Q$ and construct $f_2:= \max\{0,\eta \ell(Q) - |p-z|\} \in Lip(Q)$. We have
\[\begin{split}
\left| \int f_2 \, d\sigma - \int f_2 \, d\mu_Q \right| & \leq 2(1000\ell(Q))^n \alpha_\sigma(Q) < C_\sigma^{-1} \Big(\frac{\eta \ell(Q)}4\Big)^n < \int f_2 \, d\mu_Q.
\end{split}\]
So necessarily, the the support of $f_1$ intersects the support of $\sigma$, that is $ \dist(p,\partial \Omega) \leq \eta \ell(Q)$. The lemma follows. \ep

\subsection{Corona decomposition}

We first introduce the notion of coherent subset of $\D_{\partial \Omega}$.

\begin{definition}
Let $\mathcal S \subset \D_{\partial \Omega}$. We say that $\mathcal S$ is {\bf coherent} if
\begin{enumerate}[(a)]
\item $\cS$ contains a unique maximal element $Q(\mathcal S)$, that is $Q(\mathcal S)$ contains all the other element of $\cS$ as subsets.
\item If $Q\in \cS$ and $Q \subset R \subset Q(\cS)$, then $R\in \cS$.
\item Given a cube $Q \in \cS$, either all its children belong to $\cS$ or none of them do.
\end{enumerate}
\end{definition}

The aim of the section is to prove the following corona decomposition for a uniformly rectifiable boundary $\partial \Omega$.

\begin{lemma} \label{Lcorona}
Let $\partial \Omega$ be a uniformly rectifiable set. Given any positive constants $\epsilon_1 < \epsilon_0 \in (0,1)$, there exists a disjoint decomposition $\D_{\pom} = \G \cup \B$ such that
\begin{enumerate}[(i)]
\item The ``good'' cubes $Q\in \G$ are such that $\alpha_\sigma(Q) \leq \epsilon_1$ and
\begin{equation} \label{defP'Qb}
\sup_{y \in 999\Delta_Q} \dist(y,P_Q) + \sup_{p\in P_Q \cap 999B_Q} \dist(p,\partial \Omega) \leq \epsilon_1 \ell(Q).
\end{equation}
\item The collection $\G$ of ``good'' cubes can be further subdivided into a disjoint family $\G = \ds \bigcup_{\cS \in \mathfrak S} \cS$ of {\bf coherent} regimes that satisfy
\begin{equation} \label{angleS}
\Angle(P_Q,P_{Q'}) \leq \epsilon_0 \qquad \text{ for all } \cS \in \mathfrak S \text{ and } Q,Q'\in \cS.
\end{equation}
\item The cubes in $\B$ (the ``bad'' cubes) and the maximal cubes $Q(\cS)$ satisfies the Carleson packing condition
 \begin{equation} \label{packingBS}
\sum_{\begin{subarray}{c} Q\in \B \\ Q \subset Q_0 \end{subarray}} \sigma(Q) + \sum_{\begin{subarray}{c} \cS\in \mathfrak S \\ Q(\cS) \subset Q_0 \end{subarray}} \sigma(Q(\cS)) \leq C_{\epsilon_0,\epsilon_1} \sigma(Q_0) \qquad \text{ for all } Q_0\in \D_{\partial \Omega}.
\end{equation}
\end{enumerate}
\end{lemma}

\begin{remark}
What we secretly expect is, in addition to \eqref{angleS}, to also have a control on the constants $c_Q$  - defined in \eqref{defmuQ} - that belongs to the same $\cS$. For instance, we would like to have
\[|c_Q-c_{Q(\cS)}| \leq \epsilon_0.\]
Imposing this extra condition while keeping the number of $\cS$ low should be doable, but we do not need it, so we avoided this complication.
\end{remark}

The difficult part in the above lemma is to prove that \eqref{angleS} holds while keeping the number of coherent regimes $\cS$ small enough so that \eqref{packingBS} stays true. To avoid a long and painful proof, we shall prove Lemma \ref{Lcorona} with the following result as a startpoint.

\begin{lemma}[\cite{DS1}] \label{LcoronaDS}
Let $\partial \Omega$ be a uniformly rectifiable set. Given any positive constants $\epsilon_3 < \epsilon_2 \in (0,1)$, there exists a disjoint decomposition $\D_{\pom} = \G' \cup \B'$ such that
\begin{enumerate}[(i)]
\item The ``good'' cubes $Q\in \G'$ are such that there exists an affine plane $P'_Q \in \Xi$ such that 
\begin{equation} \label{defP'Q}
\dist(x,P'_Q) \leq \epsilon_3 \ell(Q) \qquad \text{ for } x\in 999\Delta_Q
\end{equation}
\item The collection $\G'$ of ``good'' cubes can further subdivided into a disjoint family $\G' = \ds \bigcup_{\cS' \in \mathfrak S'} \cS'$ of {\bf coherent} stopping time regimes that satisfies
\begin{equation} \label{angleSDS}
\Angle(P'_Q,P'_{Q(\cS')}) \leq \epsilon_2 \qquad \text{ for all } \cS'\in \mathfrak S' \text{ and } Q\in \cS'.
\end{equation}.
\item The cubes in $\B'$ and the maximal cubes $Q(\cS')$ satisfies the Carleson packing condition
 \begin{equation} \label{packingBSDS}
\sum_{\begin{subarray}{c} Q\in \B' \\ Q \subset Q_0 \end{subarray}} \sigma(Q) + \sum_{\begin{subarray}{c}\cS' \in \mathfrak S' \\ Q(\cS') \subset Q_0 \end{subarray}} \sigma(Q(\cS')) \leq C_{\epsilon_2,\epsilon_3} \sigma(Q_0) \qquad \text{ for all } Q_0\in \D_{\partial \Omega}.
\end{equation}.
\end{enumerate}
\end{lemma}

The proof of Lemma \ref{LcoronaDS} is contained in Sections 6 to 11 of \cite{DS1}, and the statement that we gave is the combination of Lemma 7.1 and Lemma 7.4 in \cite{DS1}. Lemma \ref{Lcorona} might already be stated and proved in another article, and we apologize if it were the case. Moreover, the proof of Lemma \ref{Lcorona} is probably obvious to anyone that is a bit familiar with this tool. However, every corona decomposition has its own small differences, and we decided to write our own using only the results of David and Semmes as a prerequisite.

\medskip

\noindent {\em Proof of Lemma \ref{Lcorona} from Lemma \ref{LcoronaDS}.} 
We pick then $\epsilon_1$ and $\epsilon_0$ small such that $\epsilon_1\ll \epsilon_0 \ll 1$. We apply Lemma \ref{LcoronaDS} with the choices of $\epsilon_2:= \epsilon_0/2$ and $\epsilon_3= \epsilon_1$. Note that we can choose
\begin{equation} \label{PQP'Q}
P'_Q = P_Q \qquad \text{ when } Q \in \G' \text{ and } \alpha_\sigma(Q) \leq C^{-n} \epsilon_1^{n}
\end{equation}
if $C>0$ is the constant from Lemma \ref{Lalphatobeta}.

\medskip

Since we applied Lemma \ref{LcoronaDS}, we have a first disjoint decomposition $\D_{\partial \Omega} = \G' \cup \B'$ and a second decomposition $\G' = \bigcup \cS'$ into coherent regimes which satisfy \eqref{defP'Q}, \eqref{angleSDS}, and \eqref{packingBSDS}. 

We define $\G$ as
\[\G:= \G' \cap \{Q\in \D, \, \alpha_\sigma(Q) \leq C^{-n}\epsilon_1^{n}\}\]
where $C$ is the constant in  Lemma \ref{Lalphatobeta}. Of course, it means that $\B:= \B' \cup (\G'\sm \G)$. The coherent regimes $\cS'$ may not be contained in $\G$, that is $\cS' \cap \G$ may not be a coherent regime anymore. 
So we split further $\cS' \cap \G$ into a disjoint union of (stopping time) coherent regimes $\{\cS_{i}\}_{i\in I_{\cS'}}$ that are maximal in the sense that the minimal cubes of $\cS_{i}$ are those for which at least one children belongs to $\D \setminus (\cS' \cap \G)$.
The collection $\{\cS\}_{\cS \in \mathfrak S}$ is then the collection of all the $\cS_{i}$ for $i\in I_{\cS'}$ and $\cS' \in \mathfrak S'$.

It remains to check that the $\mathcal G$, $\B$ and $\{\cS\}_{\cS \in \mathfrak S}$ that we just  built satisfy \eqref{angleS} and \eqref{packingBS}. For the former, we use the fact that a regime $\cS$ is necessarily included in a $\cS'$, so for any $Q\in \cS$, we have
\begin{multline} \Angle(P_Q,P_{Q(\cS)}) \leq \Angle(P_Q,P_{Q(\cS')}) + \Angle(P_{Q(\cS')},P_{Q(\cS')}) \\
 = \Angle(P'_Q,P'_{Q(\cS')}) + \Angle(P'_{Q(\cS)},P'_{Q(\cS')}) \leq 2 \epsilon_2 = \epsilon_0\end{multline}
by \eqref{PQP'Q}, \eqref{angleSDS}, and our choice of $\epsilon_2$. The fact that $\B$ satisfies the Carleson packing condition
\begin{equation} \label{packingB} \sum_{\begin{subarray}{c} Q\in \B \\ Q \subset Q_0 \end{subarray}} \sigma(Q) \leq C_{\epsilon_0,\epsilon_1} \sigma(Q_0) \qquad \text{ for all } Q_0\in \D_{\partial \Omega}
\end{equation}
 is an immediate consequence of the definition of $\B$, \eqref{geomlem}, and \eqref{packingBSDS}. Finally, by the maximality of the coherent regimes $\cS$, then either $Q(\cS)$ is the maximal cube of a coherent regime from the collection $\{\cS'\}_{\cS'\in \mathfrak S'}$, or (at least) the parent or one sibling of $Q(\cS_i)$ belongs to $\B$. Therefore, if $Q^*$ denotes the parent of a dyadic cube $Q$, then for any $Q_0\in \D_{\partial \Omega}$,
\[ \sum_{\begin{subarray}{c} \cS\in \mathfrak S \\ Q(\cS) \subset Q_0 \end{subarray}} \sigma(Q(\cS)) \leq \sum_{\begin{subarray}{c} \cS'\in \mathfrak S' \\ Q(\cS') \subset Q_0 \end{subarray}} \sigma(Q(\cS'))  + \sum_{\begin{subarray}{c} Q\in \B \\ Q \subset Q_0 \end{subarray}} \sigma(Q^*) \lesssim \sigma(Q_0)\]
because of the Carleson packing conditions \eqref{packingBSDS} and \eqref{packingB}, and because $\sigma(Q^*) \approx \ell(Q)^{n-1} \approx \sigma(Q)$. The lemma follows. 
\ep

\subsection{The approximating Lipschitz graph} \label{SSLipschitz}

In this subsection, we show that each coherent regime given by the corona decomposition is well approximated by a Lipschitz graph. We follow the outline of Section 8 in \cite{DS1} except that we are a bit more careful about our construction in order to obtain Lemma \ref{Lbbalpha} below. That is, instead of just wanting the Lipschitz graph $\Gamma_\cS$ to be close to $\partial \Omega$, we aim to prove that the Lipschitz graph is an approximation of $\partial \Omega$ at least as good as the best plane.

Pick $0 < \epsilon_1 \ll \epsilon_0 \ll 1$, and then construct the collection of coherent regimes $\mathfrak S$ given by Lemma \ref{Lcorona}. Take $\cS$ to be either in $\mathfrak S$, or a coherent regime included in an element of $\mathfrak S$, and let it be fixed. Set $P := P_{Q(\cS)}$ and define $\Pi$ as the orthogonal projection on $P$. Similarly, we write $P^\bot$ for the linear plane orthogonal to $P$ and $\Pi^\bot$ for the projection onto $P^\bot$.  We shall also need the function $d$ on $P$: for $p\in P$, define
\begin{equation} \label{defdx}
d(p) := \inf_{Q\in \cS} \{ \dist(p,\Pi(2B_Q)) + \ell(Q)\}.
\end{equation}
We want to construct a Lipschitz function $b: \, P \mapsto P^\bot$. First, we prove a small result. We claim that for $x,y\in \partial \Omega \cap 999B_{Q(\cS)}$, we have
\begin{equation} \label{claimPibot}
|\Pi^\bot(x) - \Pi^\bot(y)| \leq 2\epsilon_0 |\Pi(y) - \Pi(x)| \qquad \text{ whenever } |x-y| > 10^{-3} d(\Pi(x))
\end{equation}
Indeed, with such choices of $x$ and $y$, we can find $Q\in \cS$ such that 
\[0 < |x-y| \approx \dist(\Pi(x),\Pi(Q)) + \ell(Q)\]
and by taking an appropriate ancestor of $Q$, we find $Q^*$ such that $|x-y| \approx \ell(Q^*)$. Since $x,y\in 999B_{Q(\cS)}$, we can always take $Q^* \subset Q(\cS)$ - that is $Q^* \in \cS$ thanks to the coherence of $\cS$ - and $x,y\in 999B_{Q^*}$. Due to \eqref{defP'Qb}, we deduce that 
\[\dist(x,P_{Q^*}) + \dist(y,P_{Q^*}) \leq \epsilon_1 \ell(Q^*) \ll \epsilon_0 |x-y|\]
if $\epsilon_1/\epsilon_0$ is sufficiently small. Since $\Angle(P_{Q^*},P) \leq \epsilon_0$ by \eqref{angleS}, we conclude 
\[|\Pi^\bot(x) - \Pi^\bot(y)| \leq \dist(x,P_{Q^*}) + \dist(y,P_{Q^*}) + \frac12 \epsilon_0 |x-y| \leq \frac34 \epsilon_0 |x-y| \leq \epsilon_0 |\Pi(x) - \Pi(y)|\]
if $\epsilon_0$ is small enough. The claim \eqref{claimPibot} follows.

\medskip
Define the closed set 
\begin{equation}\label{def Z}
    Z = \{p\in P, \, d(p) = 0\}.
\end{equation}
The Lipschiz function $b$ will be defined by two cases.  

\noindent{\bf Case $d(p) = 0$.} 
That is, $p\in Z$. In this case, since $\partial \Omega$ is closed, there necessarily exists $x\in \partial \Omega$ such that $\Pi(x) = p$. Moreover, \eqref{claimPibot} shows that such $x$ is unique, that is $\Pi$ is a one to one map on $Z$, and we define 
\begin{equation} \label{defAonZ}
b(p) := \Pi^\bot(\Pi^{-1}(p))  \qquad \text{ for } p\in Z.
\end{equation}

\medskip

\noindent{\bf Case $d(p)>0$.} We partition $P \setminus Z$ with a union of dyadic cubes, in the spirit of a Whitney decomposition, as follows. Construct the collection $\mathcal W_P$ as the subset of the dyadic cubes of $P$ that are maximal for the property 
\begin{equation} \label{prR1}
0 < 21 \ell(R) \leq \inf_{q\in 3R} d(q).
\end{equation}
By construction, $d(p) \approx d(q)$ whenever $p,q\in 3R\in \mathcal W_P$. Moreover, let us check that
\begin{equation} \label{prR3}
\ell(R_1)/\ell(R_2) \in \{1/2,1,2\} \quad \text{ whenever }  R_1,R_2 \in \mathcal W_P \text{ are such that } 3R_1 \cap 3R_2 \neq \emptyset.
\end{equation}
Indeed, if $R \in \mathcal W_P$ and $S$ is such that $\ell(S) = \ell(R)$ and $3S \cap 3R \neq \emptyset$, then $3S \in 9R$ and hence 
\[20\ell(S) = 20\ell(R) \leq \inf_{p\in 3R} d(p) \leq \inf_{p\in 9R} d(p) + 6\ell(R) \leq \inf_{p\in 3S} d(p) + 6\ell(S).\]
So every children of $S$ has to satisfies \eqref{prR1}, which proves \eqref{prR3}.

By construction of $\mathcal W_P$, for each $R\in \mathcal W_P$, we can find $Q_R\in \cS$ such that
\begin{multline} \label{prR2}
\dist(R,\Pi(Q_R)) \leq (2^6-2) \ell(R), \quad \ell(Q_R) \leq 2^5 \ell(R),\\
\text{ and either } Q_R =Q(\cS) \text{ or }  \ell(Q_R) = 2^5\ell(R) \approx \inf_{q\in 2R} d(q) \approx \sup_{q\in 2R} d(q).
\end{multline}
We want to associate each $R$ with an affine function $b_R: \, P \mapsto P^\bot$ such that the image of the function $\bb_R$ defined as $\bb_R(p) = (p,b_R(p))$ approximates $\partial \Omega$ well. First, we set 
\begin{equation} \label{defbRout}
b_R \equiv 0 \quad \text{ when } Q_R = Q(\mathcal \cS).
\end{equation}
When $Q_R \neq Q(\mathcal \cS)$, we take $b_R$ such that $\bb_R$ verifies
\begin{equation} \label{defbR}
\int_{999\Delta_{Q(\cS)}} |y - \bb_R(\Pi(y))| \1_{\Pi(y) \in 2R} \, d\sigma(y) := \min_{a} \int_{999\Delta_{Q(\cS)}} |y - \aa_R(\Pi(y))| \1_{\Pi(y) \in 2R} \, d\sigma(y),
\end{equation}
where the minimum is taken over the affine functions $a:\, P \mapsto P^\bot$ and $\aa(p) := (p,a(p))$. The uniqueness of the minimum is not guaranteed, but it does not matter for us. The existence is guaranteed, because $R\subset \Pi(3B_{Q_R}) \subset P \cap 999B_{Q_R}$ by \eqref{prR2}, and hence \eqref{defP'Qb} entails that the graph of the $a$ that almost realize the infimum are very close to the plane $P_Q$ which makes a small angle with $P$. The same argument shows that
\begin{equation} \label{prbR2}
\sup_{y \in 999\Delta_{Q_R}} |y - \bb_R(\Pi(y))| + \sup_{p\in \Pi(999B_{Q_R})} \dist(\bb_R(p),\partial \Omega) \leq C \epsilon_1 \ell(Q_R).
\end{equation}
for a constant $C>0$ that depends only on $n$ and
\begin{equation} \label{prbR}
\text{$b_R$ is $1.1\epsilon_0$-Lipschitz}
\end{equation}
if $0<\epsilon_1 \ll \epsilon_0 \ll 1$. We associate to the collection $\mathcal W_P$ a partition of unity $\{\varphi_R\}_{R\in \mathcal W_P}$ such that $\varphi_R \in C^\infty_0(2R_i)$, $|\nabla \varphi_R| \lesssim \ell(R)^{-1}$, and $\sum_R \varphi_R \equiv 1$ on $P \sm Z$.
We then define
\begin{equation} \label{defAonZc}
b(p) := \sum_{R\in \mathcal W_P}  \varphi_R(p) b_{R}(p) \qquad \text{ for } p\in P \sm Z.
\end{equation}
Due to \eqref{prR3}, the sum in \eqref{defAonZc} is finite and thus the quantity $b(p)$ is actually well defined. 

For $p\in P$, we define $\bb(p) := (p,b(p))$ to be the graph of $b$.  

\medskip

\begin{lemma} \label{LALip}
The function $b$ defined by \eqref{defAonZ} and \eqref{defAonZc} is $2\epsilon_0$-Lipschitz and supported in $P\cap 4B_{Q(\cS)}$.
\end{lemma}

\bp Recall that the property \eqref{prR2} implies that $2R \subset P \cap \Pi(3B_{Q_R})$ as long as $Q_R \neq Q(\cS)$. 
So if $p\notin P\cap \Pi(3B_{Q(\cS)})$ and $R\in \mathcal W_P$ is such that $p\in 2R$, we necessarily have $Q_R = Q(\cS)$ and then $b_R(p) = 0$ by \eqref{defbRout}. We conclude that $b(p) = 0$ and thus that $b$ is supported in $P\cap \Pi(3B_{Q(\cS)}) \subset P \cap 4B_{Q(\cS)}$.

\medskip

Now, we want to show that $b$ is Lipschitz. The fact that $b$ is Lipschitz on $Z$ is an immediate consequence from the definition \eqref{defAonZ} and \eqref{claimPibot}. Let us prove now that $b$ is Lipschitz on the interior of $2R_0$ for every $R_0\in \mathcal W_P$. Take $R_0 \in \mathcal W_P$ and $p\in 2R_0 \setminus \partial (2R_0)$. Then, since $\sum \nabla \varphi_R(p) = 0$, we have 
\begin{equation} \label{LALip1}\begin{split}
|\nabla b(p)| & = \left| \sum_{\begin{subarray}{c} R\in \mathcal W_P \\ 2R \cap 2R_0 \neq \emptyset \end{subarray}} \varphi_R(p) \nabla b_{Q_R}(p) + \sum_{\begin{subarray}{c} R\in \mathcal W_P \\ 2R \cap 2R_0 \neq \emptyset \end{subarray}} b_{Q_R}(p) \nabla \varphi_R(p) \right|  \\
& \leq \sup_{\begin{subarray}{c} R\in \mathcal W_P \\ 2R \cap 2R_0 \neq \emptyset \end{subarray}} |\nabla b_{Q_R}(p)| + \sum_{\begin{subarray}{c} R\in \mathcal W_P \\ 2R \cap 2R_0 \neq \emptyset \end{subarray}} |\nabla \varphi_R(p)| |b_{Q_R}(p) - b_{Q_{R_0}}(p)| \\
& \leq 1.1\epsilon_0 + C \ell(R_0)^{-1} \sup_{\begin{subarray}{c} R\in \mathcal W_P \\ 2R \cap 2R_0 \neq \emptyset \end{subarray}}  |b_{Q_R}(p) - b_{Q_{R_0}}(p)|
\end{split}\end{equation}
by \eqref{prbR} and \eqref{prR3}. We can assume that $p\in 2R_0 \subset  P\cap 4B_{Q(\cS)}$, because we have already shown that $b(p) = 0$ otherwise. So due to \eqref{prR2} and \eqref{prR3}, both $Q_R$ and $Q_{R_0}$ are close to $2R_0$, in the sense that 
\[3R_0 \subset P \cap 999\Pi(B_{Q_R}),\]
so we can invoke \eqref{prbR2} to say that $\dist(\mathfrak b_{Q_R}(p),P_{Q_{R_0}}) \lesssim \epsilon_1 \ell(Q_{R_0})$ and then 
\begin{equation} \label{LALip2}
|b_{Q_R}(p) - b_{Q_{R_0}}(p)| \lesssim \epsilon_1 \ell(Q_{R_0}).
\end{equation}
 So if $\epsilon_1 \ll \epsilon_0$ is small enough, \eqref{LALip1} becomes $|\nabla b(p)| \leq 2\epsilon_0$.

We proved that $b$ is Lipschitz on $Z$ and $P\sm Z$, so it remains to check that $b$ is continuous at every point in $\partial Z$. Take $z\in \partial Z$ and set $x:= \mathfrak b(z) \in \partial \Omega$. Take also $p\in P\sm Z$ such that $|p-z| \ll 1$. Due to \eqref{prbR2} and \eqref{LALip2}, we have the existence of $y\in \partial \Omega$ such that, for any $R\in \mathcal W_P$ satisfying $p\in 2R$, we have 
\begin{equation} \label{LALip3}
|y - \mathfrak b_{Q_R}(p)| \lesssim \epsilon_1 \ell(R) \lesssim \epsilon_1 d(p) \leq \epsilon_1 |p-z|
\end{equation}
by \eqref{prR1} and the fact that $q\to d(q)$ is 1-Lipschitz. The latter bound shows in particular that 
\begin{equation} \label{LALip5}
|y - \mathfrak b(p)| \leq \epsilon_0 |p-z|
\end{equation}
if $\epsilon_1/\epsilon_0$ is small enough. The bound \eqref{LALip5} also implies that $\Pi(x) \neq \Pi(y)$ and then $x \neq y$, and so \eqref{claimPibot} entails that 
\begin{equation} \label{LALip4}
|b(z) - \Pi^\bot(y)| = |\Pi^\bot(x) - \Pi^\bot(y)| \leq 2\epsilon_0|z - \Pi(y)|.
\end{equation}
The combination of \eqref{LALip5} and \eqref{LALip4} proves that the restriction of $b$ to $P\sm Z$ has the limit $b(z)$ at the point $z\in \partial \Omega$. Since it is true for all $z\in \partial Z$, and since $b$ is already continuous (even Lipschitz) on $Z$ and $P\sm Z$, we conclude that $b$ is continuous on $P$.
The lemma follows.
\ep

We prove that the graph of $b$ is well approximated by the same plane as the ones that approximate $\partial \Omega$, as shown below.

\begin{lemma} \label{Lbbeps}
For  $Q\in \cS$, we have
\[  \sup_{p\in P \cap \Pi(2^8B_Q)} \Big[\dist(\bb(p),\partial \Omega) + \dist(\bb(p), P_{Q}) \Big] \lesssim \epsilon_1 \ell(Q).\]
\end{lemma}

\bp
Take $p\in \Pi(2^8B_Q)$. If $p\in Z$, then $\bb(p) \in \partial \Omega$, but since we also have \eqref{claimPibot}, we deduce $\bb(p)\in 2^9\Delta_Q$. The bound $\dist(\bb(p),P_Q) \leq C\epsilon_1\ell(Q)$ is then just a consequence of \eqref{defP'Qb}. 

Assume now that $p\in P\setminus Z$. We have $d(p) \leq 2^8\ell(Q)$ so any $R$ that verifies $p\in 2R$ is such that $21\ell(R) \leq d(p) \leq 2^8\ell(Q)$ by \eqref{prR1}, that implies $\ell(Q_R) \leq 2^9\ell(Q)$ by \eqref{prR2}. 
Since $\bb(p)$ is a weighted average of the $\bb_R(p)$, the estimate \eqref{prbR2} on $\bb_R(p)$ gives that 
\[\dist(\bb(p),\partial \Omega) \lesssim \epsilon_1 \sup_{R: \, p\in 2R}\ell(Q_R) \lesssim \epsilon_1 \ell(Q).\]
If $x\in \partial \Omega$ is such that $|\bb(p) - x| = \dist(\bb(p),\partial \Omega)$, then we have again by \eqref{claimPibot} that $x\in 2^9\Delta_Q$ so \eqref{defP'Qb} gives that $\dist(x,P_Q) \leq \epsilon_1$. We conclude
\[\dist(\bb(p),P_Q) \leq |\bb(p) - x| + \dist(x,P_Q) \lesssim \epsilon_1\]
as desired.
\ep

We also need a $L^1$ version of the above lemma, and with a better control in term of the $\alpha_\sigma(Q)$ (which is smaller than $\epsilon_1$ when $Q\in \cS$). 

\begin{lemma} \label{Lbbalpha}
For  $Q\in \cS$, we have
\[\int_{P\cap \Pi(2B_{Q})} \dist(\bb(p), P_{Q}) \, dp \lesssim \ell(Q)^n \alpha_\sigma(Q).\]
\end{lemma} 

\bp  The plane $P$ is the union of $Z$ and $P\setminus Z := \bigcup_{R\in W_P} R$, so 
\[\begin{split}
I & := \int_{P\cap \Pi(2B_{Q})} \dist(\bb(p), P_{Q}) \, dp \\
& =  \int_{Z\cap \Pi(2B_{Q})} \dist(\bb(p), P_{Q}) \, dp + \int_{Z^c \cap \Pi(2B_{Q})}\dist(\bb(p), P_{Q}) \, dp  := I_1 + I_2.
\end{split}\]
The term $I_1$ is easy, because $\bb(p) \in 4\Delta_Q \subset \partial \Omega$ by \eqref{claimPibot}, and so we have
\[I_1 \lesssim  \int_{4\Delta_Q} \dist(y,P_Q) \, d\sigma(y)\]
We apply \eqref{ftestinalpha} with the test function
\[f(y):= \min\{\dist(y,\R^n \setminus 999B_{Q}), \dist(y,P_{Q})\}\]
which lies in $Lip_Q$ and takes the value $0$ on $P_{Q}$ and $\dist(y,P_{Q})$ on $4\Delta_{Q}$, and we conclude that
\[I_1 \lesssim \int f \, d\sigma = \left| \int f \, d\sigma - \int f \, d\mu_{Q}\right| \lesssim \ell(Q)^{n} \alpha_\sigma(Q)\]
as desired.

We turn to the bound on $I_2$. We know that $\Angle(P_{Q},P) \leq \epsilon_0$ so $P_{Q}$ is the graph of an affine function $a_{Q}: \, P \mapsto P^\bot$ with small Lipschitz constant. Therefore, we have
\[ I_2 \approx \int_{P\cap \Pi(2B_Q))} |b(p)- a_{Q}(p)| \, dp\]
Let $\mathcal W_P(Q)$ be the subfamily of $\mathcal W_P$ of elements $R$ such that $2R$ that intersects $\Pi(2B_{Q})$. The fact that $2R \cap \Pi(2B_{Q}) = \emptyset$ implies by \eqref{prR1} that $21 \ell(R) \leq \ell(Q)$. Consequently, $\ell(R) \leq 2^{-5} \ell(Q)$ because both $\ell(R)$ and $\ell(Q)$ are in the form $2^k$, and then $2R \subset \Pi(3B_{Q})$. 

Assume first that $Q \varsubsetneq Q(\cS)$, and check that this condition implies that $\ell(Q_R) \leq 2^5\ell(R) \leq \ell(Q) < \ell(Q(\cS))$, hence $Q_R \neq Q(\cS)$ for every $R\in \mathcal W_P(Q)$. So we have
\[\begin{split}
I_2  & = \int_{Z^c \cap \Pi(2B_{Q})} \left|\sum_{R\in \mathcal W_P(Q)} \varphi_R(p) (b_{R}(p) - a_{Q}(p))\right| \, dp  \leq \sum_{R\in \mathcal W_P(Q)} \int_{2R}  |b_{R}(p) - a_{Q}(p)| \, dp.
\end{split}\]
We want to estimate $\int_{2R}  |b_{R}(p) - a_{Q}(p)| \, dp$, but now both $b_R$ and $a_{Q}$ are affine, so knowing $|b_{R}(p) - a_{Q}(p)|$ for $n$ different points $p\in 2R$ that are far from each other is enough. By \eqref{defP'Qb}, we know that $\Pi(\partial \Omega) \cap 2R$ contains many points all over $2R$, and by using those points to estimate the distance between $b_R$ and $a_{Q}$, we deduce that
\begin{multline*} 
\int_{2R}  |b_{R}(p) - a_{Q}(p)| \, dp  \lesssim \int_{999\Delta_{Q(\cS)}} |b_R(\Pi(y)) - a_{Q}(\Pi(y)) | \1_{\Pi(y) \in 2R} \, d\sigma(y)  \\
 \leq \int_{999\Delta_{Q(\cS)}} |\Pi^{\bot}(y) - b_R(\Pi(y))| \1_{\Pi(y) \in 2R} \, d\sigma(y) + \int_{999\Delta_{Q(\cS)}} |\Pi^\bot(y) - a_{Q}(\Pi(y)) | \1_{\Pi(y) \in 2R} \, d\sigma(y) \\
 \lesssim \int_{999\Delta_{Q(\cS)}} |\Pi^\bot(y) - a_{Q}(\Pi(y)) | \1_{\Pi(y) \in 2R} \, d\sigma(y)
\end{multline*}
by \eqref{defbR}, because $Q_R \neq Q(\cS)$. Since the $2R$ are finitely overlapping, see \eqref{prR3}, the bound on $I_2$ becomes
\begin{equation} \label{boundI2}
I_2 \lesssim  \int_{4\Delta_{Q})} |\Pi^\bot(y) - a_{Q}(\Pi(y)) |  \, d\sigma(y) \lesssim \int_{4\Delta_{Q}} \dist(y,P_{Q}) \, d\sigma(y).
\end{equation}
We had the same bound on $I_1$, and with the same strategy, we can conclude that
\[I_2 \lesssim \int f \, d\sigma = \left| \int f \, d\sigma - \int f \, d\mu_{Q}\right| \lesssim \ell(Q)^{n} \alpha_\sigma(Q)\]
as desired. 

If $Q = Q(\cS)$, the same computations apply. It is possible to have some $R$ in $\mathcal W_P(Q)$ for which $Q_R = Q(\cS)$ and thus $b_R \equiv 0$, but at the same time, we now have $a_Q \equiv 0$, so those $R$ verify $b_R - a_Q \equiv 0$ and do no have any contribution in the above bounds on $I_2$. Therefore, we also conclude that
\[I_2 \lesssim \ell(Q(\cS))^{n} \alpha_\sigma(Q(\cS)) = \ell(Q)^{n} \alpha_\sigma(Q).\]
The lemma follows.
\ep

\section{Whitney regions for coherent regimes}

\label{SWhitney}

 We associate the dyadic cubes of $\partial \Omega$ to Whitney regions in $\Omega$ and therefore associate the coherent family of dyadic cubes  obtained in the corona decomposition to a subset of $\Omega$. The idea is similar to the construction found in \cite{HMMDuke}, but we need different properties than those in \cite{HMMDuke}, so we rewrite the construction.
 
This section will prove the following extension of Lemma \ref{LcoronaG}.

\begin{lemma} \label{lemWOG}
Let $\partial \Omega$ be a uniformly rectifiable sets. We keep the notation from Lemma \ref{LcoronaG}, and we further have the existence of $K^{**}>0$ and a collection $\{\Psi_\cS\}_{\cS \in \mathfrak S}$ of functions such that
\begin{enumerate}[(a)]
    \item $\Psi_\cS$ are cut-off functions, that is $0 \leq \Psi_{\cS}\leq 1$, and $|\nabla \Psi_\cS| \leq 2\delta^{-1}$.
    \item For any $\cS \in \mathfrak S$, if $X \in \supp (1-\Psi_\cS)$, then there exists $Q\in \cS$ such that
    \[ \ell(Q)/2 < \delta(X) = \dist(X,Q) \leq \ell(Q).\]
    \item If $X \in \supp \Psi_{\cS}$, then there exists $Q\in \mathcal \cS$ and such that
    \[ \ell(Q)/2^6 < \delta(X) = \dist(X,2^6\Delta_Q) \leq 2^6 \ell(Q).\]
    \item For any $\cS\in \mathfrak S$ and any $X\in \supp \Psi_\cS$, we have
    \begin{equation} \label{prWSa}
    (1-2\epsilon_0) |X - \bb(\Pi(X))| \leq \delta(X) \leq (1+2\epsilon_0) |X - \bb(\Pi(X))|
    \end{equation}
    and, if $\Gamma_\cS$ is the graph of $\cS$,
    \begin{equation} \label{prWSb}
    (1-2\epsilon_0) \dist(X,\Gamma_{\cS}) \leq \delta(X) \leq (1+3\epsilon_0) \dist(X,\Gamma_{\cS}).
    \end{equation}
    \item There exists a collection of dyadic cubes $\{Q_i\}_{i\in I_\cS}$ in $\D_{\partial \Omega}$ such that $\{2Q_i\}_{i\in I_\cS}$ has an overlap of at most 2, and 
    \[ \Omega \cap (\supp \Psi_{\cS}) \cap \supp (1-\Psi_{\cS}) \subset \bigcup_{i\in I_\cS} \Big\{ (K^{**})^{-1}\ell(Q_i) < \delta(X) = \dist(X,K\Delta_{Q_i}) \leq K^{**}\ell(Q_i)\Big\}.\]
    In particular, $|\delta \nabla \Psi_{\cS}| \in CM_\Omega(C)$ with a constant $C>0$ that depends only on $n$.
\end{enumerate}
\end{lemma}

\subsection{Whitney decomposition}

We divide $\Omega$ into Whitney regions. Usually, one constructs them with dyadic cubes of $\R^n$, but we prefer to construct them directly. We recall that $\delta(X):= \dist(X,\partial \Omega)$, and for $Q\in \D_{\partial \Omega}$, we define 
\begin{equation} \label{defWO}
W_{\Omega}(Q) := \{X\in \Omega: \, \exists\, x\in Q \text{ such that } \ell(Q)/2 < \delta(X) = |X-x| \leq \ell(Q)\}.
\end{equation}
It is easy to see that the sets $\{W_\Omega(Q)\}_{Q\in \partial \Omega}$ covers $\Omega$. The sets $W_\Omega(Q)$ are not necessarily disjoint, but we do not care, we are perfectly happy if $\{W_\Omega(Q)\}_{Q\in \D_{\partial \Omega}}$ is finitely overlapping, and we choose $W_\Omega(Q)$ small only because it will make our estimates easier. The sets $W_\Omega(Q)$ can be disconnected and have a bad boundary, but that is not an issue, since - contrary to \cite{HMMDuke} - we won't try to prove that the $W_\Omega(Q)$ are Chord-Arc Domains. 

We also need fattened versions of $W_\Omega(Q)$, that we call $W_\Omega^*(Q)$ and $W_\Omega^{**}(Q)$, which are defined as
\begin{equation} \label{defWO*}
W_{\Omega}^*(Q) := \{X\in \Omega: \, \exists \, x\in 2^6\Delta_Q \text{ such that } 2^{-6}\ell(Q) <  \delta(X) = |X-x|  \leq 2^6\ell(Q)\}
\end{equation}
and
\begin{equation} \label{defWO**}
W_{\Omega}^{**}(Q) := \{X\in \Omega: \, \exists \, x\in K^{**}\Delta_Q \text{ such that } (K^{**})^{-1}\ell(Q) <  \delta(X) = |X-x|  \leq K^{**}\ell(Q)\}
\end{equation}
The exact value of the constant $K^{**}$ does not matter. In Lemma \ref{LprPsiS}, we will choose it large enough to fit our purpose. The first properties of $W_\Omega(Q)$ and $W_\Omega^*(Q)$ are the ones that we expect and are easy to prove. We have
\begin{equation} \label{prWO1}
\Omega = \bigcup_{Q\in D_{\partial \Omega}} W_\Omega(Q),
\end{equation}
\begin{equation} \label{prWO2}
\diam(W^{*}_\Omega(Q)) \leq 2^7 \ell(Q),
\end{equation}
and
\begin{equation} \label{prWO4}
W^{*}_\Omega(Q) \subset 2^8 B_Q.
\end{equation}
We want $W_\Omega(Q)$ and $W_\Omega^*(Q)$ to be so that we can squeeze a cut-off function between the two sets, which is possible because 
\begin{equation} \label{prWO3}
\dist(W_\Omega(Q), \R^n \sm W^*_\Omega(Q)) \geq \frac14 \ell(Q).
\end{equation}
Indeed, if $X\in W_\Omega(Q)$ and $|X-Y| \leq \ell(Q)/4$, then $\ell(Q)/4 \leq \dist(Y,\partial \Omega) \leq 5\ell(Q)/4$ and for any $y \in \partial \Omega$ such that $|Y-y| = \delta(Y)$, we have
\[|y-x| \leq |y-Y| + |Y-X| + |X-x| \leq \frac54\ell(Q) + \frac14 \ell(Q) + \ell(Q) \leq 3\ell(Q),\]
so in particular, $y\in 2^5Q$, and thus $Y\in W^*_\Omega(Q)$. The claim \eqref{prWO3} follows.

\subsection{Coherent regions associated to coherent regimes}

As before, we pick $0 < \epsilon_1 \ll \epsilon_0 \ll 1$, and then construct the collection of coherent regimes $\mathfrak S$ given by Lemma \ref{Lcorona}. Let then $\cS$ be either in $\mathfrak S$, or a coherent regime included in an element of $\mathfrak S$. For such $\cS$, we define the regions
\begin{equation} \label{defWS}
W_{\Omega}(\cS) :=  \bigcup_{Q\in \cS} W_\Omega(Q) \quad \text{ and } \quad W_{\Omega}^*(\cS) :=  \bigcup_{Q\in \cS} W^*_\Omega(Q).
\end{equation}

Associated to the coherent regime $\cS$, we have affine planes $P$ and $P^\bot$, the projections $\Pi$ and $\Pi^\bot$, a Lipschitz function $b: \, P \to P^\bot$, and $\bb(p)=(p,b(p))$ as in  Subsection~\ref{SSLipschitz}. We also have the ``distance function'' $d(p)$ defined in \eqref{defdx}. We now define the Lipschitz graph
\begin{equation} \label{defGS}
\Gamma_{\cS} := \{\mathfrak b(p), \, p\in P\} \subset \R^n.
\end{equation}

\begin{lemma} \label{LclaimPib}
If $X\in W^*_\Omega(\cS)$ and $x\in \partial \Omega$ is such that $|X-x| = \delta(X)$, then
\begin{equation} \label{claimdPib}
(1-2\epsilon_0) \delta(X) \leq |X - \bb(\Pi(X))| \leq (1+2\epsilon_0) \delta(X),
\end{equation}
\begin{equation} \label{claimdGPib}
(1-2\epsilon_0) \dist(X,\Gamma_{\cS}) \leq \delta(X) \leq (1+3\epsilon_0) \dist(X,\Gamma_{\cS}),
\end{equation}
and
\begin{equation} \label{claimbxd}
|\bb(\Pi(X)) - x | \leq 2\epsilon_0 \delta(X).
\end{equation}
\end{lemma}

\bp
Since $X\in W^*_\Omega(\cS)$, there exists $Q\in \cS$ such that $X\in W^*_\Omega(Q)$. Such $Q$ verifies $x\in 2^6\Delta_Q$ and
\[2^{-6} |X- x| \leq \ell(Q) \leq 2^6 |X-x|,\]
so $X \in 2^7 B_Q$ and $\Pi(X) \in \Pi(2^7 B_Q)$. Lemma \ref{Lbbeps} and \eqref{defP'Qb} entail that
\[\dist(x,P_Q) + \dist(\bb(\Pi(X)),P_Q) \leq C \epsilon_1 \ell(Q) \leq \frac18 \epsilon_0 |X-x| \]
if $\epsilon_1/\epsilon_0$ is small enough. Because the plane $P_Q$ makes a small angle with $P$, we deduce that
\begin{equation} \label{ccl48a}
|b(\Pi(X)) - \Pi^\bot(x)| \leq \frac14 \epsilon_0 |X-x|
\end{equation}
if $\epsilon_0$ is small enough. Define $\Pi_Q$ and $\Pi^\bot_Q$ as the projection onto $P_Q$ and $P_Q^\bot$. We have $|\Pi_Q(x) - x| \lesssim \epsilon_1 |X-x|$ thanks to \eqref{defP'Qb}. In addition, the projection $\Pi_Q$ lies in $P_Q \cap 2^8 B_Q$, so using \eqref{defP'Qb} again gives the existence of $y\in \partial \Omega$ such that $|\Pi_Q(X) - y| \leq \epsilon_1\ell(Q) \lesssim \epsilon_1 |X-x|$. By definition of $x$, the point $y$ has to be further away from $X$ that $x$ so 
\[\begin{split}
|X-x| & \leq |X-y| \leq |X - \Pi_Q(X) - x + \Pi_Q(x)| + |\Pi_Q(x) - x| + |\Pi_Q(X) - y| \\
& \leq |\Pi^\bot_Q(X)- \Pi^\bot_Q(x)| + C\epsilon_1|X-x|.
\end{split}\]
So one has $|\Pi^\bot_Q(X)-\Pi^\bot_Q(x)| \geq (1-C\epsilon_1) |X-x|$ and hence $|\Pi_Q(X)- \Pi_Q(x)| \leq C\sqrt{\epsilon_1}$. Since $P_Q$ makes an angle at most $\epsilon_0$ with $P$, we conclude that
\begin{equation} \label{ccl48b}
|\Pi(X) - \Pi(x)| \leq \frac32\epsilon_0|X-x|
\end{equation}
if $\epsilon_0$ and $\epsilon_1/\epsilon_0$ are small enough. The two bounds \eqref{ccl48a} and \eqref{ccl48b} easily prove \eqref{claimbxd}, and also prove \eqref{claimdPib} by writing
\[\begin{split}
\Big| |X - \bb(\Pi(X))| - |X-x| \Big| & \leq \Big| |\Pi^\bot(X) - b(\Pi(X))| - |\Pi^\bot(X)- \Pi^\bot(x)| \Big| + |\Pi(X) - \Pi(x)| \\
& \leq |\Pi^\bot(x) - b(\Pi(X))| + |\Pi(X) - \Pi(x)| \\
& \leq 2\epsilon_0|X-x|.
\end{split}\]
The bounds \eqref{claimdGPib} is just a consequence of \eqref{claimdPib} and the fact that $\Gamma_{\cS}$ is the graph of $b$ which is a $2\epsilon_0$-Lipschitz function with $\epsilon_0 \ll 1$.
The lemma follows.
\ep

Let $\psi \in C^\infty_0(\R)$ be such that $0 \leq \psi \leq 1$, $\psi \equiv 1$ on $[0,1]$, $\psi \equiv 0$ on $[2,\infty)$ and $|\nabla \psi| \leq 2$. We set
\begin{equation} \label{defPsiS}
\Psi_{\cS}(X) = \1_\Omega(X) \psi\Big(\frac{d(\Pi(X))}{3|X - \bb(\Pi(X))|}\Big) \psi\Big(\frac{|X - \bb(\Pi(X))|}{2\ell(Q(\cS))}\Big).
\end{equation}

We want to prove the points $(b)$, $(c)$, and $(d)$ of Lemma \ref{lemWOG}, that is

\begin{lemma} \label{LprWS}
The function $\Psi_{\cS}$ is constant equal to 1 on $W_\Omega(\cS)$ and $\Omega \cap \supp \Psi_{\cS} \subset  W^*_\Omega(\cS)$.
Consequently, for any $X\in \supp \Psi_\cS $, we have \eqref{prWSa} and \eqref{prWSb} by Lemma \ref{LclaimPib}.
\end{lemma}

\begin{remark}
We know from its definition that $\Psi_\cS \equiv 0$ on $\R^n \setminus \Omega$, but the support of $\Psi_\cS$ can reach the boundary $\partial \Omega$. So if $\Omega \cap \supp \Psi_\cS \subset W_\Omega^*(\cS)$, then we actually have
\[\supp \Psi_{\cS} \subset W_\Omega^*(\cS) \cup \Big( \partial \Omega \cap \overline{W_\Omega^*(\cS)} \Big).\]
\end{remark}

\bp Take $Q\in \cS$ and $X\in W_\Omega(Q)$, and pick $x\in Q$ such that $|X-x| = \delta(X)$. We want to show that $\Psi_{\cS} (X) = 1$, i.e. that
\begin{equation} \label{414a}
d(\Pi(X)) \leq 3|X - \bb(\Pi(X))|
\end{equation}
and
\begin{equation} \label{414b}
|X - \bb(\Pi(X))| \leq 2\ell(Q(\cS)).
\end{equation}
For \eqref{414b}, it suffices to notice that $|X - \bb(\Pi(X))| \leq 2 |X-x| \leq 2\ell(Q) \leq 2\ell(Q(\cS))$ by \eqref{claimdPib} and by the definition of $x$ and $Q$. As for \eqref{414a}, observe that 
$\abs{X-x}\le 2\epsilon_0\delta(X)+\abs{X-\bb(\Pi(X))}$ by the triangle inequality and \eqref{claimbxd},
and thus
\[d(\Pi(X)) \leq \ell(Q) \leq 2|X-x| \leq 3 |X - \bb(\Pi(X))|\]
by \eqref{claimdPib}.

It remains to verify that $\supp \Psi_{\cS}$ is supported in $W^*_\Omega(\cS)$, because \eqref{prWSa} and \eqref{prWSb} are then just \eqref{claimdPib} and \eqref{claimdGPib}.
So we pick $X\in \supp \Psi_{\cS}$ which means in particular that 
\begin{equation} \label{414c}
d(\Pi(X)) \leq 6|X - \bb(\Pi(X))|
\end{equation}
and
\begin{equation} \label{414d}
|X - \bb(\Pi(X))| \leq 4\ell(Q(\cS)),
\end{equation}
and we want to show that $X\in W^*_\Omega(\cS)$. By definition of $d(\Pi(X))$, there exists $Q \in \cS$ such that 
\[\dist(\Pi(X),\Pi(2B_Q)) + \ell(Q) = d(\Pi(X))  \leq 6 |X - \bb(\Pi(X))| \leq 24 \ell(Q(\cS))\]
by \eqref{414c} and \eqref{414d}. Since $\cS$ is coherent, by taking a suitable ancestor of $Q$, we can find $Q_X \in \cS$ such that
\begin{equation} \label{414f1}
\frac14 |X - \bb(\Pi(X))| \leq \ell(Q_X) \leq 6 |X - \bb(\Pi(X))|
\end{equation}
and
\begin{equation} \label{414f2}
\Pi(X) \in 26 \Pi(B_{Q_X}).
\end{equation}
We want to prove that $X\in W_\Omega^*(Q_X)$. The combination of \eqref{414f2}, Lemma \ref{Lbbeps}, and \eqref{claimPibot} forces $\bb(\Pi(X)) \in 27B_{Q(\cS)}$ when $\epsilon_0$ is small, and hence $X \in 31B_X$ by \eqref{414f1}.
Let $x\in \partial \Omega$ such that $|X-x| = \delta(X)$. Since $X\in 31B_X$, we have $x\in 2^6\Delta_{Q_X}$, and of course $|X-x| \leq 2^6\ell(Q_{X})$. So it remains to verify if $|X-x| \geq 2^{-6} \ell(Q_X)$. 
In one hand, thanks to \eqref{defP'Qb}, we know that $x$ lies close to $P_Q$, in the sense that $\dist(x,P_{Q_X}) \leq \epsilon_1 \ell(Q_X)$.  In the other hand, if $P_{Q_X}$ is the graph of the function $a_{Q_X}: \, P \mapsto P^\bot$, we have
\[\begin{split}
\dist(X,P_{Q_X}) & \geq (1-\epsilon_0)  |\Pi^\bot(X) - a_{Q_X}(\Pi(X))| \\
& \geq (1-\epsilon_0) \Big[ |X-\bb(\Pi(X))| - \dist(\bb(\Pi(X)),P_{Q_X}) \Big] \\
& \geq (1-\epsilon_0-C\epsilon_1) |X-\bb(\Pi(X))| \\
& \geq \frac16(1-\epsilon_0-C\epsilon_1)  \ell(Q_X)
\end{split}\]
by Lemma \ref{Lbbeps} and \eqref{414f1}; that is $X$ is far from $P_{Q_X}$. Altogether, we deduce that 
\[|X-x| \geq (1-C\epsilon_1) \dist(X,P_{Q_X}) \geq (1-\epsilon_0-C\epsilon_1) |X-\bb(\Pi(X))| \geq \frac18 \ell(Q_X).\]
if $\epsilon_0$ and $\epsilon_1$ are small. The lemma follows. 
\ep

We are left with the proof of point $(e)$ in Lemma \ref{lemWOG}, which is:

\begin{lemma} \label{LprPsiS}
There exists a collection of dyadic cubes $\{Q_i\}_{i\in I_\cS}$ in $\D_{\partial \Omega}$ such that $\{2Q_i\}_{i\in I_\cS}$ has an overlap of at most 2, and 
\[ \Omega \cap (\supp \Psi_{\cS}) \cap \supp (1-\Psi_{\cS}) \subset \bigcup_i W^{**}_\Omega(Q_i).\]
\end{lemma}

\bp
Observe that $(\supp \Psi_{\cS}) \cap \supp (1-\Psi_{\cS}) \subset E_1 \cup E_2$ where
\[E_1 := \{X\in W_\Omega^*(\cS),  \, 2\ell(Q(\cS)) \leq |X- \bb(\Pi(X))| \leq 4\ell(Q(\cS))\}\]
and
\[E_2 := \{X\in W_\Omega^*(\cS),  \, d(\Pi(X))/6 \leq |X- \bb(\Pi(X))| \leq d(\Pi(X))/3\}.
\]
Thanks to \eqref{prWSa}, the set $E_1$ is included in $W^{*}_\Omega(Q(\cS))$.

For each $X\in E_2$, we construct the ball $B_X:= B(\bb(\Pi(X)),d(\Pi(X))/100) \subset \R^n$. The radius of $B_X$ is bounded uniformly by $\ell(Q(\cS))/4$. 
So by the Vitali lemma, we can find a non overlapping subfamily $\{B_{X_i}\}_{i\in I_2}$ such that $E_2 \subset \bigcup_{i\in I_2} 5B_{X_i}$.
We use \eqref{claimbxd} and \eqref{claimdPib} to find a point $x_i \in \frac12B_{X_i} \cap \partial \Omega$. We take $Q_i \in \D_{\partial \Omega}$ to be the unique dyadic cube such that such that $x_i \in Q_i$ and $\ell(Q_i) < d(\Pi(X_i))/400 \leq 2\ell(Q_i)$. 
By construction, we have $2Q_i \subset B_{X_i}$, so the $\{2Q_i\}_{i\in I_2}$ are non-overlapping, and $5B_{X_i} \subset 100B_{Q_i}$.

It remains to check that $E_2 \subset \bigcup_i W^{**}_\Omega(Q_i)$. Take $X\in E_2$. From what we proved, there exists an $i\in I_2$ such that 
\begin{equation} \label{423a}
|\Pi(X) - \Pi(X_i)| \leq |\bb(\Pi(X))- \bb(\Pi(X_i))| \leq d(\Pi(X_i))/20.
\end{equation}
Observe from the definition that $d$ is $1$-Lipschitz. Therefore,
\[|d(\Pi(X)) - d(\Pi(X_i))| \leq |\Pi(X) - \Pi(X_i)| \leq d(\Pi(X_i))/20\]
and
\begin{equation} \label{423b}
\frac{19}{20} d(\Pi(X_i)) \leq d(\Pi(X)) \leq \frac{21}{20} d(\Pi(X_i)).
\end{equation}
From \eqref{423a} and \eqref{423b}, we obtain
\[|X-x_i| \leq |X-\bb(X)| + |\bb(X) - \bb(X_i)| + |\bb(X_i) - x_i| \leq d(\Pi(X)) \leq 800\ell(Q_i)\]
and, from \eqref{prWSa} and \eqref{423b}, we get 
\[\delta(X) \geq (1-2\epsilon_0)|X-\bb(\Pi(X))| \geq \frac17 d(\Pi(X)) \geq \frac18 d(\Pi(X_i)) \geq 50\ell(Q_i).\]
The last two computations show that $X\in W^{**}_\Omega(Q_i)$ if $K^{**} \geq 1601$. The lemma follows.
\ep

\section{Replacement lemma and application to the smooth distance $D$}

\label{SDbeta}

As usual, let $0 < \epsilon_1 \ll \epsilon_0 \ll 1$, and then construct the collection of coherent regimes $\mathfrak S$ given by Lemma \ref{Lcorona}. We take then $\cS$ to be either in $\mathfrak S$, or a coherent regime included in an element of $\mathfrak S$. 

In Lemma \ref{Lbbalpha}, we started to show that the graph of $b$ behaves well with respect to the approximating planes $P_Q$, and we want to use the graph of $b$ as a substitute for $\partial \Omega$. Roughly speaking, the graph of the Lipschitz function $b$ is ``a good approximation of $\partial \Omega$ for the regime $\cS$''. Let us explain what we mean by this. The Lipschitz graph $\Gamma_\cS$ defined in \eqref{defGS}
is uniformly rectifiable, that is, $\Gamma_{\cS}$ is well approximated by planes. And even better, we can easily construct explicit planes that approximate $\Gamma_{\cS}$. 

First, we equip $P$ with an Euclidean structure, which means that $P$ can be identified to $\R^{n-1}$. Similarly, we identify $P^\bot$ to $\R$, and of course, we choose $P$ and $P^\bot$ such that $\Pi^{\bot}(P) = \{0\}$ and $\Pi(P^{\bot}) = \{0\}$, and so $\R^n$ can be identified to $P \times P^\bot$.

We take a non-negative radial smooth function $\eta \in C^\infty_0(P,\R_+)$ which is supported in the unit ball and that satisfies $\int_{P} \eta dx = 1$. Even if $P$ depends on the regime $\cS$, $P$ is  identified to $\R^{n-1}$, so morally the smooth function $\eta$ is defined on $\R^{n-1}$ and does not depend on anything but the dimension $n$. For $t\neq 0$, we construct the approximation of identity by  
\begin{equation} \label{defetat}
\eta_t(p) := \abs{t}^{1-n} \eta\Big(\frac p{\abs{t}}\Big),
\end{equation}
then the functions 
\begin{equation} \label{defbt}
b^t := \eta_t * b, \quad \mathfrak b^t := \eta_t * \mathfrak b,
\end{equation}
and the planes 
\begin{equation} \label{defLambda}
\Lambda(p,t):= \{(q,(q-p) \nabla b^t(p) + b^t(p)), \, q\in P\}.
\end{equation}
Notice that $\Lambda(p,t)$ is the tangent plane of the approximating graph $\set{\bb^t(p), \, p\in P}$ at $\bb^t(p)$.
What we actually want is flat measures, so we fix a radial function $\theta \in C^\infty(\R^n)$ such that $0\leq \theta \leq 1$, $\supp \theta \subset B(0,1)$, and $\theta \equiv 1$ on $B(0,\frac12)$. We set then
\begin{equation} \label{deftheta}
\theta_{p,t}(y) := \theta \left( \frac{\mathfrak b^t(p) - y}{t} \right)
\end{equation}
and
\begin{equation} \label{deflambda}
\lambda(p,t) := \dfrac{\ds \int_{\partial \Omega} \theta_{p,t} \, d\sigma}{\ds \int_{\Lambda(p,t)} \theta_{p,t} \, d\mu_{\Lambda(p,t)}} = c_\theta^{-1} \abs{t}^{1-n} \int_{\partial \Omega} \theta_{p,t} \, d\sigma
\end{equation}
where the second inequality uses the fact that we centered $\theta_{p,t}$ at $\mathfrak b^t(p) \in \Lambda(p,t)$, and $c_\theta := \int_{\R^{n-1}} \theta(y) dy$. Note that the Ahlfors regularity of $\sigma$ implies that 
\begin{equation} \label{lambda1}
\lambda(p,t) \approx 1,
\end{equation}
whenever $\bb^t(p)$ is close to $\partial \Omega$ - which is the case when $ d(p) \lesssim \abs{t} \lesssim \ell(Q(\cS))$ - and with constants that depend only on $C_\sigma$ and $n$. 
Finally, we introduce the flat measures 
\begin{equation} \label{defmupr}
\mu_{p,t} := \lambda(p,t) \mu_{\Lambda(p,t)}.
\end{equation}
The flat measures $\mu_{p,t}$ are approximations of the Hausdorff measure on $\Gamma_{\cS}$, and we shall show that the same explicit measures almost minimize the distance from $\sigma$ to flat measures, for the local Wasserstein distances $\dist_Q$ with $Q\in \cS$.

\begin{lemma} \label{Lmupr}
For $Q\in \cS$, $p\in \Pi(\frac32B_Q)$, and $\ell(Q)/4 \leq \abs{t} \leq \ell(Q)/2$,  we have
\begin{equation} \label{Lmupr1}
 \dist_Q(\sigma,\mu_{p,t}) \leq C \alpha_\sigma(Q),
\end{equation}
where $C>0$ depends only on $n$ and $C_\sigma$.
\end{lemma} 

The lemma is not very surprising. The plane $\Lambda(p,t)$ is obtained by locally smoothing $\Gamma_{\cS}$, which is composed of pieces of planes that approximate $\partial \Omega$. 

\medskip

\bp
Thanks to the good approximation properties of the Lipschitz graph $\bb(p)$ that we obtain in Section~\ref{SSLipschitz}, this lemma can be proved similarly as Lemma 5.22 in \cite{DFM3}. 

Let $Q\in \cS$, $p\in \Pi(\frac32B_Q)$, and $t$ with $\ell(Q)/4 \leq \abs{t} \leq \ell(Q)/2$ be fixed. Denote $r=\abs{t}$. By Lemma \ref{Lcorona}, $\alpha_\sigma(Q)\le\epsilon_1$. Since we have chosen $\epsilon_1$ sufficiently small, Lemma \ref{Lalphatobeta} gives that 
\begin{equation}\label{0PQ}
    \sup_{y\in 999\Delta_Q}\dist(y,P_Q)\le C\epsilon_1^{1/n}\ell(Q)\le 10\ell(Q).
\end{equation}
Define a Lipschitz function $\Psi$ by
\[
\Psi(z):=\begin{cases}
\frac14 \qquad\qquad z\in B(x_Q,100\ell(Q)),\\
\frac{1}{3600\ell(Q)}\br{10^3\ell(Q)-\abs{z-x_Q}}_+ \quad\text{otherwise},
\end{cases}
\]
where $(f(z))_+:=\max\set{0,f(z)}$. Then set $f(z)=\Psi(z)\dist(z, P_Q)$. Observe that $\supp f\subset B(x_Q,10^3\ell(Q))$, and that $\abs{\nabla f(z)}\le \Psi(z)+\dist(z, P_Q)\abs{\nabla\Psi}\le 1$, because $\dist(z,P_Q)\le 10\ell(Q)+10^3\ell(Q)$ by \eqref{0PQ}. Hence $f\in Lip(Q)$. By using successively the facts that $f\geq 0$, $\int f \, d\mu_Q = 0$ and \eqref{ftestinalpha}, we have that
\begin{multline}\label{1PQ}
    \int_{\Delta(x_Q,100\ell(Q))}\dist(z, P_Q)d\sigma(z)=4\int_{\Delta(x_Q,100\ell(Q))}\Psi(z)\dist(z,P_Q)d\sigma(z)\\
    \le 4 \int f\, d\sigma =  4\int f(z)(d\sigma-d\mu_Q)\le C\ell(Q)^n\alpha_\sigma(Q).
\end{multline}
Now we estimate the distance from $\Lambda(p,t)$ to $P_Q$. Write 
\[
 \dist(\bb^t(p),P_Q)\le\int_{q\in B(p,r)\cap P}\eta_t(p-q)\dist(\bb(q),P_Q)dq.
\]
Notice that our choice of $p$ and $t$ ensures that
\begin{equation}\label{Bptsubset}
    B(p,r)\cap P\subset\Pi(2B_Q).
\end{equation} So we have that 
\begin{multline}\label{LambdaPQ1}
    \dist(\bb^t(p),P_Q)\le r^{1-n}\norm{\eta}_\infty \int_{q\in\Pi(2B_Q)}\dist(\bb(q),P_Q)dq \\
    \le Cr^{1-n}\norm{\eta}_\infty\ell(Q)^n\alpha_\sigma(Q)
    \le C\ell(Q)\alpha_\sigma(Q),
\end{multline}
where we have used Lemma \ref{Lbbalpha}. We claim that 
\begin{equation}\label{LambdaPQ2}
    \dist(y,P_Q)\le C\alpha_\sigma(Q)\br{\abs{y-\bb^t(p)}+\ell(Q)} \qquad\text{for all }y\in \Lambda(p,t).
\end{equation}
 Let $y=(q,(q-p) \nabla b^t(p) + b^t(p))\in\Lambda(p,t)$ be fixed. Denote by $\Pi_Q^\bot$ the orthogonal projection on the orthogonal complement of $P_Q$.
Then 
\begin{equation}\label{LambdaPQ21}
     \dist(y,P_Q)\le \abs{\Pi_Q^\bot\br{y-\bb^t(p)}}+\dist(\bb^t(p), P_Q).
\end{equation}
Also, $\Pi_Q^\bot(P_Q)$ is a single point $\xi_Q\in\R$. Denote $v:=y-\bb^t(p)=(q-p, (q-p)\nabla b^t(p))$. Let $\hat v^i=\hat v^i(p,t)=\partial_{p_i}\bb^t(p)$, $i=1,2,\dots,n-1$. Then $v=\sum_{i=1}^{n-1}(q_i-p_i)\hat v^i$. We estimate $\abs{\Pi_Q^\bot(\hat v^i)}$. By definition, we write
\begin{multline*}
    \abs{\Pi_Q^\bot(\hat v^i)}=\abs{\Pi_Q^\bot(\partial_{p_i}\bb^t(p))}=\frac{1}{r}\abs{\Pi_Q^\bot((\partial_i\eta)_t*\bb(p))}
    =\frac{1}{r}\abs{\int_{q\in B(p,r)\cap P}(\partial_i\eta)_t(p-q)\Pi_Q^\bot(\bb(q))dq}\\
    =\frac{1}{r}\abs{\int_{q\in B(p,r)\cap P}(\partial_i\eta)_t(p-q)\br{\Pi_Q^\bot(\bb(q))-\xi_Q}dq},
\end{multline*}
where in the last equality we have used that $\int(\partial_i\eta)_t(x) dx=0$. Notice that $\abs{\Pi_Q^\bot(z)-\xi_Q}=\abs{\Pi_Q^\bot(z)-\Pi_Q^\bot(P_Q)}=\dist(z,P_Q)$, so we have that
\[
 \abs{\Pi_Q^\bot(\hat v^i)}\le\frac{1}{r^n}\norm{\partial_i\eta}_\infty\int_{q\in B(p,r)\cap P}\dist(\bb(q),P_Q)dq\le \frac{C}{r^n}\ell(Q)^n\alpha_\sigma(Q)\le C\alpha_\sigma(Q)
\]
by \eqref{Bptsubset} and Lemma \ref{Lbbalpha}. This gives that \begin{equation}\label{PiQbotv}
    \abs{\Pi_Q^\bot(v)}\le C\alpha_\sigma(Q)\abs{v}.
\end{equation} Then the claim \eqref{LambdaPQ2} follows from \eqref{LambdaPQ21} and \eqref{LambdaPQ1}.

Next we compare $c_Q$ (defined in \eqref{defmuQ}) and $\lambda(p,t)$, and claim that 
\begin{equation}\label{lambdapt-cQ}
    \abs{\lambda(p,t)-c_Q}\le C\alpha_\sigma(Q).
\end{equation}
We intend to apply \eqref{ftestinalpha} to the 1-Lipschitz function $|t|\,\theta_{p,t}/\norm{\theta}_{Lip}$. So we need to check that $\supp\theta_{p,t}\subset B(x_Q,10^3\ell(Q))$. By the construction of $\theta_{p,t}$, we have that $\supp\theta_{p,t}\subset B(\bb^t(p),r)$. By Lemma \ref{LALip} and the fact that $\epsilon_0$ has been chosen to be small,
\begin{equation}\label{eqbtb}
     \abs{\bb^t(p)-\bb(p)}=\abs{b^t(p)-b(p)}=\abs{\int\eta_t(q)\br{b(q)-b(p)}dq}\le\norm{\nabla b}_\infty r\le 2\,\epsilon_0\,r<r.
\end{equation}
So $B(\bb^t(p),r)\subset B(\bb(p),2r)$. We show that 
\begin{equation}\label{bbpxQ}
    \abs{\bb(p)-x_Q}\le 10\ell(Q).
\end{equation}
Then the assumption $r\in [\ell(Q)/4,\ell(Q)/2]$ gives that $\supp\theta_{p,t}\subset B(\bb(p),2r)\subset B(x_Q,10^3\ell(Q))$, as desired. 
To see \eqref{bbpxQ}, we recall that $p\in\Pi(\frac32 B_Q)$, and so $\abs{p-\Pi(x_Q)}\le 3\ell(Q)/2$. Let $x\in\pom$ be a point such that $\abs{\bb(p)-x}=\dist(\bb(p),\pom)\lesssim\epsilon_1\ell(Q)$, where the last inequality is due to Lemma \ref{Lbbeps}.  Notice that by the definition \eqref{defdx}, $d(\Pi(x_Q))\le \ell(Q)$. So if $\abs{x-x_Q}\le 10^{-3}d(\Pi(x_Q))$, then  $\abs{x-x_Q}\le 10^{-3}\ell(Q)$, and thus \[
  \abs{\bb(p)-x_Q}\le \abs{\bb(p)-x}+\abs{x-x_Q}\le C\epsilon_1\ell(Q)+10^{-3}\ell(Q)\le 10\ell(Q),
  \]
  as desired. If $\abs{x-x_Q}> 10^{-3}d(\Pi(x_Q))$, then we can apply \eqref{claimPibot} to get that $\abs{\Pi^\bot(x)-\Pi^\bot(x_Q)}\le 2\epsilon_0\abs{\Pi(x)-\Pi(x_Q)}$.  By the triangle inequality, 
  \(
  \abs{\Pi(x)-\Pi(x_Q)}\le \abs{\Pi(x)-p}+\abs{p-\Pi(x_Q)}\le \br{C\epsilon_1+\frac32}\ell(Q)
  \), and so $\abs{\Pi^\bot(x)-\Pi^\bot(x_Q)}\le 2\epsilon_0\br{\frac32+C\epsilon_1}\ell(Q)$. Hence, we still have that \[
  \abs{\bb(p)-x_Q}\le \abs{\bb(p)-x}+\abs{x-x_Q}\le C\epsilon_1\ell(Q)+2\br{C\epsilon_1+3/2}\ell(Q)\le 10\ell(Q),
  \]
which completes the proof of \eqref{bbpxQ}.
We have justified that $t\, \theta_{p,t}/\norm{\theta}_{Lip}\in Lip(Q)$, so we can apply \eqref{ftestinalpha} to this function and obtain that
\begin{equation}\label{AmuAsigma}
\abs{c_\theta r^{n-1}\lambda(p,t)-c_Q\int\theta_{p,t}d\mu_{P_Q}}
=\abs{\int \theta_{p,t}(d\sigma-d\mu_Q)}\le C\ell(Q)^{n-1}\alpha_\sigma(Q).
\end{equation}
We now estimate $A_\mu:=c_Q\int\theta_{p,t}(z)d\mu_{P_Q}(z)$. Denote by $\Pi_Q$ the orthogonal projection from $\Lambda(p,t)$ to $P_Q$; by \eqref{PiQbotv} this is an affine bijection, with a constant Jacobian $J_Q$ that satisfies 
\begin{equation}\label{eqdetJQ}
    \abs{\det(J_Q)-1}\le\abs{\sqrt{1- C\alpha_\sigma(Q)^2}-1}\le C\alpha_\sigma(Q).
\end{equation}
By a change of variables $z=\Pi_Q(y)$, we write
\begin{equation}\label{eqAmu}
    A_\mu=c_Q\det(J_Q)\int_{y\in\Lambda(p,t)}\theta_{p,t}\br{\Pi_Q(y)}d\mu_{\Lambda(p,t)}(y).
\end{equation}
We compare $\int_{y\in\Lambda(p,t)}\theta_{p,t}\br{\Pi_Q(y)}d\mu_{\Lambda(p,t)}(y)$ and $\int_{y\in\Lambda(p,t)}\theta_{p,t}(y)d\mu_{\Lambda(p,t)}(y)$. For $y\in\Lambda(p,t)$, $\abs{\Pi_Q(y)-y}=\dist(y,P_Q)$. So by \eqref{LambdaPQ2}, for $y\in\Lambda(p,t)$
\begin{equation}\label{lambda-cQ1}
  \abs{\theta_{p,t}\br{\Pi_Q(y)}-\theta_{p,t}(y)}\le\norm{\theta}_{Lip}r^{-1}\abs{\Pi_Q(y)-y}\le C\,r^{-1}\alpha_\sigma(Q)\br{\abs{y-\bb^t(p)}+\ell(Q)}.  
\end{equation}
Moreover, the support property of $\theta_{p,t}$ implies that $\abs{\theta_{p,t}\br{\Pi_Q(y)}-\theta_{p,t}(y)}$ is not zero when either $y\in B(\bb^t(p),r)$ or $\Pi_Q(y)\in B(\bb^t(p),r)$. By the triangle inequality and \eqref{LambdaPQ2},
\[
\abs{y-\bb^t(p)}\le\abs{\Pi_Q(y)-\bb^t(p)}+C\alpha_\sigma(Q)\br{\abs{y-\bb^t(p)}+\ell(Q)}.
\]
Since $\alpha_\sigma(Q)\le\epsilon_1$ is sufficiently small, we get that when $\Pi_Q(y)\in B(\bb^t(p),r)$, 
\(
\abs{y-\bb^t(p)}\le 2r+\C\epsilon_1\ell(Q)\le 3\ell(Q)
\). 
So by \eqref{lambda-cQ1} and the fact that $\supp\theta_{p,t}\subset B(\bb^t(p),r)$, 
\begin{equation}\label{eqthetapt}
   \abs{\int_{y\in\Lambda(p,t)}\br{\theta_{p,t}\br{\Pi_Q(y)}-\theta_{p,t}(y)}d\mu_{\Lambda(p,t)}(y)}\le C\,r^{n-1}\alpha_\sigma(Q). 
\end{equation}
Recalling the definition of $c_\theta$, we have obtained that
$\abs{A_\mu-c_Q\det(J_Q)c_\theta\, r^{n-1}}\le C\,c_Q\,r^{n-1}\alpha_\sigma(Q)$.
By the triangle inequality, \eqref{eqdetJQ}, and the fact that $r \approx \ell(Q)$,
\[
\abs{A_\mu-c_Q\,c_\theta\,r^{n-1}}\le C\,c_Q\alpha_\sigma(Q)\ell(Q)^{n-1}.
\]
By this, the triangle inequality and \eqref{AmuAsigma}, 
\begin{equation}\label{lambda-cQ2}
    c_\theta\,r^{n-1}\abs{\lambda(p,t)-c_Q}=\abs{c_\theta r^{n-1}\lambda(p,t)-c_Q\,c_\theta\,r^{n-1}}\le C(1+c_Q)\alpha_\sigma(Q)\ell(Q)^{n-1},
\end{equation}
which implies that $\abs{\lambda(p,t)-c_Q}\le\frac{1}{2}(1+c_Q)$ because $\alpha_\sigma(Q)$ is sufficiently small. But we know $\lambda(p,t)\approx 1$ by \eqref{lambda1}, so $c_Q\approx 1$ and then \eqref{lambda-cQ2} yields the desired estimate \eqref{lambdapt-cQ}.

Finally, we are ready to show that 
\begin{equation}\label{muQmupt}
    \dist_Q(\mu_Q,\mu_{p,t})\le C\alpha_\sigma(Q).
\end{equation}
Let $f\in Lip(Q)$. We have that
\[
\int f(z)d\mu_Q(z)=c_Q\int_{P_Q}f(z)d\mu_{P_Q}(z)=c_Q\det(J_Q)\int_{\Lambda(p,t)}f\br{\Pi_Q(y)}d\mu_{\Lambda(p,t)}(y).
\]
An argument similar to the one for \eqref{eqthetapt} gives that \[\abs{\int_{\Lambda(p,t)}\br{f\br{\Pi_Q(y)}-f(y)}d\mu_{\Lambda(p,t)}(y)}\le C\alpha_\sigma(Q)\ell(Q)^n.\]
So 
\begin{multline*}
    \abs{\int f\,d\mu_Q-\lambda(p,t)\int f\,d\mu_{\Lambda(p,t)}}\le
    \abs{\int f\,d\mu_Q-c_Q\int f\,d\mu_{\Lambda(p,t)}}+\abs{\lambda(p,t)-c_Q}\abs{\int  f\,d\mu_{\Lambda(p,t)}}\\
    \le Cc_Q\det(J_Q)\alpha_\sigma(Q)\ell(Q)^n +c_Q\abs{\det(J_Q)-1}\abs{\int f\,d\mu_{\Lambda(p,t)}}+\abs{\lambda(p,t)-c_Q}\abs{\int  f\,d\mu_{\Lambda(p,t)}}.
\end{multline*}
By \eqref{eqdetJQ}, \eqref{lambdapt-cQ}, and $c_Q\approx 1$, 
\begin{equation}\label{eqfQfLambda}
    \abs{\int f\,d\mu_Q-\lambda(p,t)\int f\,d\mu_{\Lambda(p,t)}}\le C\alpha_\sigma(Q)\ell(Q)^n,
\end{equation}
which proves \eqref{muQmupt}. Now \eqref{Lmupr1} follows from \eqref{muQmupt} and \eqref{ftestinalpha}.
\qed

\ms
We want to use the flat measures $\mu_{p,t}$ to estimate the smooth distance $D_\beta$ introduced in \eqref{IdefD}. But before that, we shall need to introduce 
\begin{equation} \label{defalphak}
\alpha_\sigma(Q,k) := \alpha_\sigma(Q^{(k)}),
\end{equation}
where $Q^{(k)}$ is the unique ancestor of $Q$ such that $\ell(Q^{(k)}) = 2^k\ell(Q)$, and then for $\beta>0$,
\begin{equation} \label{defalphabeta}
\alpha_{\sigma,\beta}(Q) := \sum_{k\in \N} 2^{-k\beta} \alpha_\sigma(Q,k).
\end{equation}
The collection $\{\alpha_{\sigma,\beta}(Q)\}_Q$ is nice, because we have 
\begin{equation} \label{pralphabeta}
\alpha_{\sigma,\beta}(Q^*) \approx \alpha_{\sigma,\beta}(Q)
\end{equation}
 whenever $Q\in \D_{\partial \Omega}$ and $Q^*$ is the parent of $Q$, a property which is not satisfied by the $\{\alpha_\sigma(Q)\}_Q$. And of course the $\alpha_{\sigma,\beta}(Q)$'s still satisfies the Carleson packing condition.
 
 \begin{lemma} \label{LalphabetaCM} Let $\partial \Omega$ be uniformly rectifiable and $\sigma$ be an Ahlfors regular measure satisfying \eqref{defADR}.
There exists a constant $C_{\sigma,\beta}$ that depends only on the constant in \eqref{defUR}, the Ahlfors regular constant $C_\sigma$, and $\beta$ such that, for any $Q_0 \in \D_{\partial \Omega}$,
\begin{equation} \label{pralphabeta2}
\sum_{Q\in \D_{\pom}(Q_0)} |\alpha_{\sigma,\beta}(Q)|^2 \sigma(Q) \leq C_{\sigma,\beta} \sigma(Q_0).
\end{equation}
\end{lemma}

\bp
By Cauchy-Schwarz, 
\[
\abs{\alpha_{\sigma,\beta}(Q)}^2\le \br{\sum_{k\in\N}2^{-\beta k}\alpha_{\sigma}(Q,k)^2}\br{\sum_{k\in\N}2^{-\beta k}}\le C\sum_{k\in\N}2^{-\beta k}\alpha_{\sigma}(Q,k)^2.
\]
Therefore,
\begin{multline}\label{alphabeta1}
    \sum_{Q\in \D_{\pom}(Q_0)} |\alpha_{\sigma,\beta}(Q)|^2 \sigma(Q)\le C\sum_{Q\in\D_{\pom}(Q_0)}\sum_{k\in\N}2^{-\beta k}\alpha_{\sigma}(Q,k)^2\sigma(Q)\\
    =C\sum_{k\in\N}2^{-\beta k}\sum_{Q\in\D_{\om}(Q_0)}\alpha_{\sigma}(Q,k)^2\sigma(Q)
    =: C\sum_{k\in\N}2^{-\beta k} \br{I_1+I_2},
\end{multline}
where $I_1$ is the sum over $Q\in\D_{\sigma}(Q_0)$ such that $2^k\ell(Q)<\ell(Q_0)$, and $I_2$ is the rest. By \eqref{eqalphaQbdd} and Alfhors regularity of $\pom$,
\[
    I_2\le C\sum_{Q\in\D_{\pom}(Q_0), \, \ell(Q)\ge 2^{-k}\ell(Q_0)}\sigma(Q)\le C\sum_{j=0}^k\sum_{\ell(Q)=2^{-j}\ell(Q_0)}\sigma(Q)\le C\sum_{j=0}^k\ell(Q_0)^d=Ck\sigma(Q_0).
\]
For $I_1$, we observe that $\sigma(Q^{(k)})\approx 2^{kd}\sigma(Q)$, and that for each $Q^{(k)}$, it has at most $C2^{kd}$ descendants such that $\ell(Q^{(k)}) = 2^k\ell(Q)$. Therefore, 
\begin{multline*}
    I_1\le C\sum_{Q\in\D_{\pom}(Q_0), \, \ell(Q)< 2^{-k}\ell(Q_0)}\alpha_\sigma(Q^{(k)})^2\sigma(Q^{(k)})2^{-kd}\\
    \le C\sum_{Q^{(k)}\in\D_{\pom}(Q_0), \, \ell(Q^{(k)})
    \le \ell(Q_0)}\alpha_\sigma(Q^{(k)})^2\sigma(Q^{(k)})=C\sum_{Q\in\D_{\pom}(Q_0)}\alpha_\sigma(Q)^2\sigma(Q)\le C\sigma(Q_0)
\end{multline*}
by \eqref{defUR}. Returning to \eqref{alphabeta1}, we have that 
\[
\sum_{Q\in \D_{\pom}(Q_0)} |\alpha_{\sigma,\beta}(Q)|^2 \sigma(Q)\le C\sum_{k\in\N}2^{-\beta k}(k+1)\sigma(Q_0)\le C\sigma(Q_0),
\]
as desired.
\ep

The quantities $\alpha_{\sigma,\beta}$ are convenient, because we can now obtain an analogue of Lemma \ref{Lmupr} where we don't need to pay too much attention on the choices of $p$ and $t$.

\begin{lemma} \label{LmuprG} Let $\beta >0$ and $K\geq 1$.
For $Q\in \cS$, $p\in \Pi( K B_Q)$, and $\ell(Q)/K \leq \abs{t} \leq K\ell(Q)$,  we have
\begin{equation} \label{Lmupr1G}
 \dist_Q(\sigma,\mu_{p,t}) \leq C_{\beta,K} \alpha_{\sigma,\beta}(Q),
\end{equation}
where $C_{\beta,K}>0$ depends only on $n$, $C_\sigma$, $\beta$, and $K$.
\end{lemma} 

\bp
First, we prove that when $p\in \Pi(\frac32 B_Q)$ and $\ell(Q)/K \leq |t| \leq \ell(Q)/2$, we have
\begin{equation} \label{Lmupr2G}
 \dist_Q(\sigma,\mu_{p,t}) \leq C_K \alpha_{\sigma}(Q),
\end{equation}
We set $t_j = 2^{-2-j}\ell(Q)$. We also take a dyadic cube $Q' \subset Q$ such that $\ell(Q')/4 \leq |t| \leq \ell(Q')/2$, and then we pick $p_0\in \Pi(\frac32 B_{Q'})$. By Lemma \ref{Lmupr}, we have that 
\[ \dist_{Q}(\sigma,\mu_{p,t_0}) + \dist_{Q}(\sigma,\mu_{p_0,t_0}) \lesssim \alpha_\sigma(Q),\]
so $\dist_{Q}(\mu_{p,t_0},\mu_{p_0,t_0}) \lesssim \alpha_\sigma(Q)$ too. Consequently, the claim \eqref{Lmupr2G} is reduced to
\begin{equation} \label{Lmupr3G}
\dist_{Q}(\mu_{p_0,t},\mu_{p_0,t_0}) \leq C_K \alpha_\sigma(Q).
\end{equation}
For this latter bound, we decompose
\begin{equation} \label{Lmupr4G}
\dist_{Q}(\mu_{p_0,t},\mu_{p_0,t_0}) \leq \dist_{Q}(\mu_{p_0,t},\mu_{p_0,t_k}) + \sum_{j=0}^{k-1} \dist_{Q}(\mu_{p_0,t_j},\mu_{p_0,t_{j+1}}),
\end{equation}
where $k$ is chosen so that $t_k = \ell(Q')/2$, and $k\leq 1+\log_2(K)$ is bounded by $K$. We look at $\dist_{Q}(\mu_{p_0,t},\mu_{p_0,t_k})$, but since we are dealing with two flat measures that intersects $B_{Q'}$, Lemma A.5  in \cite{FenUR} shows that
\begin{equation} \label{Lmupr5G}
\dist_{Q}(\mu_{p_0,t},\mu_{p_0,t_k}) \lesssim \dist_{Q'} (\mu_{p_0,t},\mu_{p_0,t_k})
\end{equation}
and then Lemma \ref{Lmupr} and the fact that $\ell(Q') \approx_K \ell(Q)$ entail that
\begin{equation} \label{Lmupr6G}
\dist_{Q}(\mu_{p_0,t},\mu_{p_0,t_k}) \lesssim \dist_{Q'} (\mu_{p_0,t},\sigma) + \dist_{Q'} (\sigma,\mu_{p_0,t_k}) \lesssim \alpha_{\sigma}(Q') \leq C_K \alpha_{\sigma}(Q).
\end{equation}
A similar reasoning gives that 
\begin{equation} \label{Lmupr7G}
\dist_{Q}(\mu_{p_0,t_j},\mu_{p_0,t_{j+1}}) \lesssim C_K \alpha_{\sigma}(Q)
\end{equation}
whenever $0\leq j \leq k-1$. The combination of \eqref{Lmupr4G}, \eqref{Lmupr6G}, and \eqref{Lmupr7G} shows the claim \eqref{Lmupr3G} and thus \eqref{Lmupr2G}.

In the general case, we pick the smallest ancestor $Q^*$ of $Q$ such that $p\in \Pi(\frac32B_{Q^*})$ and $|t| \leq \ell(Q^*)/2$, and we apply \eqref{Lmupr2G} to get
\[ \dist_Q(\sigma,\mu_{p,t}) \leq C_K \alpha_{\sigma}(Q^*).\]
The lemma follows then by simply observing that $\alpha_\sigma(Q^*) \lesssim \alpha_{\sigma,\beta}(Q)$.
\ep

We need the constant
\begin{equation} \label{defcbeta}
c_\beta := \int_{\R^{n-1}} (1+|p|^2)^{-\frac{d+\beta}2} dy
\end{equation}
and the unit vector $N_{p,t}$ defined as the vector 
\begin{equation} \label{defNpr}
N_{p,t}(X) := [\nabla \dist(.,\Lambda(p,t))](X)
\end{equation}
which is of course constant on the two connected components of $\R^n \sm \Lambda(p,t)$. We are now ready to compare $D_\beta$ with the distance to $\Lambda(p,t)$.

\begin{lemma}  \label{LestD}
Let $Q\in \cS$, $X\in W_{\Omega}(Q)$, $p\in \Pi(2^5Q)$ and $2^{-5}\ell(Q) \leq |t| \leq 2^5 \ell(Q)$. We have
\begin{equation} \label{estD1}
|D^{-\beta}_\beta(X) - c_\beta \lambda(p,t) \dist(X,\Lambda(p,t))^{-\beta}| \leq C \ell(Q)^{-\beta} \alpha_{\sigma,\beta}(Q).
\end{equation}
and
\begin{equation} \label{estD2}
|\nabla [D^{-\beta}_\beta](X) + \beta c_\beta \lambda(p,t) \dist(X,\Lambda(p,t))^{-\beta-1} N_{p,t}(X)| \\ \leq C \ell(Q)^{-\beta-1} \alpha_{\sigma,\beta+1}(Q),
\end{equation}
where the constant $C>0$ depends only $C_\sigma$ and $\beta$.
\end{lemma}

\bp
Denote $r=\abs{t}$, and $d=n-1$. 
By the definition of $W_\om(Q)$, $\dist(X,\pom)>\ell(Q)/2$ and $X\in 2B_Q$. We show that in addition, 
\begin{equation}\label{eqXball}
    X\in B(\bb^t(p),2^6\ell(Q)),
\end{equation}
and 
\begin{equation}\label{eqXdist}
    \dist(X,\Lambda(p,t)\cup\pom)>\frac{\ell(Q)}{20}.
\end{equation}
Since $X\in 2B_Q$, $\abs{\Pi(X)-p}\le (2^5+2) \ell(Q)$. Then $\abs{\bb(\Pi(X))-\bb(p)}\le (1+2\epsilon_0)(2^5+2)\ell(Q)$ because $\bb$ is the graph of a $2\epsilon_0$-Lipschitz function.  Write 
\[
\abs{X-\bb^t(p)}\le \abs{X-\bb(\Pi(X))}+\abs{\bb(\Pi(X))-\bb(p)}+\abs{\bb(p)-\bb^t(p)},
\]
then use \eqref{claimdPib} and \eqref{eqbtb} to get 
\[\abs{X-\bb^t(p)}\le (1+2\epsilon_0)\br{\delta(X)+(2^5+2)\ell(Q)}+2\epsilon_0r\le 2^6\ell(Q),\]
 and thus \eqref{eqXball} follows. To see \eqref{eqXdist}, we only need to show that 
$\dist(X,\Lambda(p,t))>\frac{\ell(Q)}{20}$. Notice that $(\nabla b^t(p),-1)$ is a normal vector of the plane $\Lambda(p,t)$, and that $\bb^t(p)\in\Lambda(p,t)$. So 
\begin{multline}\label{distXLambda}
    \dist(X,\Lambda(p,t))=\frac{\abs{\br{X-\bb^t(p)}\cdot\br{\nabla b^t(p),-1}}}{\abs{(\nabla b^t(p),1)}}=\frac{\abs{(\Pi(X)-p)\cdot\nabla b^t(p)+\br{b^t(p)-\Pi^\bot(X)}}}{\sqrt{\abs{\nabla b^t(p)}^2+1}}\\
    \ge \frac12\br{\abs{\Pi^\bot(X)-b^t(p)}-\abs{(\Pi(X)-p)\cdot\nabla b^t(p)}}\ge\frac12\abs{\Pi^\bot(X)-b^t(p)}-C\,2^5\ell(Q)\epsilon_0,
\end{multline}
by $\norm{\nabla b^t}_\infty\le C\epsilon_0$ (see \eqref{brbdd}). We have that $\abs{b(\Pi(X))-b(p)}\le 2\epsilon_0\abs{\Pi(X)-p}\le 2^6\epsilon_0$, and that 
\(
\abs{\Pi^\bot(X)-b(\Pi(X))}\ge\dist(X,\Gamma_S)\ge\frac{\delta(X)}{1+3\epsilon_0}\ge \frac{\ell(Q)}{2(1+3\epsilon_0)}
\)
by \eqref{claimdGPib}. So 
\[
\abs{\Pi^\bot(X)-b^t(p)}\ge\abs{\Pi^\bot(X)-b(\Pi(X))}-\abs{b(\Pi(X))-b(p)}-\abs{b^t(p)-b(p)}\ge \frac{\ell(Q)}{5}
\]
by \eqref{eqbtb}. Then $\dist(X,\Lambda(p,t))>\frac{\ell(Q)}{20}$ follows from this and \eqref{distXLambda}.

Now we prove \eqref{estD1}. We intend to cut the integral $D_\beta^{-\beta}=\int_{\pom}\abs{X-y}^{-d-\beta}d\sigma(y)$ into pieces. So we introduce a cut-off function $\theta_0\in C_c^\infty(B(0,r/2))$, which is radial, $\1_{B(0,r/4)}\le\theta_0\le \1_{B(0,r/2)}$, and $\abs{\nabla\theta_0}\le 2r$. Then we set $\theta_k(y):=\theta_0(2^{-k}y)-\theta_0(-2^{-k+1}y)$ for $k\ge 1$ and $y\in\Rn$, and define $\wt\theta_k(y)=\theta_k(y-\bb^t(p))$ for $k\in\N$. Denote $B_k=B(\bb^t(p),2^{k-1}r)$. We have that $\supp\wt\theta_0\subset B_0$, $\supp\wt\theta_k\subset B_k\setminus B_{k-2}$ for $k\ge 1$, and that 
\[
\sum_{k\in\N}\wt\theta_k=1.
\]
Now we can write 
\[D_\beta(X)^{-\beta}=\sum_{k\in\N}\int_{\pom}\abs{X-y}^{-d-\beta}\wt\theta_k(y)d\sigma(y)=:\sum_{k\in\N}\int_{\pom} f_k(y)d\sigma(y),\]
with $f_k(y)=\abs{X-y}^{-d-\beta}\wt\theta_k(y)$. We intend to compare $\int f_k(y)d\sigma(y)$ and $\int f_k(y)d\mu_{p,t}(y)$. Both integrals are well-defined because of \eqref{eqXdist}. Observe that 
\begin{multline*}
    \sum_{k\in\N}\int f_k(y)d\mu_{p,t}(y)=\lambda(p,t)\int_{\Lambda(p,t)}\abs{X-y}^{-d-\beta}d\mu_{\Lambda(p,t)}(y)\\
    =\lambda(p,t)\int_{\R^d}\br{\dist(X,\Lambda(p,t))^2+\abs{y}^2}^{-(d+\beta)/2}dy=\lambda(p,t)c_\beta\dist(X,\Lambda(p,t))^{-\beta}
\end{multline*}
by a change of variables.  So
\begin{equation}\label{Dbeta}
     D^{-\beta}_\beta(X) - c_\beta \lambda(p,t)\dist(X,\Lambda(p,t))^{-\beta}=\sum_{k\in\N}\int f_k\,(d\sigma-d\mu_{p,t}).
\end{equation}
We are interested in the Lipschitz properties of $f_k$ because we intend to use Wasserstein distances. 
We claim that 
\begin{equation}\label{X-ylwbd}
    \abs{X-y}\ge c2^kr \quad\text{when }y\in\pom\cup\Lambda(p,t) \text{ is such that } \wt\theta_k(y)\neq0,
\end{equation}
where $c=10^{-1}2^{-21}$.
In fact, by \eqref{eqXball} and the support properties of $\wt\theta_k$, if $k\ge 15$, then 
\[
\abs{X-y}\ge 2^{k-3}r-2^6\ell(Q)\ge (2^{k-3}-2^{11})r\ge 2^{k-4}r \quad\text{for }y\in\supp\wt\theta_k.
\]
If $0\le k<15$, then by \eqref{eqXdist}, for $y\in\pom\cup\Lambda(p,t)$,
\[
\abs{X-y}\ge\dist(X,\pom\cup\Lambda(p,t))\ge\frac{\ell(Q)}{20}\ge \frac{2^{-6}r}{10}\ge \frac{2^{-21}}{10}2^kr.
\]
So \eqref{X-ylwbd} is justified. But $f_k$ is not a Lipschitz function in $\Rn$ because $y$ can get arbitrary close to $X$ when $k$ is small. Set
\[
\wt f_k(y):=\max\set{\abs{X-y},c\,2^kr}^{-d-\beta}\wt\theta_k(y).
\]
Then by \eqref{X-ylwbd}, $\wt f_k(y)=f_k(y)$ for $y\in\pom\cup\Lambda(p,t)$, and therefore, 
\begin{equation}\label{fkwtfk}
    \int f_k\,(d\sigma-d\mu_{p,t})=\int \wt f_k\,(d\sigma-d\mu_{p,t}).
\end{equation}
The good thing about $\wt f_k$ is that it is Lipschitz. A direct computation shows that 
    $\norm{\wt f_k}_\infty\le C\br{2^kr}^{-d-\beta}$, and $\norm{\nabla \wt f_k}_{\infty}\le C(2^kr)^{-d-\beta-1}$. Moreover, $\wt f_k$ is supported on $B(\bb^t(p),2^{k-1}r)$, which is contained in $B(x_{Q^{(k)}},10^3\ell(Q^{(k)}))$. To see this, one can use \eqref{eqbtb}, \eqref{bbpxQ}, and $\abs{x_Q-x_{Q^{(k)}}}\le 2^{k-1}\ell(Q)$ to get that 
\[
\abs{\bb^t(p)-x_{Q^{(k)}}}\le \br{2\epsilon_0+2^{k-1}}r+10\,2^5\ell(Q)+2^{k-1}\ell(Q)\le 2^5(2^k+11)\ell(Q)\le 10^32^k\ell(Q).
\]
Write 
\begin{multline*}
    \int \wt f_k\,(d\sigma-d\mu_{p,t})
    =\int \wt f_k\,(d\sigma-d\mu_{Q^{(k)}})+\int \wt f_k\,(d\mu_Q-d\mu_{p,t})
    +\sum_{j=1}^k\int \wt f_k\,(d\mu_{Q^{(j)}}-d\mu_{Q^{(j-1)}})\\
    =:I+II+\sum_{j=1}^kIII_j.
\end{multline*}
By the definition \eqref{defmuQ} of $\mu_{Q^{(k)}}$ and properties of $\wt f_k$, 
\(
\abs{I}\le C\br{2^kr}^{-\beta}\alpha_\sigma(Q,k)
\). 
 We then have $|II| \leq \br{2^kr}^{-\beta} \dist_{Q^{(k)}} (\mu_Q , \mu_{p,t})$, but because we are looking at the Wasserstein distance between two flat measures whose supports intersect $10B_Q$, Lemma A.5 in \cite{FenUR} shows that 
\[\dist_{Q^{(k)}} (\mu_Q , \mu_{p,t}) \lesssim \dist_{Q} (\mu_Q , \mu_{p,t})\]
and thus
\[ |II| \lesssim \br{2^kr}^{-\beta} \dist_{Q} (\mu_Q , \mu_{p,t}) \leq \br{2^kr}^{-\beta} \Big( \dist_{Q} (\mu_Q, \sigma) + \dist_Q( \sigma , \mu_{p,t}) \Big) \lesssim \br{2^kr}^{-\beta} \alpha_{\sigma,\beta}(Q)\]
by Lemma \ref{LmuprG}. The terms $III_j$ can be bounded by a Wasserstein distance between planes, and similarly to $II$, we get
\[ |III_j| \lesssim \br{2^kr}^{-\beta} \dist_{Q^{(k)}} (\mu_{Q^{(j)}} , \mu_{Q^{(j-1)}}) \lesssim \br{2^kr}^{-\beta} \dist_{Q^{(j)}} (\mu_{Q^{(j)}} , \mu_{Q^{(j-1)}}) \lesssim \br{2^kr}^{-\beta} \alpha_\sigma(Q,j).\]
Altogether, we obtain that 
\[
\abs{ \int \wt f_k\,(d\sigma-d\mu_{p,t})}\le C\br{2^kr}^{-\beta} \left( \alpha_{\sigma,\beta}(Q) + \sum_{j=0}^k\alpha_\sigma(Q,j) \right).
\]
Then by \eqref{fkwtfk} and \eqref{Dbeta},
\begin{multline*}
     \abs{D^{-\beta}_\beta(X) - c_\beta \lambda(p,t)\dist(X,\Lambda(p,t))^{-\beta}}\le C\sum_{k\in\N}\br{2^kr}^{-\beta}\left( \alpha_{\sigma,\beta}(Q) + \sum_{j=0}^k\alpha_\sigma(Q,j) \right) \\
    \le C\ell(Q)^{-\beta}\alpha_{\sigma,\beta}(Q),
\end{multline*}
which is \eqref{estD1}.

We claim that \eqref{estD2} can be established similarly to \eqref{estD1} as long as one expresses the left-hand side of \eqref{estD2} appropriately. 
A direct computation shows that 
\[\nabla(D_\beta^{-\beta})(X)=-(d+\beta)\int\abs{X-y}^{-d-\beta-2}(X-y)d\sigma(y).\] 
On the other hand,
\begin{multline*}
    \int\abs{X-y}^{-d-\beta-2}(X-y)d\mu_{\Lambda(p,t)}(y)= N_{p,t}(X)\int\abs{X-y}^{-d-\beta-2}(X-y)\cdot N_{p,t}(X)d\mu_{\Lambda(p,t)}(y)\\
    =N_{p,t}(X)\int\abs{X-y}^{-d-\beta-2}\dist(X,\Lambda(p,t))d\mu_{\Lambda(p,t)}(y)=c_{\beta+2}\dist(X,\Lambda(p,t))^{-\beta-1}N_{p,t}(X).
\end{multline*}
By \cite{FenUR} (3.30), $(\beta+d)c_{\beta+2}=\beta c_\beta$ for all $\beta>0$. Hence
\begin{multline*}
    \nabla [D^{-\beta}_\beta](X) + \beta c_\beta \lambda(p,t) \dist(X,\Lambda(p,t))^{-\beta-1} N_{p,t}(X)\\
    =-(d+\beta)\int\abs{X-y}^{-d-\beta-2}(X-y)\br{d\sigma(y)-d\mu_{p,t}(y)}.
\end{multline*}
Now we set $f_k'(y)=\abs{X-y}^{-d-\beta-2}(X-y)$. Using \eqref{eqXball} and \eqref{eqXdist}, we can see that $f_k'$ is Lipschitz on $\pom\cup\Lambda(p,t)$. Then we can play with measures as before to obtain \eqref{estD2}. 
\ep

\begin{corollary} \label{CestD}
Let $Q\in \cS$, $X\in W_{\Omega}(Q)$, $p\in \Pi(2^5Q)$ and $2^{-5}\ell(Q) \leq |t| \leq 2^5 \ell(Q)$. We have
\begin{equation} \label{estD3}
\left|\dfrac{\nabla D_\beta(X)}{D_\beta(X)} - \dfrac{N_{p,t}(X)}{\dist(X,\Lambda(p,t))}\right| \leq C \ell(Q)^{-1} \alpha_{\sigma,\beta}(Q).
\end{equation}
where the constant $C>0$ still depends only $C_\sigma$ and $\beta$.
\end{corollary}

\bp
To lighten the notation, we denote by $\mathcal O_{CM}$ any quantity such that $$|\mathcal O_{CM}| \leq C \alpha_{\sigma,\beta}(Q)$$ for some constant $C$. Then by \eqref{estD2},
\begin{multline*}
    \frac{\nabla D_\beta(X)}{D_\beta(X)}  = - \frac1{\beta} \frac{\nabla[D_\beta^{-\beta}](X)}{D_\beta^{-\beta}(X)} \\
 =  - \frac1{\beta}\br{ \frac{-\beta c_\beta \lambda(p,t) \dist(X,\Lambda(p,t))^{-\beta-1}N_{p,t}(X)}{D_{\beta}^{-\beta}(X)} + \frac{\ell(Q)^{-\beta-1}\mathcal O_{CM}}{D_{\beta}^{-\beta}(X)}}. 
\end{multline*}
Using $D_\beta(X)\approx\delta(X)\approx\ell(Q)$ and \eqref{estD1}, we can further write the above as 
\[
    \frac{\nabla D_\beta(X)}{D_\beta(X)} = \dfrac{N_{p,t}(X)}{\dist(X,\Lambda(p,t))} + \ell(Q)^{-1} \mathcal O_{CM},
\]
which implies the corollary.
\ep

\section{The bi-Lipschitz change of variable $\rho_{\cS}$}

\label{Srho}

The results in this section are similar, identical, or often even easier than the ones found in Sections 2, 3, and 4 of \cite{DFM3}. Many proofs will only be sketched and we will refer to the corresponding result in \cite{DFM3} for details.

As in the previous sections, we take $0<\epsilon_0 \ll \epsilon_1 \ll 1$ and we use Lemma \ref{Lcorona} with such $\epsilon_0,\epsilon_1$ to obtain a collection $\mathfrak S$ of coherent regimes. We take then $\cS$ that either belongs to $\mathfrak S$, or is a coherent regime included in an element of $\mathfrak S$.
We keep the notations introduced in Sections \ref{SUR}, \ref{SWhitney}, and \ref{SDbeta}. 

\subsection{Construction of $\rho_{\cS}$}

In this section, the gradients are column vectors. The other notation is fairly transparent. A hyperplane $P$ is equipped with an orthonormal basis and $\nabla_p$ correspond to the vector of the derivatives in each coordinate of $p$ in this basis; $\partial_t$ or $\partial_s$ are the derivatives with respect to $t$ or $s$, that are always explicitly written; $\nabla_{p,t}$ or $\nabla_{p,s}$ are the gradients in $\R^n$ seen as $P \times P^\bot$.

\begin{lemma} \label{LbrCM}
The quantities $\nabla_{p,t} b^t$ and $t \nabla_{p,t} \nabla_p b^t$ are bounded, that is, for any $t \neq0$ and any $p_0\in P$,
\begin{equation} \label{brbdd}
|\nabla_{p,t} b^t| + |t \nabla_{p,t} \nabla_p b^t| \leq C\epsilon_0.
\end{equation}
In addition, $|\partial_{t} b^t| + |t \nabla_{p,t} \nabla_p b^t| \in CM_{P\times (P^\bot \sm \{0\})}$, that is, for any $r>0$ and any $p_0\in P$,
\begin{equation} \label{brCM}
\iint_{B(p_0,r)} \Big( |\partial_t b^t(p)|^2 + |t \nabla_{p,t} \nabla_p b^t(p)|^2  \Big) \frac{dt}{t} \, dp \leq C \epsilon_0^2 r^{n-1}.
\end{equation}
In both cases, the constant $C>0$ depends only on $n$ (and $\eta$).
\end{lemma}

\bp
The result is well known and fairly easy. The boundedness is proven in Lemma 3.17 of \cite{DFM3}, while the Carleson bound is established Lemma 4.11 in \cite{DFM3} (which is itself a simple application of the Littlewood-Paley theory found in \cite[Section I.6.3]{Stein93} to the bounded function $\nabla b$).
\ep

Observe that the convention that we established shows that $(\nabla_p b^t(p))^T$ is $n-1$ dimensional horizontal vector. We define the map $\rho: P\times P^\bot \to P\times P^\bot$ as
\begin{equation} \label{defrhoS}
\rho_{\cS}(p,t) := (p - t(\nabla b^t(p))^T, t + b^t(p))
\end{equation}
if $t\neq 0$ and $\rho_{\cS}(p,0) = \bb(p)$.
Because the codimension of our boundary in 1 in our paper, contrary to \cite{DFM3} which stands in the context of domains with higher codimensional boundaries, our map is way easier than the one found in \cite{DFM3}. However, the present mapping is still different from the one found in \cite{KePiDrift}, and has the same weak and strong features as the change of variable in \cite{DFM3}. Note that the $i^{th}$ coordinate of $\rho_{\cS}$, $1\leq i \leq n-1$, is
\begin{equation} \label{defrhoSi}
\rho_{\cS}^i(p,t) := p_i - t\partial_{p_i} b^t(p).
\end{equation}
Note that $\rho_{\cS}$ is continuous on $P\times P^\bot = \R^n$, because both $t\nabla b^t$ and $b^t-b$ converges (uniformly in $p \in P$) to $0$ as $t\to 0$. The map $\rho_{\cS}$ is $C^\infty$ on $\R^n \setminus P$, and we compute the Jacobian $\Jac$ of $\rho_{\cS}$ which is
\begin{equation} \label{defJacS}
\Jac(p,t) = \begin{pmatrix} I - t\partial_{p_i} \partial_{p_j} b^t(p) & \partial_{p_i} b^t(p) \\ - t \partial_t \partial_{p_j} b^t(p) - \partial_{p_j} b^t(p) & 1 + \partial_t b^t(p) \end{pmatrix},
\end{equation}
where $i$ and $j$ refers to respectively the line and the column of the matrix. We define the approximation of the Jacobian $\Jac$ as
\begin{equation} \label{defJS}
J = \begin{pmatrix} I & \partial_{p_i} b^t(p) \\  - \partial_{p_j} b^t(p) & 1 \end{pmatrix} = \begin{pmatrix} I & \nabla_p b^t(p) \\  - (\nabla_p b^t(p))^T & 1 \end{pmatrix}.
\end{equation}

\begin{lemma} \label{LestonJ}
We have the following pointwise bounds:
\begin{enumerate}[(i)]
\item $\ds \|J-I\| \lesssim |\nabla_p b^t|  \lesssim \epsilon_0$, \smallskip
\item $\ds \|\Jac - J\| \lesssim |\partial_t b^t| + |t\nabla_{p,t} \nabla_p b^t| \lesssim \epsilon_0$, \smallskip
\item $\ds |\det(J)-1| \lesssim |\nabla_p b^t|  \lesssim \epsilon_0$, \smallskip
\item $\ds |\det(\Jac) - \det(J)| \lesssim |\partial_t b^t| + |t\nabla_{p,t} \nabla_p b^t|$, \smallskip
\item $\ds \|(\Jac)^{-1} - J^{-1}\| \lesssim |\partial_t b| + |t\nabla_{p,t} \nabla_p b^t|$, \smallskip
\item $\ds |\nabla_{p,t} \det(J)| + \||\nabla_{p,t} J^{-1}|\| \lesssim |\nabla_{p,t} \nabla_p b^t|$, \smallskip
\end{enumerate}
In each estimate, the constants depends only on $n$ and $\eta$.
\end{lemma}

\bp Only a rapid proof is provided, and details are carried out in the proof of Lemmas 3.26, 4.12, 4.13, and 4.15 in \cite{DFM3}.  

\medskip

The items (i) and (ii) are direct consequences of \eqref{brbdd} and the definitions of $J$ and $\Jac$.

For items $(iii)$ and $(iv)$, we use the fact that the determinant is the sum of products of coefficients of the matrix. More precisely, the Leibniz formula states that
\begin{equation} \label{defdetM}
\det(M) := \sum_{\sigma \in S_n} \sgn(\sigma) \prod_{i=1}^n M_{i,\sigma(i)},
\end{equation}
where $S_n$ is the sets of permutations of $\{1,\dots,n\}$ and $\sgn$ is the signature. So the difference between the determinant of two matrices $M_1$ and $M_2$ is the sum of products of coefficients of $M_1$ and $M_2-M_1$, and each product contains at least one coefficient of $M_2-M_1$. With this observation, $(iii)$ and $(iv)$ follow from $(i)$ and $(ii)$.

The items $(iii)$ and $(iv)$ shows that both $\det(J)$ and $\det(\Jac)$ are close to $1$ - say in $(1/2,2)$ - as long as $\epsilon_0$ is small enough. This implies that
\begin{equation} \label{diff1/detJ}
\left|\frac1{\det(\Jac)} - \frac1{\det(J)}\right| = \left|\frac{\det(J) - \det(\Jac)}{\det(\Jac)\det(J)} \right| \lesssim |\partial_t b| + |t\nabla_{p,t} \nabla_p b^t|
\end{equation} 
by $(iv)$. Cramer's rule states that the coefficients of $M^{-1}$ is the quotient of a linear combination of product of coefficients of $M$ over $\det(M)$. By using Cramer's rule to $\Jac$ and $J$, \eqref{diff1/detJ}, and $(ii)$, we obtain $(v)$. 

Finally, the bound on $\nabla \det(J)$ and $\nabla J^{-1}$ are obtained by taking the gradient respectively in \eqref{defdetM} and in Cramer's rule.
\ep

\begin{lemma} \label{LprrhoS}
For any $p\in P$ and $t\in P^\bot \sm \{0\}$, we have
\begin{equation} \label{prrhoS1}
(1-C\epsilon_0) |t| \leq \dist(\rho_{\cS}(p,t),\Gamma_{\cS}) \leq |\rho_{\cS}(p,t) - \bb(p)| \leq (1+C\epsilon_0) |t|
\end{equation}
and
\begin{equation} \label{prrhoS4}
|\rho_{\cS}(p,t) - \bb(p) - (0,t)| \leq C\epsilon_0 |t|,
\end{equation}
where $C>0$ depends only on $n$ (and $\eta$).
\end{lemma}

\bp The lemma is an analogue of Lemma 3.40 in \cite{DFM3}. But since the lemma is key to understand why $\rho_\cS$ is a bi-Lipschitz change of variable, and since it is much easier in our case, we prove it carefully.

By definition of $\rho_{\cS}$, 
\[\rho_{\cS}(p,t) - \bb(p) - (0,t) = (-t(\nabla b^t(p))^T, b^t(p) - b(p)).\]
So the mean value theorem applied to the continuous function $t \mapsto b^t(p)$ [recall that $b$ is Lipschitz and $b^t$ is the convolution of $b$ with a mollifier, so we even have a uniform convergence of $b^t$ to $b$] entails that
\begin{equation} \label{prrhoS3}
\begin{split}
|\rho_{\cS}(p,t) - \bb(p) - (0,t)| & \leq |t\nabla b^t(p)| + |b^t(p) - b(p)| \\
& \leq |t\nabla b^t(p)| + |t| \sup_{s\in (0,|t|)} |\partial_s b^s(p)| \lesssim \epsilon_0|t|
\end{split}
\end{equation}
by \eqref{brbdd}. Therefore, \eqref{prrhoS4} is proven and we have 
\begin{equation} \label{prrhoS2}
|\rho_{\cS}(p,t) - \bb(p)| \leq (1+C\epsilon_0) |t|,
\end{equation}
is the upper bound in \eqref{prrhoS1}. The middle bound of \eqref{prrhoS1} is immediate, since $\bb(p) \in \Gamma_{\cS}$. It remains thus to prove the lower bound in \eqref{prrhoS1}. 
Let $q\in P$ be such that $|\rho_{\cS}(p,t) - \bb(q)| = \dist(\rho_{\cS}(p,t), \Gamma_{\cS})$. We know that
\[|\bb(q) - \bb(p)| \leq |\bb(q) - \rho_{\cS}(p,t)| + |\rho_{\cS}(p,t) - \bb(p)| \leq 2|\rho_{\cS}(p,t) - \bb(p)| \leq 3|t|,\]
if $\epsilon_0\ll 1$ is small enough, hence $|q-p| \leq 3|t|$ too. So 
\[\begin{split}
\dist(\rho_{\cS}(p,t), \Gamma_{\cS}) & =  |\rho_{\cS}(p,t) - \bb(q)| \geq |\bb(p) - \bb(q) + (0,t)| - |\rho_{\cS}(p,t) - \bb(p) - (0,t)| \\
& \geq |b(p) - b(q) + t| - C\epsilon_0|t|
\end{split}\]
by \eqref{prrhoS4}. But by Lemma \ref{LALip}, the function $b$ is $2\epsilon_0$-Lipschitz, so we can continue with
\[\dist(\rho_{\cS}(p,t), \Gamma_{\cS}) \geq (1-C{\epsilon_0})|t| - |b(p)-b(q)| \geq (1-C{\epsilon_0}) |t| - 2\epsilon_0 |p-q| \geq (1 - C'{\epsilon_0}) |t|. \]
The lemma follows.
\ep

\begin{lemma} 
The map $\rho_{\cS}$ is a bi-Lipschitz change of variable that maps $P$ to $\Gamma_{\cS}$.
\end{lemma}

\bp See Theorem 3.53 in \cite{DFM3} for more details. We shall show that $\rho_{\cS}$ is a bi-Lipschitz change of variable from $P \times (0,\infty)$ to
\[\Omega^+_{\cS} := \{(p,t) \in P \times P^\bot, \, t>b(p)\}\]
and a similar argument also give that $\rho_{\cS}$ is a bi-Lipschitz change of variable from $P \times (-\infty,0)$ to
\[\Omega^-_{\cS} := \{(p,t) \in P \times P^\bot, \, t<b(p)\}.\]
The lemma follows because we know that $\rho_{\cS}$ is continuous on $P \times P^\bot$.

\medskip

First, we know by the lower bound in \eqref{prrhoS1} that the range of $\rho_{\cS}(P\times (0,\infty))$ never intersects $\Gamma_{\cS}$, so since $\rho_{\cS}$ is connected, it means that $\rho_{\cS}(P\times (0,\infty))$ is included in either $\Omega^+_{\cS}$ or $\Omega^-_{\cS}$. A quick analysis of $\rho_{\cS}$, for instance \eqref{prrhoS4}, shows that $\rho_{\cS}(P\times (0,\infty)) \subset \Omega^+_{\cS}$.

At any point $(p,t) \in P\times (0,\infty)$ the Jacobian of $\rho_{\cS}$ is close to the identity, as shown by $(i)$ and $(ii)$ of Lemma \ref{LestonJ}. So $\rho_{\cS}$ is a local diffeomorphism, and the inversion function theorem shows that there exists a neighborhood $V_{p,t} \subset P\times (0,\infty)$ of $(p,t)$ such that $\rho_{\cS}$ is a bijection between $V_{p,t}$ and its range $\rho(V_{p,t})$, which is a neighborhood of $\rho_{\cS}(p,t)$. Since the Jacobian is uniformly close to the identity, all the $\rho_{\cS}:\, V_{p,t} \mapsto \rho(V_{p,t})$ are bi-Lipschitz maps with uniform Lipschitz constant.

\medskip

If $z\in \Omega^+_{\cS}$, we define the degree of the map $\rho_{\cS}$ as
\[N(z) := \text{``number of points $(p,t) \in P\times (0,\infty)$ such that $\rho(p,t) = z$''} \in \N \cup \{+\infty\}.\]
We want to prove that $N(z)$ is constantly equal to 1. If this is true, then the lemma is proven and we can construct the inverse $\rho^{-1}$ locally by inversing the appropriate bijection $\rho_{\cS}:\, V_{p,t} \mapsto \rho(V_{p,t})$.

We already know that the number of points that satisfies $\rho(p,t) = z$ is countable, because we can cover $P\times (0,\infty)$ by a countable union of the neighborhoods $V_{p,t}$ introduced before. Moreover, if $N(z) \geq v >0$, then we can find $v$ points $(p_i,t_i)\in P\times (0,\infty)$ such that $\rho_{\cS}(p_i,t_i) = z$ and so $v$ disjoint neighborhoods $V_{p_i,t_i}$ of $(p_i,t_i)$. Consequently, each point $z'$ in the neighborbood $\bigcap_i \rho_{\cS}(V_{p_i,t_i})$ of $z$ satisfies $N(z') \geq v$. This proves that $N$ is constant on any connected component, that is
\[ \text{$N$ is constant on $\Omega^+_{\cS}$}.\]
It remains to prove that $N(z_0) = 1$ for one point $z_0$ in $\Omega^+_{\cS}$. Take $p_0$ far from the support of $b$, for instance $\dist(p_0,\Pi(Q(\cS)) \geq 99 \ell(Q(\cS))$ and $t_0 = \ell(Q(\cS))$. In this case, we have $\rho_{\cS}(p_0,t_0) = (p_0,t_0)$ and $\dist(\rho_{\cS}(p_0,t_0),\Gamma_{\cS}) = t_0$. Let $(p_1,t_1) \in P \times (0,\infty)$ be such that $\rho_{\cS}(p_1,t_1) = (p_0,t_0)$, the bound \ref{prrhoS1} entails that $|t_1-t_0| \leq C\epsilon_1 |t_0| \leq \ell(Q(\cS))$ and \eqref{prrhoS1} implies that $|p_1 - p_0| \leq C\epsilon_0 |t_1| \leq \ell(Q(\cS))$. Those conditions force $p_1$ to stay far away from the support of $b$, which implies that $\rho_{\cS}(p_1,t_1) = (p_1,t_1) = (p_0,t_0)$. We just proved that $N(p_0,t_0) = 1$, as desired.
\ep

\subsection{Properties of the operator $L_\cS$}

\begin{lemma} \label{LdefAS}
Let $L = - \div \A \nabla$ is a uniformly elliptic operator satisfying \eqref{defelliptic} and \eqref{defbounded} on $\Omega$. Construct on $\rho_{\cS}^{-1}(\Omega)$ the operator $L_{\cS} =-\div \A_{\cS} \nabla$ with
\begin{equation} \label{defAS}
\A_{\cS}(p,t) := \det(\Jac(p,t)) \Jac^{-T}(p,t) \A(\rho_{\cS}(p,t)) \Jac^{-1}(p,t) \qquad \text{ for } (p,t) \in \rho_{\cS}^{-1}(\Omega).
\end{equation}
Then $L_{\cS}$ is the conjugate operator of $L$ by $\rho_{\cS}$, that is, $u\circ \rho_{\cS}$ is a weak solution to $L_{\cS}(u\circ \rho_{\cS}) = 0$ in $\rho_{\cS}^{-1}(\Omega)$ if and only if $u$ is a weak solution to $Lu = 0$ in $\Omega$.
\end{lemma}

\bp The maps $\rho_{\cS}$ is a bi-Lipschitz change of variable on $\R^n = P \times P^\bot$, so the construction \eqref{defAS} properly define a matrix of coefficients in $L^\infty(\rho_{\cS}^{-1}(\Omega))$. 

Let $u$ be a weak solution to $Lu=0$ in $\Omega$. Then, for any $\varphi \in C^\infty_0(\rho_{\cS}^{-1}(\Omega))$, we have
\[\begin{split}
\iint_{\R^n} \A_{\cS} \nabla (u\circ \rho_{\cS}) \cdot \nabla \varphi \, dp\, dt & 
= \iint_{\R^n} \det(\Jac) \Jac^{-T} (\A\circ \rho_{\cS}) \Jac^{-1}  \nabla (u\circ \rho_{\cS}) \cdot \nabla \varphi \, dp\, dt  \\
& = \iint_{\R^n} \det(\Jac) (\A\circ \rho_{\cS}) \Jac^{-1} \nabla (u\circ \rho_{\cS}) \cdot \Jac^{-1} \nabla \varphi \, dp\, dt \\
& = \iint_{\R^n} \det(\Jac) (\A\circ \rho_{\cS}) (\nabla u \circ \rho_{\cS}) \cdot (\nabla [\varphi \circ  \rho_{\cS}^{-1}] \circ \rho_{\cS})\, dp\, dt 
\end{split}\]
because $\nabla (f\circ \rho_{\cS})$ is equal to the matrix multiplication $\Jac (\nabla f \circ \rho_{\cS})$ by definition of the Jacobian. Recall that $\det(\Jac) >0$, so doing the change of variable $X=\rho_{\cS}(p,t)$ gives
\begin{equation} \label{defAS1}
\iint_{\R^n} \A_{\cS} \nabla (u\circ \rho_{\cS}) \cdot \nabla \varphi \, dp\, dt = \iint_{\R^n}  \A \nabla u \cdot \nabla [\varphi \circ  \rho_{\cS}^{-1}] \, dX.
\end{equation}
The function $\varphi \circ  \rho_{\cS}^{-1}$ may not be smooth anymore, but is still compactly supported in $\Omega$ and in $W^{1,\infty}(\Omega) \subset W^{1,2}_{loc}(\Omega)$, so $\varphi \circ  \rho_{\cS}^{-1}$ is a valid test function for the weak solution $u$, and so the right-hand side of \eqref{defAS1} is 0. We conclude that 
\[\iint_{\R^n} \A_{\cS} \nabla (u\circ \rho_{\cS}) \cdot \nabla \varphi \, dp\, dt = 0\]
for any $\varphi \in C^\infty_0(\rho_{\cS}^{-1}(\Omega))$, hence $u\circ \rho_{\cS}$ is a weak solution to $L_{\cS}(u\circ \rho_{\cS}) = 0$ in $\rho_{\cS}^{-1}(\Omega)$. 

The same reasoning shows that $u$ is a weak solution to $Lu=0$ in $\Omega$ whenever  $u\circ \rho_{\cS}$ is a weak solution to $L_{\cS}(u\circ \rho_{\cS}) = 0$ in $\rho_{\cS}^{-1}(\Omega)$. The lemma follows. \ep

We want to say that $A_{\cS}$ satisfies the same Carleson-type condition as $A \circ \rho_\cS$. For instance, we want to say that $\delta \nabla A \in CM_{\Omega}$ - which implies $(\delta\circ \rho_\cS) \nabla (A\circ \rho_\cS) \in CM_{\rho_{\cS}^{-1}(\Omega)}$ - will give that $(\delta\circ \rho_\cS) \nabla A_{\cS} \in CM_{\rho_{\cS}^{-1}(\Omega)}$. 
However, it is not true, for the simple reason that the Carleson estimates on $\Jac$ are related to the set $\R^n \setminus P$ while the ones on $A \circ \rho_\cS$ are linked to the domain $\rho_{\cS}^{-1}(\Omega)$. Since $A_\cS$
is the product of these two objects, we only have Carleson estimates for $A_{\cS}$ in the areas of $\R^n$ where $\rho_{\cS}(\partial \Omega)$ looks like $P$. 

\begin{lemma} \label{LprAS}
Assume that the matrix function $\A$ defined on $\Omega$ satisfies \eqref{defelliptic} and \eqref{defbounded}, and can be decomposed as $\A=\B+\C$ where 
\begin{equation} \label{prAS2}
 |\delta \nabla \B| + |\C| \in CM_{\Omega}(M)
 \end{equation}
Then the matrix $\A_{\cS}$ constructed in \eqref{defAS} can also be decomposed as $\A_{\cS} = \B_{\cS} + \C_{\cS}$ where $\B_{\cS}$ satisfies \eqref{defelliptic} and \eqref{defbounded} with the constant $2C_{\A}$, $|t\nabla \B_\cS|$ is uniformly bounded by $CC_{\A}$, and
\begin{equation} \label{prAS1}
(|t \nabla \B_{\cS}| + |\C_{\cS}|)\1_{\rho^{-1}_{\cS}(W^*_\Omega(\cS))} \in CM_{\R^n \sm P}(C(\epsilon_0^2+M)) 
\end{equation}
for a constant $C$ that depends only on $n$ and the ellipticity constant $C_\A$.
\end{lemma}

\bp
Let $\A = \B + \C$ as in the lemma. Without loss of generality, we can choose $\B$ to be a smooth average of $\A$ (see Lemma \ref{LellipB=ellipA}) and so $\B$ satisfies \eqref{defelliptic} and \eqref{defbounded} with the constant $C_{\A}$ and $|\delta \nabla \B| \leq CC_\A$. Define
\[\B_{\cS} := \det(J) J^{-T} (\B \circ \rho_{\cS}) J^{-1} \]
and of course $\C_{\cS} := \A_{\cS} - \B_{\cS}$. First, Lemma \ref{LestonJ} shows that $\det(J)$ is close to $1$ and $J^{-1}$ is close to the identity, so $\B_{\cS}$ satisfies \eqref{defelliptic} and \eqref{defbounded} with the constant $(1+C\epsilon_0)C_\A \leq 2C_\A$.
Moreover, the same Lemma \ref{LestonJ}  gives that $|\det(J)| + \|J^{-1}\| \leq 3$, $\|\Jac - I\| \leq 3$, and $\ds |\nabla_{p,t} \det(J)| + \||\nabla_{p,t} J^{-1}|\| \lesssim |\nabla_{p,t} \nabla_p b^t|$, and hence
\[|\nabla \B_{\cS}| \lesssim |(\nabla \B) \circ \rho_{\cS}| +  |\nabla_{p,t} \nabla_p b^t|,\]
and
\[\begin{split} 
|\C_{\cS}| & \lesssim |\det(\Jac) - \det(J)| + \|\Jac^{-1} - J^{-1}\| + |\C \circ \rho_{\cS}| \\
& \lesssim |\partial_t b^t| + |t\nabla_{p,t} \nabla_p b^t| + |\C \circ \rho_{\cS}|.
\end{split}\]
Lemma \ref{LbrCM} entails that $|t\nabla_{p,t} \nabla_p b^t| \lesssim \epsilon_0 \leq 1 \leq C_{\A}$, so $\B_{\cS}$ verifies $|t\nabla \B_{\cS}| \lesssim C_\A$, so thus \eqref{prAS1} is the only statement we still have to prove. Lemma \ref{LbrCM} also implies that $|\partial_t b^t| + |t\nabla_{p,t} \nabla_p b^t| \in CM_{P\times (0,\infty)}(C\epsilon_0^2)$. Therefore, it suffices to establish that 
\begin{equation} \label{prAS3}
(|t\nabla \B \circ \rho_{\cS}| + |\C \circ \rho_{\cS}|) \1_{\rho^{-1}_{\cS}(W^*_\Omega(\cS))} \in CM_{P\times (0,\infty)}(CM).
 \end{equation}
Take $p_0 \in P$ and $r_0>0$. We want to show that 
\begin{equation} \label{prAS4}
\iint_{B(p_0,r_0) \cap \rho^{-1}_{\cS}(W^*_\Omega(\cS))} (|t\nabla \B \circ \rho_{\cS}|^2 + |\C \circ \rho_{\cS}|^2) \, \frac{dt}{t} \, dp  \leq CM(r_0)^{n-1}.
 \end{equation}
If $\rho_{\cS}(B(p_0,r_0)) \cap W^*_{\Omega}(\cS) = \emptyset$, the left hand side above is zero and there is nothing to prove. Otherwise, pick a point $X\in \rho_{\cS}(B(p_0,r_0)) \cap W^*_{\Omega}(\cS)$. The fact that $X\in \rho_{\cS}(B(p_0,r_0))$ means that 
\begin{equation} \label{prAS5}
|X-\bb(p_0)| \leq (1+C\epsilon_0)r_0.
 \end{equation}
since $\rho_{\cS}(p_0) = \bb(p_0)$ and $\|\Jac - I\| \leq C\epsilon_0$ by Lemma \ref{LestonJ}. Because $b$ is $2\epsilon_0$-Lipschitz with $\epsilon_0 \ll 1$, we deduce
\begin{equation} \label{prAS6}
|X-\bb(\Pi(X))| \leq (1+\epsilon_0) |X-\bb(p_0)| \leq (1+2C\epsilon_0)r_0.
 \end{equation}
 The fact that $X\in W^*_{\Omega}(\cS)$ implies by \eqref{claimdPib} that
 \begin{equation} \label{prAS7} 
\delta(X) \leq (1+2\epsilon_0)|X-\bb(\Pi(X))| \leq 2r_0
\end{equation}
thanks to \eqref{prAS6}.
Moreover, if $x \in \partial \Omega$ is such that $|X-x| = \delta(X)$,
\begin{multline} \label{prAS8}
|x-\bb(p_0)| \leq |x-\bb(\Pi(X))| + |\bb(\Pi(X) - \bb(p_0)| \\ \leq  2\epsilon_0\delta(X) + (1+\epsilon_0) |\Pi(X) - p_0|  \leq \frac12 r_0 + (1+\epsilon_0) |X-\bb(p_0)| \leq 2r_0
 \end{multline}
by using in order \eqref{claimbxd}, the fact that $b$ is $2\epsilon_0$-Lipschitz, \eqref{prAS7}, and \eqref{prAS5}. Fix $X_0 \in \rho_{\cS}(B(p_0,r_0)) \cap W^*_{\Omega}(\cS)$ and $x_0\in\pom$ such that $|X_0 - x_0| = \delta(X_0)$. The inequalities \eqref{prAS5} and \eqref{prAS8} show that, 
\[|X-x_0| \leq |X-\bb(p_0)| + |x_0 - \bb(p_0)| \leq 4r_0 \qquad \text{ for } X\in \rho_{\cS}(B(p_0,r_0)) \cap W^*_{\Omega}(\cS),\]
that is 
\begin{equation} \label{prAS9}
\rho_{\cS}(B(p_0,r_0)) \cap W^*_{\Omega}(\cS) \subset B(x_0,4r_0).
 \end{equation}
 We are now ready to conclude. We make the change of variable $X = \rho_{\cS}(p,s)$ in \eqref{prAS4}, and since $\rho_\cS$ is a bi-Lipschitz change of variable that almost preserves the distances (because $\|\Jac - I\| \leq C\epsilon_0 \ll 1$), we obtain 
\[\begin{split}
 \iint_{B(p_0,r) \cap \rho^{-1}_{\cS}(W^*_\Omega(\cS))} & (|t\nabla \B \circ \rho_{\cS}|^2 + |\C \circ \rho_{\cS}|^2) \, \frac{dt}{t} \, dp \\
 & \leq 2  \iint_{B(x_0,4r) \cap W^*_\Omega(\cS)} (|\dist(X,\Gamma_{\cS}) \nabla \B|^2 + |\C|^2) \,  \frac{dX}{\dist(X,\Gamma_{\cS})} \\
 & \leq 4 \iint_{B(x_0,4r)} (|\delta\nabla \B|^2 + |\C|^2) \,  \frac{dX}{\delta(X)} \leq CM(r_0)^{n-1}  
 \end{split}\]
by using \eqref{claimdGPib} and then the fact that $|\delta\nabla \B| + |\C| \in CM_{\Omega}(M)$. The lemma follows.
\ep

\subsection{Properties of the composition of the smooth distance by $\rho_\cS$}

The change of variable $\rho_{\cS}$ maps $P\times (P^\bot\sm \{0\})$ to $\R^n \setminus \Gamma_{\cS}$, so for any $X\in \R^n \sm \Gamma_{\cS}$, the quantities $N_{\rho^{-1}_{\cS}(X)}(Y)$ and $\Lambda(\rho^{-1}_{\cS}(X))$ make sense as $N_{p,t}(Y)$ and $\Lambda(p,t)$, respectively, where $(p,t) = \rho^{-1}_{\cS}(X)$.  With this in mind, we have the following result.

\begin{lemma} \label{LprDb}
For any $Q\in \cS$, we have
\begin{equation} \label{prDb1}
\iint_{W_\Omega(Q)} \left| \frac{\nabla D_\beta(X)}{D_\beta(X)} - \frac{N_{\rho^{-1}_{\cS}(X)}(X)}{\dist(X,\Lambda(\rho^{-1}_{\cS}(X))} \right|^2\,  \delta(X) \, dX \leq C |\alpha_{\sigma,\beta}(Q)|^2 \sigma(Q),
\end{equation}
with a constant $C>0$ that depends only on $n$, $C_\sigma$, and $\beta$.
\end{lemma}

\bp
The lemma is a consequence of Corollary \ref{CestD} and the definition of $\rho_{\cS}$.

First , Lemma \ref{lemWOG} (d) entails that $W_\Omega(Q) \subset  \R^n \setminus \Gamma_{\cS}$, which means that the quantities $N_{\rho^{-1}_{\cS}(X)}$ and $\Lambda(\rho^{-1}_{\cS}(X))$ are well defined in \eqref{prDb1}. Let $X\in W_\Omega(Q)$ and set $(p,t) = \rho^{-1}_{\cS}(X)$. 

On one hand, Lemma \ref{LprrhoS} gives that
\[\dist(X,\Gamma_{\cS}) \leq |X- \bb(p)| \leq (1+C\epsilon_0) |t|  \leq (1+C'\epsilon_0) \dist(X,\Gamma_{\cS})\]
and
\[|X - \bb(p) - (0,t)| \leq C\epsilon_0|t|.\]
By projecting the left-hand side on $P$, the latter implies that
\[|\Pi(X) - p| \leq C\epsilon_0|t|.\]
On the other hand, since $X\in W_\Omega(Q)$, Lemma \ref{LclaimPib} gives that
\[\dist(X,\Gamma_{\cS}) \leq |X- \bb(\Pi(X))| \leq (1+2\epsilon_0) \delta(X) \leq (1+C\epsilon_0) \dist(X,\Gamma_{\cS}),\]
and, if $x \in Q$ is such that $|X-x| = \delta(X)$, 
then by \eqref{claimbxd},
\[|\bb(\Pi(X)) - x| \leq 2 \epsilon_0 \delta(X),\]
which implies that 
\[|\Pi(X) - \Pi(x)| \leq 2 \epsilon_0 \delta(X),\]
Altogether, we have 
\[ \delta(X) (1-C\epsilon_0) \leq |t| \leq (1+C\epsilon_0) \delta(X)\]
and
\[\dist(p,\Pi(Q)) \leq |p - \Pi(x)|\le\abs{p-\Pi(X)}+\abs{\Pi(X)-\Pi(x)} \leq C\epsilon_0 \delta(X).\]
If we throw in the fact that $\delta(X) \in [\ell(Q)/2,\ell(Q)]$ by definition of $W_\Omega(Q)$, then we easily observe that $p$ and $t$ satisfy the assumptions of Corollary \ref{CestD}, and so
\[\left| \frac{\nabla D_\beta(X)}{D_\beta(X)} - \frac{N_{\rho^{-1}_{\cS}(X)}(X)}{\dist(X,\Lambda(\rho^{-1}_{\cS}(X))} \right| \leq C \ell(Q)^{-1} \alpha_{\sigma,\beta}(Q) \qquad \text{ for } X\in W_\Omega.\]
We conclude that
\[\begin{split}
\iint_{W_\Omega(Q)} \left| \frac{\nabla D_\beta(X)}{D_\beta(X)} - \frac{N_{\rho^{-1}_{\cS}(X)}(X)}{\dist(X,\Lambda(\rho^{-1}_{\cS}(X))} \right|^2\,  \delta(X) \, dX 
& \leq C |W_\Omega(Q)| |\ell(Q)^{-1} \alpha_{\sigma,\beta}(Q)|^2 \ell(Q) \\
& \leq C |\alpha_{\sigma,\beta}(Q)|^2 \sigma(Q) 
\end{split}\]
because $|W_\Omega(Q)| \approx \sigma(Q) \ell(Q)$ by \eqref{prWO2} and \eqref{defADR}. The lemma follows.
\ep

\begin{lemma} \label{LprNpt}
We have
\begin{equation}\label{prNpt1}
\iint_{\rho_{\cS}^{-1}(W_\Omega(\cS))} \left| \frac{\nabla t}{t} - \frac{\Jac(p,t) N_{p,t}(\rho_{\cS}(p,t))}{\dist(\rho_{\cS}(p,t), \Lambda(p,t))}\right|^2 \abs{t} \, dt\, dp \leq C(\epsilon_0)^2 \sigma(Q(\cS))
\end{equation}
where $C>0$ depends only on $n$ (and $\eta$).
\end{lemma}

\bp
From the definition, we can see that $\Lambda(p,t)$ is the affine plane that goes through the point $\bb^t(p)$ and whose directions are given by the vectors $(q,q \nabla b^t(p))$, that is $\Lambda(p,t)$ is the codimension 1 plane that goes through $\bb^t(p)$ and with upward unit normal vector 
\[N_{p,t} = \frac{1}{|(-(\nabla b^r(p))^T,1)|} \begin{pmatrix} -\nabla b^t(p) \\1  \end{pmatrix} = \frac{1}{\sqrt{1+|\nabla b^t(p)|^2}} \begin{pmatrix} -\nabla b^t(p) \\1  \end{pmatrix} \]
The vector function $N_{p,t}(X)$ is just $+N_{p,t}$ or $-N_{p,t}$, depending whether $X$ lies above or below $\Lambda(p,t)$.

Observe that $\rho_{\cS}(p,t) - \bb^t(p) = t(-(\nabla b^t(p))^T,1)$, which means that $\bb^t(p)$ is the projection of $\rho_{\cS}(p,t)$ onto $\Lambda(p,t)$ and that 
\[\dist(\rho_{\cS}(p,t), \Lambda(p,t)) = |t| |(-(\nabla b^t(p))^T,1)| = |t| \sqrt{1+|\nabla b^t(p)|^2}.\]
Moreover, $\rho_{\cS}(p,t)$ lies above $\Lambda(p,t)$ if $t>0$ and below otherwise, that is
\[N_{p,t}(\rho(p,t)) = \sgn(t) N_{p,t} =  \frac{\sgn(t)}{\sqrt{1+|\nabla b^t(p)|^2}} \begin{pmatrix} -\nabla b^t(p) \\1  \end{pmatrix}.\]
From all this, we deduce
\begin{equation}\label{prNpt2}
\begin{split}
\frac{J(p,t) N_{p,t}(\rho_{\cS}(p,t))}{\dist(\rho_{\cS}(p,t), \Lambda(p,t))} & = \frac{1}{t(1+|\nabla b^t(p)|^2)} \begin{pmatrix} I & \nabla b^t(p) \\ -(\nabla b^t(p))^T & 1  \end{pmatrix} \begin{pmatrix} -\nabla b^t(p) \\1  \end{pmatrix} \\
& = \frac{1}{t} \begin{pmatrix} 0_{\R^{n-1}} \\ 1  \end{pmatrix} = \frac{\nabla t}t.
\end{split}\end{equation} 
Recall that $|\Jac - J| \lesssim |\partial_t b^t| + |t \nabla_{p,t} \nabla_p b^t|$. Together with \eqref{prNpt2}, we obtain that the left-hand side of \eqref{prNpt1} is equal to
\begin{multline}
I = \iint_{\rho^{-1}_{\cS}(W_\Omega(\cS))} \left| \frac{\Jac(p,t) - J(p,t)}{t(1+|\nabla b^t(p)|^2)} \begin{pmatrix} -\nabla b^t(p) \\1  \end{pmatrix} \right|^2 \abs{t} \, dt\, dp 
\\ \lesssim \iint_{\rho_{\cS}^{-1}(W_\Omega(\cS))} (|\partial_t b^t|^2 + |t \nabla_{p,t} \nabla_p b^t|^2) \, \frac{dt}{\abs{t}} \, dp.
\end{multline}
Take $X_0 \in W_{\Omega}(\cS)$, and notice that the set $W_{\Omega}(\cS)$ is included in $B(\bb(\Pi(X_0)),4\ell(Q(\cS)))$ by definition of $W_\Omega(\cS)$ and by \eqref{claimbxd}. Since the Jacobian of $\rho_{\cS}$ is close to the identity, $\rho_{\cS}^{-1}$ almost preserves the distance, and hence $\rho^{-1}_{\cS}(W_{\Omega}(\cS)) \subset B(\Pi(X_0),5\ell(Q(\cS)))$. We conclude that
\[I \lesssim \iint_{B(\Pi(X_0),5\ell(Q(\cS)))} (|\partial_t b^t|^2 + |t \nabla_{p,t} \nabla_p b^t|^2) \, \frac{dt}{t} \, dp \lesssim (\epsilon_0)^2 \ell(Q(\cS))^{n-1} \lesssim (\epsilon_0)^2 \sigma(Q(\cS))\]
by Lemma \ref{LbrCM} and then \eqref{defADR}. The lemma follows.
\ep

\section{The flat case.}

\label{Sflat}

In this section, we intend to prove an analogue of Theorem \ref{Main1} when the boundary is flat, that is when the domain is  $\Omega_0:= \R^n_+$. This is our main argument on the PDE side (contrary to other sections which are devoted to geometric arguments) and the general case of Chord-Arc Domains is eventually brought back to this simpler case.

We shall bring a little bit of flexibility in the following manner. We will allow $\Omega$ to be different from $\R^n_+$, but we shall stay away from the parts where $\partial \Omega$ differs from $\partial\Rn_+$ with some cut-off functions. More exactly, we shall use cut-off functions $\phi$ that guarantee that $\delta(X) := \dist(X,\partial \Omega) \approx t$ whenever $X = (x,t) \in \supp \phi$. We shall simply use $\R^{n-1}$ for $\partial\Rn=\R^{n-1}\times\set{0}$. We start with the precise definition of the cut-off functions that we are allowing.

\begin{definition} \label{defcutoffboth}
We say that $\phi \in L^\infty(\Omega)$ is a cut-off function associated to both $\partial \Omega$ and $\R^{n-1}$ if $0 \leq \phi \leq 1$, and there is a constant $C_\phi \geq 1$ such that $|\nabla \phi| \leq C_\phi \delta^{-1}$,
\begin{equation} \label{t=dist}
(C_\phi)^{-1} |t| \leq \delta(X) \leq C_\phi |t| \qquad \text{ for all } X=(x,t) \in \supp \phi,
\end{equation}
and there exists a collection of dyadic cubes $\{Q_i\}_{i\in I_\phi}$ in $\D_{\partial \Omega}$ such that
\begin{equation} \label{Qioverlap}
\text{$\{Q_i\}_{i\in I_\phi}$ is finitely overlapping with an overlap of at most $C_\phi$,}
\end{equation}
and
\begin{equation} \label{1phiCM3}
\Omega \cap (\supp \phi) \cap \supp (1-\phi) \subset \bigcup_{i\in I_\phi} W_\Omega^{**}(Q_i). 
\end{equation}
\end{definition}

The condition \eqref{t=dist} allows us to say that  
\begin{multline} \label{BOtoBR}
\text{ if, for $x\in \partial \Omega$ and $r>0$, $B(x,r) \cap \supp \phi \neq \emptyset$,} \\
\text{ then there exists $y\in \R^{n-1}$ such that $B(x,r) \subset B(y,Cr)$;}
\end{multline}
so we can pass from Carleson measures in $\Omega$ to Carleson measure in $\R^n \setminus \R^{n-1}$. For instance, we have
\begin{equation} \label{CMOCMR} \begin{array}{c}
f \in CM_{\Omega}(M) \implies f\phi, \, f\1_{\supp \phi} \in CM_{\R^n \setminus \R^{n-1}}(C'_\phi M), \\
\delta \nabla g \in CM_{\Omega}(M) \implies t \phi \nabla g \in  CM_{\R^n \setminus \R^{n-1}}(C'_\phi M).
\end{array} \end{equation}
and vice versa. The conditions \eqref{Qioverlap} and \eqref{1phiCM3} ensure that $\1_{(\supp \phi) \cap \supp (1-\phi)}$ (and hence $\delta \nabla \phi$) satisfies the Carleson measure condition on $\om$. So by \eqref{CMOCMR}, 
\begin{equation} \label{1phiCM}
|t \nabla \phi| + \1_{\supp \nabla \phi } + \1_{(\supp \phi) \cap \supp (1-\phi)}  \in CM_{ \R^n \setminus \R^{n-1}}(C'_\phi).
\end{equation}
And if the support of of $\phi$ is contained in a ball of radius $r$ centered on $\partial \Omega$, then 
\begin{equation} \label{1phiCM2}
\iint_\Omega \big(|\nabla \phi|t + |t\nabla \phi|^2\big) \, \frac{dt}{t}\, dy \lesssim r^{n-1}.
\end{equation}

\medskip

We are ready to state the main result of the section.

\begin{lemma} \label{lemflat}
Let $\Omega$ be a Chord-Arc Domain and let $L=-\diver \mathcal A \nabla$ be a uniformly elliptic operator on $\Omega$, that is $\A$ verifies \eqref{defelliptic} and \eqref{defbounded}. Assume that the $L^*$-elliptic measure $\omega_{L^*}\in A_\infty(\sigma)$, where $L^*$ is the adjoint operator of $L$, and $\sigma$ is an Ahlfors regular measure on $\pom$. Let $\phi$ be as in Definition \ref{defcutoffboth} and be supported in a ball $B:=B(x,r)$ centered on the boundary $\partial \Omega$. Assume that the coefficients $\A$ can be decomposed as $\mathcal A  = \mathcal B + \mathcal C$ where
\begin{equation} \label{NB+CareCM}
 (|t\nabla \mathcal B| + |\mathcal C|)\1_{\supp \phi} \in CM_{\R^n \sm \R^{n-1}}(M)\footnote{We actually only need the Carleson condition on the last column of $\mathcal B$ and $\mathcal C$ (instead of the full matrix).}.
\end{equation}
 Then for any non-negative nontrivial weak solution $u$ to $Lu = 0$ in $2B \cap \Omega$ with zero trace on $\partial \Omega \cap 2B$, one has
\begin{equation} \label{flat1}
\iint_{\Omega}  |t|\left|\frac{\nabla u}{u} - \frac{\nabla t}{t} \right|^2 \phi^2 \, dt\, dy = \iint_{\Omega}  |t|\left|\nabla \ln\Big( \frac{u}{|t|} \Big) \right|^2 \phi^2 \, dt\, dy \leq C(1+M) r^{n-1},\end{equation}
where $C$ depends only on the dimension $n$, the elliptic constant $C_\A$, the 1-sided CAD constants of $\Omega$, the constant $C_\phi$ in Definition \ref{defcutoffboth}, and the intrinsic constants in $\omega_{L^*} \in A_\infty(\sigma)$.
\end{lemma}

The above lemma is the analogue of Theorem 2.21 from \cite{DFMGinfty} in our context, and part of our proof will follow the one from \cite{DFMGinfty} but a new argument is needed to treat the non-diagonal structure of $\mathcal A$.

We need $\omega_{L^*} \in A_\infty(\sigma)$ for the proof of the following intermediate lemma. Essentially, we need that the logarithm of the Poisson kernel lies in $BMO$. Let us state and prove it directly in the form that we need.

\begin{lemma} \label{lemlogk}
Let $\Omega$, $L$, $\phi$, $B:=B(x,r)$, and $u$ be as in Lemma \ref{lemflat}. Assume that $\omega_{L^*}\in A_\infty(\sigma)$ as in Lemma \ref{lemflat}. Then there exists $K:= K(u,B)$ such that 
\[\iint_{\Omega} |\nabla \phi| \left|\ln\Big( \frac{K u}{|t|} \Big) \right| dt\, dy \leq C r^{n-1},\]
where $C$ depends only on $n$, $C_\A$, the 1-sided CAD constants of $\Omega$, the constant $C_\phi$ in Definition \ref{defcutoffboth},  and the intrinsic constants in $\omega_{L^*} \in A_\infty(\sigma)$.
\end{lemma}

\noindent {\em Proof of Lemma \ref{lemlogk}.} The first step is to replace $Ku/t$ by the elliptic measure. Take $X_0 \in B(x,r) \cap \Omega$ and $X_1\in \Omega \setminus B(x,4r)$ to be two corkscrew points for $x$ at the scale $r$. 
If $G(Y,X)$ is the Green function associated to $L$ in $\Omega$ and $\{\omega^X_*\}_{X\in \Omega}$ is the elliptic measure associated to the adjoint $L^*$, the CFMS estimates (Lemma \ref{LCFMS}) entails, for $Y \in W_{\Omega}^{**}(Q) \cap B$, that
\[\frac{u(Y)}{u(X_0)} \approx \frac{G(Y,X_1)}{G(X_0,X_1)} \approx  \frac{\ell(Q)}{r} \frac{\sigma(\Delta)}{\sigma(Q)} \frac{\omega_*^{X_1}(Q)}{\omega_*^{X_1}(\Delta)},\]
where $\Delta = B\cap \partial \Omega$.
Moreover, if $Y=(y,t)\in \supp \phi \cap W^{**}_\Omega(Q)$, then $\ell(Q) \approx |t|$ by \eqref{t=dist}. Altogether, we have
\begin{equation} \label{CFMSa}
\frac{u(Y)}{|t|} \approx \frac{u(X_0)}{r} \frac{\sigma(\Delta)}{\sigma(Q)} \frac{\omega_*^{X_1}(Q)}{\omega_*^{X_1}(\Delta)} \qquad \text{ for } Y=(y,t) \in \supp \phi \cap W_{\Omega}^{**}(Q).
\end{equation}
Set $K := r/u(X_0)$, and $I_\phi':=\set{i\in I_\phi:  W_\om^{**}(Q_i) \text{ intersects }\supp\nabla\phi}$,
\begin{equation} \label{logk1}
\begin{split}
\iint_{\Omega} |\nabla \phi| \left|\ln\Big( \frac{K u}{|t|} \Big) \right| dt\, dy
& \lesssim \sum_{i\in I_\phi'} \ell(Q_i)^{-1} \int_{W_\Omega^{**}(Q_i)} \left|\ln\Big( \frac{K u}{|t|} \Big) \right| \, dt\, dy \\
& \lesssim \sum_{i\in I_\phi'} \sigma(Q_i) \left[ 1 + \left|\ln\Big(\frac{\sigma(\Delta)}{\sigma(Q_i)} \frac{\omega_*^{X_1}(Q_i)}{\omega_*^{X_1}(\Delta)} \Big) \right| \right]
 \end{split}
 \end{equation}
 by \eqref{1phiCM3},  \eqref{CFMSa}, and the fact that $|W^{**}_\Omega(Q_i)| \approx \ell(Q_i) \sigma(Q_i)$.
 
 The second step is to use the fact that $\omega^{X_1}_*$ is $A_\infty$-absolutely continuous with respect to $\sigma$. To that objective, we define for $k\in \mathbb Z$
 \[\mathcal I_k := \Big\{i \in I_\phi', \, 2^{k} \leq \frac{\sigma(\Delta)}{\sigma(Q_i)} \frac{\omega_*^{X_1}(Q_i)}{\omega_*^{X_1}(\Delta)}  \leq 2^{k+1}\Big\}\]
 and then $E_k:= \bigcup_{i\in \mathcal I_k} Q_i$. Since the collection $\{Q_i\}_{i\in I_\phi}$ is finitely overlapping, due to \eqref{Qioverlap}, the bound \eqref{logk1} becomes
\begin{equation} \label{logk2}
\begin{split}
\iint_{\Omega} |\nabla \phi| \left|\ln\Big( \frac{K u}{|t|} \Big) \right| dt\, dy
& \lesssim \sum_{k\in \mathbb Z} (1+|k|) \sigma(E_k). 
 \end{split}
 \end{equation}
We want thus to estimate $\sigma(E_k)$. Observe first that for any $i\in I_\phi'$, $W_\Omega^{**}(Q_i)$ intersects $\supp \phi \subset B$. Therefore $Q_i$ and $E_k$ have to be inside $\Delta^*:=C\Delta$ for a large $C$ depending only on the constant $K^{**}$ in \eqref{defWO**}. The finite overlapping \eqref{Qioverlap} also implies that 
\[\frac{\sigma(\Delta^*)}{\sigma(E_k)} \frac{\omega_*^{X_1}(E_k)}{\omega_*^{X_1}(\Delta^*)} \approx 2^k\]
For $k\geq 0$, we have
\begin{equation} \label{logk3}
\frac{\sigma(E_k)}{\sigma(\Delta^*)} \approx 2^{-k} \frac{\omega_*^{X_1}(E_k)}{\omega_*^{X_1}(\Delta^*)} \lesssim 2^{-k}.
 \end{equation}
The elliptic measure $\omega_*^{X_1}$ is $A_\infty$-absolutely continuous with respect to $\sigma$ by assumption, so for $k\leq 0$, we use the characterization (iv) from Theorem 1.4.13 in \cite{KenigB} to deduce
\begin{equation} \label{logk4}
\frac{\sigma(E_k)}{\sigma(\Delta^*)} \lesssim \left(\frac{\omega_*^{X_1}(E_k)}{\omega_*^{X_1}(\Delta^*)}\right)^\theta  \approx 2^{k\theta} \left(\frac{\sigma(E_k)}{\sigma(\Delta^*)}\right)^\theta  \lesssim 2^{k\theta}
 \end{equation}
for some $\theta\in (0,1)$ independent of $x$, $r$, and $k$. We reinject \eqref{logk3} and \eqref{logk4} in \eqref{logk2} to conclude that
\[\begin{split}
\iint_{\Omega} |\nabla \phi| \left|\ln\Big( \frac{K u}{|t|} \Big) \right| dt\, dy
& \lesssim \sigma(\Delta^*) \sum_{k\in \mathbb Z} (1+|k|) 2^{-|k|\theta} \lesssim \sigma(\Delta^*) \lesssim r^{n-1} 
 \end{split}\]
because $\sigma$ is Ahlfors regular. The lemma follows.
\ep

\medskip

\noindent {\em Proof of Lemma \ref{lemflat}.} The proof is divided in two parts: the first one treats the case where $\B_{i,n}=0$ for $i<n$, 
and the second one shows that we can come back to the first case by a change of variable, by adapting the method presented in \cite{FenDKP}.

Observe that 
$\phi$ can be decomposed as $\phi=\phi_+ + \phi_-$ where $\phi_+ = \1_{t>0} \phi$ and $\phi_- = \1_{t<0} \phi$. Both $\phi_+$ and $\phi_-$ are as in Definition \ref{defcutoffboth} with constant $C_\phi$. So it is enough to prove the lemma while assuming 
\begin{equation} \label{phit>0}
\supp \phi \subset  \{t\geq 0\} = \overline{\R^n_+}.
\end{equation}
The proof of the case $\supp \phi \subset \R^n_-$ is of course identical up to obvious changes.

\medskip

\noindent {\bf Step 1:} Case where $\mathcal B_{i,n} = 0$ for $i<n$ on $\supp \phi$ and $\B$ satisfies \eqref{defelliptic} and \eqref{defbounded} with the same constant $C_\A$ as $\A$. If $b:= \B_{n,n}$, this assumption on $\mathcal B$ implies that 
\begin{equation} \label{prBinS1}
\B \nabla t \cdot \nabla v \, \phi^2= b \, \partial_t v \, \phi^2.
\end{equation}
whenever $v\in W^{1,1}_{loc}(\Omega)$ and
\begin{equation}  \label{Blambda}
(C_\A)^{-1} \leq b \leq C_\A.
\end{equation}
We want to prove \eqref{flat1} with the assumption \eqref{phit>0}, and for this, we intend to establish that 
\begin{equation} \label{flat2}
T:= \iint_{\R^n_+} t\left| \nabla \ln\Big( \frac{u}{t} \Big) \right|^2 \phi^2 \, dt\, dy \lesssim T^{\frac12} r^{\frac{n-1}2} + r^{n-1},
\end{equation}
which implies the desired inequality \eqref{flat1} provided that $T$ is {\em a priori} finite. However that is not necessary the case, because some problems can occur when $t$ is close to 0. 
So we take $\psi \in C^\infty(\R)$ such that $\psi(t) = 0$ when $t<1$, $\psi(t) = 1$ when $t\geq 2$, and $0 \leq \psi \leq 1$. We construct then $\psi_k(Y) = \psi(2^k\delta(Y))$ and $\phi_k = \phi \psi_k$. 
It is not very hard to see that 
\[\supp \nabla \psi_k := \{X\in \Omega, \, 2^{-k} \leq \delta(X) \leq 2^{1-k}\} \subset \bigcup_{Q\in \mathbb D_k} W_\Omega^{**}(Q)\]
and therefore that $\phi_k$ is as in Definition \ref{defcutoffboth} (with $C_{\phi_k} = C_\phi +1$). The quantity 
\[T(k) := \iint_{\R^n_+} t \left| \nabla \ln\Big( \frac{u}{t} \Big) \right|^2 \phi_k^2 \, dt\, dy
 = \iint_{\R^n_+} t \left| \frac{\nabla u}{u} - \frac{\nabla t}{t} \right|^2 \phi_k^2  \, dt\, dy\]
 is finite, because $\phi_k$ is compactly supported in both $\Omega$ and $\R^{n}_+$ (the fact that $\nabla u/u$ is in $L^2_{loc}(\Omega)$ for a non-negative nontrivial solution to $Lu=0$ is a consequence of the Caccioppoli inequality and the Harnack inequality). So, we prove \eqref{flat2} for $T(k)$ instead of $T$, which implies $T(k) \lesssim r^{n-1}$ as we said, and take $k\to \infty$ to deduce \eqref{flat1}. 
 
 \medskip
 
 We are now ready for the core part of the proof, which can be seen as an elaborate integration by parts. Our previous discussion established that we (only) have to prove \eqref{flat2}, and that we can assume that $\phi$ is compactly supported in $\Omega \cap \R^{n}_+$. 
We use the ellipticity of $\A$ and the boundedness of $b$ to write
\[\begin{split}
T & = \iint_{\R^n_+} t \left| \frac{\nabla u}{u} - \frac{\nabla t}{t}\right|^2 \phi^2 \, dt\, dy
\le C_{\A}^2\iint_{\R^n_+}\frac{\A}{b}\br{\frac{\nabla u}{u}-\frac{\nabla t}{t}}\cdot\br{\frac{\nabla u}{u}-\frac{\nabla t}{t}}\phi^2dtdy
\\
& = C_{\A}^2\br{ \iint_{\R^n_+} \frac{\A \nabla u}{bu} \cdot \left( \frac{\nabla u}{u} - \frac{\nabla t}{t} \right) \, t\phi^2 \, dt\, dy 
- \iint_{\R^n_+} \frac{\A \nabla t}{bt} \cdot \nabla \ln \Big( \frac{u}{t}\Big) \, t\phi^2 \, dt\, dy} \\
& :=C_{\A}^2( T_1 + T_2).
\end{split}\]
We deal first with $T_2$. We use the fact that $\A = \B + \mathcal C$ and \eqref{prBinS1} to obtain
\[T_2  = - \iint_{\R^n_+} \partial_t \ln \Big( \frac{u}{t}\Big) \, \phi^2 \, dt\, dy
- \iint_{\R^n_+} \frac{\mathcal C}{b} \nabla t \cdot \nabla \ln \Big( \frac{u}{t}\Big) \, \phi^2 \, dt\, dy := T_{21} + T_{22}.\]
The term $T_{22}$ can be then bounded with the help of the Cauchy-Schwarz inequality as follows
\[T_{22}   \leq \|b^{-1}\|_\infty \left( \iint_{\R^{n}_+} |\mathcal C|^2\phi^2 \, \frac{dt}{t} \, dy \right)^\frac12 \left( \iint_{\R^{n}_+}  t \left|\nabla \ln \Big( \frac{u}{t}\Big)\right|^2 \, \phi^2 \, dt \, dy \right)^\frac12 \lesssim r^{\frac{n-1}2}T^{\frac12}\]
by \eqref{NB+CareCM}.
As for $T_{21}$, observe that multiplying by any constant $K$ inside the logarithm will not change the term (because we differentiate the logarithm). As a consequence, we have
\[\begin{split}
T_{21} & = - \iint_{\R^{n}_+} \partial_t \ln \Big( \frac{Ku}{t}\Big) \, \phi^2 \, dt\, dy = \iint_{\R^{n}_+}  \ln \Big( \frac{Ku}{t}\Big) \, \partial_n [\phi^2] \, dt\, dy \\
& \leq \iint_{\R^{n}_+}  |\nabla \phi| \left|\ln \Big( \frac{Ku}{t}\Big)\right|  \, dt\, dy \lesssim r^{n-1}
\end{split}\]
by successively using integration by parts and Lemma \ref{lemlogk}.

We turn to $T_1$, and we want now to use the fact that $u$ is a weak solution to $Lu = 0$. So we notice that
\[\begin{split}
T_1 & = - \iint_{\R^{n}_+} \frac{\A}{b} \nabla u \cdot \nabla \Big( \frac{t}{u} \Big) \, \phi^2 \, dt\, dy \\
& = - \iint_{\R^{n}_+} \A \nabla u \cdot \nabla \Big( \frac{t\phi^2}{bu} \Big) \, \, dt\, dy + 2 \iint_{\R^{n}_+}  \A \nabla u \cdot \nabla \phi \, \Big(\frac{t\phi}{bu} \Big) \, dt\, dy - \iint_{\R^n_+} \A \nabla u \cdot \nabla b \, \Big(\frac{t\phi^2}{b^2u} \Big) \, dt\, dy \\
& := - T_{11} + 2 T_{12} - T_{13}.
\end{split}\]
Since $\phi$ is compactly supported, we have that $u > \epsilon_\phi$ on $\supp \phi$ (by the Harnack inequality, see Lemma \ref{Harnack}) and $\nabla u\in L^2_{loc}(\Omega)$ (by the Caccioppoli inequality, see Lemma \ref{Caccio}). Therefore $t\phi^2/(bu)$ is a valid test function for the solution $u\in W^{1,2}_{loc}(\Omega)$ to $Lu=0$, and then $T_{11} = 0$.
As for $T_{12}$, we have
\[\begin{split}
T_{12} & = \iint_{\R^{n}_+} \frac{\A}{b} \left( \frac{\nabla u}{u} -\frac{\nabla t}{t} \right) \cdot \nabla \phi \, (t\phi) \, dt\, dy + \iint_{\R^{n}_+} \frac{\A}b \nabla t \cdot \nabla \phi \, \phi \, dt\, dy := T_{121} + T_{122}.
\end{split}\]
The term $T_{121}$ is similar to $T_{22}$. The boundedness of $\A/b$ and the Cauchy-Schwarz inequality infer that
\[T_{121}   \leq \left( \iint_{\R^{n}_+} t|\nabla \phi|^2 \, \frac{dt}{t} \, dy \right)^\frac12 \left( \iint_{\R^{n}_+}  t \left|\nabla \ln \Big( \frac{u}{t}\Big)\right|^2 \, \phi^2 \, dt \, dy \right)^\frac12 \lesssim r^{\frac{n-1}2}T^{\frac12}\]
by \eqref{1phiCM2}. The quantity $T_{122}$ is even easier since
\[T_{122} \lesssim \iint_{\R^{n}_+}  |\nabla \phi| \, dt\, dy \lesssim r^{n-1},\]
again by \eqref{1phiCM2}. It remains to bound $T_{13}$.  We start as for $T_{12}$ by writing
\[\begin{split}
T_{13} & = \iint_{\R^{n}_+}  \mathcal A \left( \frac{\nabla u}{u} - \frac{\nabla t}{t} \right) \cdot \nabla b \, \frac{t\phi^2}{b^2} \, dt\, dy + \iint_{\R^{n}_+} \mathcal A \nabla t \cdot \nabla b \, \frac{\phi^2}{b^2} \, dt\, dy := T_{131} + T_{132}.
\end{split}\]
The term $T_{131}$ is like $T_{121}$, and by using $t\nabla b \in CM_{\R^n_+}$ instead of $t\nabla \phi \in CM_{\R^n_+}$ , we obtain that $T_{131} \lesssim r^{(n-1)/2}T^{1/2}$. The term $T_{132}$ does not contain the solution $u$, but it is a bit harder than $T_{122}$ to deal with, because $\nabla b$ is not as nice as $\nabla \phi$. We use $\mathcal A = \mathcal B + \mathcal C$ and \eqref{prBinS1} to get
\[\begin{split}
T_{132} & = \iint_{\R^{n}_+}  (\partial_t b) \, \frac{\phi^2}{b} \, dt\, dy + \iint_{\R^{n}_+} \mathcal C \nabla t \cdot \nabla b \, \frac{\phi^2}{b^2} \, dt\, dy := T_{1321} + T_{1322}.
\end{split}\]
We easily deal with $T_{1322}$ by using the Cauchy-Schwarz inequality as follows:
\[T_{1322}   \leq \|b^{-1}\|_\infty^2 \left( \iint_{\R^{n}_+} |\mathcal C|^2 \phi^2 \, \frac{dt}{t} \, dy \right)^\frac12 \left( \iint_{\R^{n}_+}  |t\nabla b|^2 \phi^2 \, \frac{dt}{t} \, dy\right)^\frac12 \lesssim r^{n-1}\]
by \eqref{NB+CareCM}. As last, observe that 
\[T_{1321} = \iint_{\R^{n}_+}  \partial_t [\ln(b)\phi^2] \, \, dt\, dy - \iint_{\R^{n}_+}  \partial_t \phi \, \phi \ln(b) \, dt\, dy, \]
but the first integral in the right-hand side above is zero, so
\[|T_{1321}| \lesssim \|\ln(b)\|_\infty \iint_{\R^{n}_+}  |\partial_t \phi| \, dt\, dy \lesssim r^{n-1}, \]
by \eqref{1phiCM2} and the fact that $b\approx 1$. The inequality \eqref{flat1} under the three assumptions \eqref{prBinS1}, \eqref{Blambda}, and \eqref{NB+CareCM} follows.

\medskip

\noindent {\bf Step 2: We can assume that $\norm{t \abs{\nabla_y\B}}_\infty$ is as small as we want.}

We construct 
\begin{equation} \label{defwtA}
\wt{\mathcal A} := \mathcal A \phi + (1-\phi) I,
\end{equation}
where $I$ is the identity matrix. Note that $\wt{\mathcal A}$ is elliptic with the same elliptic constant $C_\A$ as $\mathcal A$. We choose then a bump function $\theta \in C^\infty_0(\R^n)$ supported in $B(0,1/10)$, that is $0 \leq \theta \leq 1$ and $\iint_{\R^n} \theta \, dX = 1$. We construct $\theta_{y,t}(z,s) = t^{-n}\theta\big(\frac{z-y}{t},\frac{s-t}{t}\big)$, which satisfies $\iint_{\R^n} \theta_{y,t} = 1$, and then
\begin{equation} \label{defwtBbbb}
\wt{\mathcal B}(y,t) := \iint_{\R^n} \wt{\mathcal A} \, \theta_{y,Nt} \, dz\, ds.
\end{equation}
for a large $N$ to be fixed later to ensure that \eqref{defJacrho} below is invertible. Since $\wt{\mathcal B}$ is some average of $\wt{\mathcal A}$, then 
\begin{equation} \label{BellipasA}
\text{$\wt{\mathcal B}$ is elliptic and bounded with the same constant $C_{\A}$ as $\wt{\mathcal A}$ and $\mathcal A$.}
\end{equation}
The construction is similar to the one done in Lemma \ref{LellipB=ellipA}, so we do not write the details again. Observe also that 
\begin{equation} \label{NBissmall2}
|t \nabla_y \wt{\mathcal B}(y,t)| \lesssim \frac1{N} \|\wt{\mathcal A}\|_\infty \quad \text{ and } \quad  |t\, \partial_t \wt{\mathcal B}(y,t)| \lesssim \|\wt{\mathcal A}\|_\infty.
\end{equation}
In addition, we have that
\[|\nabla \wt{\mathcal B}(y,t)| \lesssim t^{-n} \iint_{B_{Nt/10}(y,Nt)} \Big( |\nabla \mathcal B| \phi + |\nabla \phi| + \frac1t |\mathcal C| \phi \Big) dz\, ds,\]
and if $\wt {\mathcal C}$ denotes $(\mathcal A - \wt{\mathcal B}) \1_{\supp \phi}$, the Poincar\'e inequality entails that
\begin{multline*}\int_{\Delta(x,t)} \int_{t}^{3t} |\wt{\mathcal C}(z,s)|^2 \frac{ds}{s}\, dz  
\\ \lesssim \int_{\Delta(x,2Nt)} \int_{t}^{9Nt} \Big( s^2 |\nabla \mathcal B|^2 \phi^2 +  |\mathcal C|^2 \phi^2 + s^2 |\nabla \phi|^2 + |\1_{(\supp \phi) \cap \supp (1-\phi)}|^2 \Big) \frac{ds}{s}\, dz,
\end{multline*}
which means that $t|\nabla \wt{\mathcal B}| + |\wt{\mathcal C}| \in CM_{\R^n_+}$ by \eqref{NB+CareCM}, and \eqref{1phiCM}. 

\medskip

\noindent {\bf Step 3: The change of variable.} We write $\wt{\mathcal B}$ as the block matrix
\begin{equation} \label{defBi} 
\wt{\mathcal B} = \begin{pmatrix} B_1 & B_2 \\ B_3 & b \end{pmatrix},
\end{equation}
where $b$ is the scalar function $\wt{\mathcal B}_{n,n}$, so $B_1$ is a matrix of order $n-1$, $B_2$ and $B_3$ are respectively a vertical and a horizontal vector of length $n-1$. We use $v$ for the horizontal vector $v = - (B_2)^T/b$, and we define 
\begin{equation} \label{defrho12}
\rho(y,t) := (y+t v(y,t), t), 
\end{equation}
which is a Lipschitz map from $\R^n_+$ to $\R^n_+$ (since $v$ and $t |\nabla v|$ are uniformly bounded, see \eqref{BellipasA} and \eqref{NBissmall2}), and we compute its Jacobian
\begin{equation} \label{defJacrho}
Jac_\rho := \begin{pmatrix} I + t\nabla_y v & 0 \\ v+t\partial_t v & 1 \end{pmatrix}.
\end{equation}
We can choose $N$ big enough in \eqref{NBissmall2} such that $Jac_\rho$ is invertible and even $\det(Jac_\rho) \geq 1/2$. Let $J_\rho$ be the matrix 
\begin{equation} \label{defJrho}
J_\rho := \begin{pmatrix} I & 0 \\ v & 1 \end{pmatrix}.
\end{equation}
We easily have that 
\begin{equation} \label{prJrho}
|Jac_\rho - J_\rho| + |\det(Jac_{\rho})^{-1} - 1| \lesssim |t\nabla v| \lesssim |t\nabla \wt {\mathcal B}| . 
\end{equation}

We aim to use $\rho$ for a change of variable. If $u$ is a weak solution to $L=-\diver \mathcal A \nabla$, then $u\circ \rho^{-1}$ is solution to $L_\rho = -\diver (\mathcal A_\rho \circ \rho^{-1}) \nabla$ where 
\begin{equation}
\mathcal A_\rho = \det(Jac_\rho)^{-1} (Jac_\rho)^T \mathcal A \, Jac_\rho.
\end{equation}
We want to compute $\mathcal A_\rho$. To lighten the notation, we write $\mathcal O_{CM}$ for a scalar function, a vector, or a matrix which satisfies the Carleson measure condition with respect to $\R^{n}_+$, i.e. $\mathcal O_{CM}$ can change from one line to another as long as $\mathcal O_{CM} \in CM_{\R^n_+}$. So \eqref{prJrho} becomes
\begin{equation} \label{prJrho2}
Jac_\rho = J_\rho + \mathcal O_{CM} \quad \text{ and } \quad \det(Jac_\rho)^{-1} = 1  + \mathcal O_{CM}. 
\end{equation}
Remember that by construction, the matrix $\mathcal A$ equals $\wt{\mathcal B} + \wt{\mathcal C} = \wt{\mathcal B} + \mathcal O_{CM}$ on $\supp \phi$, and that $\Jac_\rho$ and $\A$ are uniformly bounded, so
\begin{equation}\label{eqBrho}
    \begin{split}
(\1_{\supp \phi}) \mathcal A_\rho 
& = \1_{\supp \phi} \begin{pmatrix} I & v^T \\ 0 & 1 \end{pmatrix} \begin{pmatrix} B_1 & B_2 \\ B_3 & b \end{pmatrix} \begin{pmatrix} I & 0 \\ v & 1 \end{pmatrix} +  \mathcal O_{CM} \\
& = \1_{\supp \phi} \begin{pmatrix} B_1 + v^TB_3 + B_2v + bvv^T & B_2+bv^T \\ B_3 +bv & b \end{pmatrix} +  \mathcal O_{CM} \\
& = \1_{\supp \phi} \underbrace{\begin{pmatrix} b(B_1 + v^TB_3 + B_2v + bvv^T) & 0 \\ B_3-(B_2)^T & b \end{pmatrix}}_{:=\mathcal B_\rho} +  \mathcal O_{CM} 
\end{split}
\end{equation}
with our choices of $v$. We write $\mathcal C_\rho$ for $(\mathcal A_\rho - \mathcal B_\rho)\1_{\supp \phi} = \mathcal O_{CM}$. The matrices $\mathcal B_\rho \circ \rho^{-1}$ and $\mathcal C_\rho \circ \rho^{-1}$ satisfy \eqref{NB+CareCM} (because the Carleson measure condition is stable under bi-Lipschitz transformations) and $\mathcal B_\rho \circ \rho^{-1}$ has the structure \eqref{prBinS1} as in Step 1. So Step 1 gives that 
\begin{equation} \label{flat4}
\iint_{\R^{n}_+} s \left| \nabla \ln\Big( \frac{u\circ \rho^{-1}}{s} \Big) \right|^2 \phi^2\circ \rho^{-1} \, ds\, dz \lesssim r^{n-1}.\end{equation}
If $s$ (and $t$) is also used, by notation abuse, for the projection on the last coordinate, then
\[\begin{split}
\iint_{\R^{n}_+} t \left| \nabla \ln\Big( \frac{u}{t} \Big) \right|^2 \phi^2  \, dt\, dy 
 & = \iint_{\R^{n}_+} t \left| \frac{\nabla u}{u} - \frac{\nabla t}{t} \right|^2 \phi^2 \, dt\, dy \\
& = \iint_{\R^{n}_+} t \left| \frac{Jac_\rho \nabla (u\circ \rho^{-1}) \circ \rho}{u} - \frac{\nabla t}{t} \right|^2 \phi^2 \, dt\, dy \\
&\leq \iint_{\R^{n}_+} t \left| \frac{Jac_\rho \nabla (u\circ \rho^{-1}) \circ \rho}{u} - \frac{Jac_\rho(\nabla s)\circ \rho}{s\circ \rho} \right|^2 \phi^2 \, dt\, dy \\
& \hspace{4.5cm}
+ \iint_{\R^{n}_+} t \left| \frac{Jac_\rho(\nabla s)\circ \rho}{s\circ \rho} - \frac{\nabla t}{t} \right|^2 \phi^2 \, dt\, dy \\
& := I_1 + I_2.
\end{split}\]
Yet, $\rho$ is a bi-Lipschitz change of variable, so $Jac_\rho$ and $\det(Jac_\rho)^{-1}$ are uniformly bounded, and we have
\begin{multline}
I_1  \lesssim \iint_{\R^{n}_+} t \left| \frac{\nabla (u\circ \rho^{-1}) \circ \rho}{u} - \frac{(\nabla s)\circ \rho}{s\circ \rho} \right|^2 \phi^2 \, dt\, dy \\
 \lesssim  \iint_{\R^{n}_+} s \left|\frac{\nabla (u\circ \rho^{-1})}{u \circ \rho^{-1}} - \frac{\nabla s}{s}\right|^2 \phi^2\circ \rho^{-1} \, ds\, dz \\
 = \iint_{\R^{n}_+} s \left| \nabla \ln\Big( \frac{u\circ \rho^{-1}}{s} \Big) \right|^2 \phi^2\circ \rho^{-1} \, ds\, dz \lesssim r^{n-1}
\end{multline}
by \eqref{flat4}. As for $I_2$, we simply observe that $s\circ \rho = t$ and 
\[ Jac_\rho(\nabla s)\circ \rho = \nabla t\]
to deduce that $I_2 = 0$.
The lemma follows.
\ep

\section{Proof of Theorem \ref{Main4}}\label{SecPfofThm4}
In this section we prove Theorem \ref{Main4}, using the same strategy as our proof of Theorem \ref{Main1}. As mentioned in the introduction, we shall explain how to change the 5-step sketch of proof given in Subsection~\ref{Sproof} to prove Theorem \ref{Main4}.

Fix a bounded solution $u$ of $Lu=0$ in $\om$ with $\norm{u}_{L^\infty(\om)}\le 1$ and a ball $B=B(x_0,r)$ centered on $\pom$ with radius $r$. By the same argument as Step 1 in in Subection~\ref{Sproof}, it suffices to show that there exists some constant $C\in(0,\infty)$ depending only on $n$, $M$ and the UR constants of $\pom$, such that 
\begin{equation}
    I := \sum_{Q\in\D_{\pom}(Q_0)}\iint_{W_\om(Q)}\abs{\nabla u(X)}^2\delta(X)dX\le C\sigma(Q_0)
\end{equation}
for any cube $Q_0\in\D_{\pom}$ that satisfies $Q_0\subset\frac87B\cap\pom$ and $\ell(Q_0)\le 2^{-8}r$. 

Then observe that if $E\subset\om$ is a Whitney region, that is, $E\subset\frac74 B$ and $\diam(E)\le K\delta(E)$, then 
\begin{equation}\label{eqWhitney}
    \iint_E\abs{\nabla u}^2\delta\, dX\le C_K\diam(E)^{-1}\iint_{E^*}\abs{u}^2dX\le C_K\delta(E)^{n-1},
\end{equation}
by the Caccioppoli inequality and $\norm{u}_{L^\infty(\om)}\le 1$, where $E^*$ is an enlargement of $E$. This bound \eqref{eqWhitney} is the analogue of \eqref{Main1e}, and proves Step 2.

Step 3 is not modified. We pick $0 <\epsilon_1\ll \epsilon_1 \ll 1$ and  we use the corona decomposition constructed in Section~\ref{SUR} to decompose $I$ as follows.
\[
    I=\sum_{Q\in\B(Q_0)}\iint_{W_\om(Q)}\abs{\nabla u}^2\delta\, dX+\sum_{\cS\in \mathfrak S(Q_0)}\iint_{W_\om(\cS)}\abs{\nabla u}^2\delta\, dX=: I_1+\sum_{\cS\in \mathfrak S(Q_0)} I_\cS.
\]
By \eqref{eqWhitney} and \eqref{packingBS}, 
\[
I_1\le C\sum_{Q\in\B(Q_0)}\ell(Q)^{n-1}\le C\sigma(Q_0).
\]

 Step 4 is significantly simpler for Theorem \ref{Main4}, because we do not need any estimate on the smooth distance $D_\beta$, but the spirit is the same. That is, by using the bi-Lipschitz map $\rho_\cS$ constructed in Section~\ref{Srho}, $I_\cS$ can be turned into an integral on $\R^n \setminus \R^{n-1}$, which can be estimated by an integration by parts argument. More precisely, for any fixed $\cS\in\mathfrak S(Q_0)$, 
\begin{multline*}
    I_\cS=\iint_{\rho_\cS^{-1}(W_\om(\cS))}\abs{(\nabla u)\circ\rho_\cS(p,t)}^2\delta\circ\rho_\cS(p,t)\det\Jac(p,t)dpdt\\
    \le 2\iint\abs{\nabla( u\circ\rho_\cS(p,t))}^2\dist(\rho_\cS(p,t),\Gamma_{\cS})\br{\Psi_{\cS}\circ\rho_{\cS}(p,t)}^2dpdt\\
    \le 3\iint\abs{\nabla v(p,t)}^2\abs{t}\phi(p,t)^2dpdt, \qquad v=u\circ\rho_\cS, \quad\phi=\Psi_\cS\circ\rho_\cS
\end{multline*}
by \eqref{claimdGPib}, Lemmata \ref{lemWOG} (d) and \ref{LestonJ}, as well as \eqref{prrhoS1}, for $\epsilon_0$ sufficiently small. 

 The fifth step consists roughly in proving the result in $\R^n \setminus \R^{n-1}$. The function $\phi$ is  the same as the one used to prove Theorem \ref{Main1}, in particular it is a cutoff function associated to both $\rho_\cS^{-1}(\pom)$ and $\R^{n-1}$ as defined in Definition \ref{defcutoffboth}, and it satisfies
\begin{equation}\label{eqsuppphi}
    \supp \phi\subset \rho_{\cS}^{-1}(W_\om^*(\cS)) ,
\end{equation}
and
\begin{equation}\label{eqphi}
    \iint\abs{\nabla\phi}dtdp+\iint\abs{\nabla\phi}^2tdtdp\lesssim\sigma(Q(\cS)),
\end{equation}
where the implicit constant depends on $n$ and the AR constant in \eqref{defADR}. 
 Notice that $v=u\circ\rho_{\cS}$ is a bounded solution of $L_{\cS}=-\diver\A_\cS\nabla$ that satisfies $\norm{v}_{L^\infty}\le1$, where $\A_\cS$ is defined in \eqref{defAS}. By Lemma \ref{LprAS}, $I_\cS\le C\sigma(Q(\cS))$ will follow from the following lemma, which is essentially a result in $\R^n \setminus \R^{n-1}$. 
 
 \begin{lemma}\label{LCarlesonflat}
 Let $L=-\diver \mathcal A \nabla$ be a uniformly elliptic operator on $\om_{\cS}:=\rho^{-1}_{\cS}(\Omega)$. Assume that the coefficients $\A$ can be decomposed as $\mathcal A  = \mathcal B + \mathcal C$ where
\begin{equation} \label{B+CCMflat}
 (|t\nabla \mathcal B| + |\mathcal C|)\1_{\supp \phi} \in CM_{\R^n \sm \R^{n-1}}(M),
\end{equation}
where $\phi=\Psi_\cS\circ\rho_{\cS}$ is as above. 
Then for any solution $v$ of $Lv=0$ in $\rho^{-1}_{\cS}(\Omega)$ that satisfies $\norm{v}_{L^\infty}\le 1$, there holds
\begin{equation} \label{vflat}
\iint_{ \R^n \setminus \R^{n-1}} \abs{\nabla v}^2\phi^2 { |t|}\,dtdy\leq C(1+M) \sigma(Q(\cS)),
\end{equation}
where $C$ depends only on the dimension $n$, the elliptic constant $C_\A$, the AR constant of $\pom$, and the implicit constant in \eqref{eqphi}.
 \end{lemma}
 The proof of this lemma is similar to the proof of Lemma \ref{lemflat}, except that there is no need to invoke the CFMS estimates and $A_\infty$ as in Lemma \ref{lemlogk}, essentially because $v$ is bounded and we do not need information of $v$ on the boundary. For the same reason, with the properties of the cutoff function $\phi$ in mind, we can forget about the domain $\om_\cS$, and in particular, we do not need the corkscrew and Harnack chain conditions in the proof.
 
 \begin{proof}[Proof of Lemma \ref{LCarlesonflat}]
 We can decompose $\phi = \phi \,\1_{t>0} + \phi \,\1_{t<0} := \phi_+ + \phi_-$ and prove the result for each of the functions $\phi_+$ and $\phi_- $, and since the proof is the same in both cases (up to a sign), we can restrain ourselves as in the proof of Lemma \ref{lemflat} to the case where $\phi = \phi \1_{t>0}$. By an approximation argument as in Step 1 of the proof of Lemma \ref{lemflat}, we can assume that $T:=\iint_{\R^n_+}\abs{\nabla v}^2t\phi^2dydt$ is finite, and that $\phi$ is compactly supported in $\om\cap\R^n_+$. We first assume that $\B$ has the special structure that 
 \begin{equation}\label{Bni}
     \B_{n i}=0\qquad \text{for all } 1\le i\le n-1, \qquad \B_{nn}=b.
 \end{equation}
 Then for any $f\in W_0^{1,2}(\R^n_+)$, 
 \begin{equation}\label{eqtsol}
     \iint\frac{\B}{b}\nabla f\cdot\nabla t\,dydt=\iint \partial_t f\,dydt=0.
 \end{equation}
Using ellipticity of $\A$ and boundeness of $b$, we write
\begin{multline*}
    T\le C_{\A}^2\iint\frac{\A}{b}\nabla v\cdot\nabla v\,\phi^2t\,dydt\\
    =C_{\A}^2\Big\{\iint \A\nabla v\cdot\nabla\br{v\phi^2 b^{-1}t}dydt-\iint \A\nabla v\cdot\nabla\br{\phi^2b^{-1}}vt\,dydt-\iint \A\nabla v\cdot\nabla t\,v\phi^2b^{-1}dydt\Big\}\\
    =-C_\A^2\Big\{\iint \A\nabla v\cdot\nabla\br{\phi^2b^{-1}}vt\,dydt+\iint \A\nabla v\cdot\nabla t\,v\phi^2b^{-1}dydt\Big\}=: -C_\A^2\br{T_1+T_2}
\end{multline*}
since $Lv=0$. We write $T_1$ as 
\[
    T_1=2\iint \A\nabla v\cdot\nabla\phi\, \phi b^{-1}vt \,dydt-\iint\A \nabla v\cdot\nabla b \,\phi^2 b^{-2}vt\,dydt=:T_{11}-T_{12}.
\]
By Cauchy-Schwarz and Young's inequalities, as well as the boundedness of $v$ and $b$,
\[
\abs{T_{11}}\le \frac{C_\A^{-2}}{6} T+ C\iint\abs{\nabla\phi}^2t\,dydt, \qquad \abs{T_{12}}\le \frac{C_\A^{-2}}8 T+ C\iint\abs{\nabla b}^2t\phi^2dtdy.
\]
So \eqref{eqphi} and \eqref{B+CCMflat}, as well as \eqref{eqsuppphi} give that 
\begin{equation*}
    \abs{T_1}\le \frac{C_\A^2}{4} T + C\ell(Q(\cS))^{n-1}.
\end{equation*}
For $T_2$, we write
\begin{multline*}
    T_2=\frac12\iint\frac{\A}{b}\nabla\br{v^2\phi^2}\cdot\nabla t\,dydt -\iint \frac{\A}{b}\nabla\phi\cdot\nabla t\,v^2\phi\,dydt\\
    =\frac12\iint\frac{\C}{b}\nabla\br{v^2\phi^2}\cdot\nabla t\,dydt -\iint \frac{\A}{b}\nabla\phi\cdot\nabla t\,v^2\phi\,dydt=:T_{21}+T_{22}
\end{multline*}
by writing $\A=\B+\C$ and applying \eqref{eqtsol}. For $T_{21}$, we use Cauchy-Schwarz and Young's inequalities, and get
\begin{multline*}
    \abs{T_{21}}\le\abs{\iint \C\nabla v\cdot\nabla t\,v\phi^2b^{-1}dydt}+2\abs{\iint \C\nabla\phi\cdot\nabla t\,v^2\phi b^{-1}dydt}\\
    \le \frac{C_\A^{-2}}{4}T+ C\iint\abs{\C}^2\phi^2t^{-1}dydt +\iint\abs{\nabla\phi^2}tdydt\le \frac{C_\A^{-2}}{4}T+ C\ell(Q(\cS))^{n-1}
\end{multline*}
by the boundedness of $v$, \eqref{eqphi}, \eqref{B+CCMflat}, and \eqref{eqsuppphi}. The boundedness of the coefficients and $v$ implies that \[
\abs{T_{22}}\le C\iint\abs{\nabla\phi}dydt\le C\ell(Q(\cS))^{n-1}
\]
by \eqref{eqphi}. Altogether, we have obtained that
$T\le \frac12 T+ C\ell(Q(\cS))^{n-1}$, and thus the desired estimate \eqref{vflat} follows. 

We claim that the lemma reduces to the case when \eqref{Bni} holds by almost the same argument as in Steps 2 and 3 in the proof of Lemma \ref{lemflat}. That is, we can assume that $\norm{\abs{\nabla_y\B}t}_\infty\lesssim\frac{C_\A}{N}$ with $N$ to be chosen to be sufficiently large, and then we do a change of variables, which produces the structure \eqref{Bni} in the conjugate operator. The only difference from the proof of Lemma \ref{lemflat} is that now we need to choose $v=- B_3/b$ in the bi-Lipschitz map $\rho$ defined in \eqref{defrho12} because we want $B_3+bv=0$ in \eqref{eqBrho}. We leave the details to the reader. 
 \end{proof}

\section{The converse} \label{Sconv}
In this section, we show 
that $(v) \implies (i)$ in Theorem \ref{Main1}, that is, we establish that under certain conditions on the domain $\om$ and the operator $L$, the Carleson condition \eqref{Main2b} on the Green function implies that $\pom$ is uniformly rectifiable. More precisely, we prove the following. 

\begin{theorem}\label{thm.conv}
Let $\Omega$ be a 1-sided Chord-Arc Domain (bounded or unbounded) and let $L=-\div \A \nabla$ be an uniformly elliptic operator which satisfies the weak DKP condition with constant $M\in(0,\infty)$ on $\om$. Let $X_0\in\om$, and when $\om$ is unbounded, $X_0$ can be $\infty$. We write $G^{X_0}$ for the Green funtion of $L$ with pole at $X_0$. Suppose that there exists $C\in(0,\infty)$ and $\beta>0$ such that for all balls $B$ centered at the boundary and such that $X_0 \notin 2B$, we have
\begin{equation} \label{Main1b'}
\iint_{\Omega \cap B} \left| \frac{\nabla G^{X_0}}{G^{X_0}} - \frac{\nabla D_\beta}{D_\beta} \right|^2 D_\beta \, dX \leq C \sigma(B\cap\pom).
\end{equation}
Then $\pom$ is uniformly rectifiable.
\end{theorem}
In \cite{DM2} Theorem 7.1, uniform rectifiability is obtained from some weak condition on the Green function, namely, $G^\infty$ being prevalently close to $D_\beta$. Following \cite{DM2}, we say that $G^\infty$ is prevalently close to $D_\beta$ if for each choice of $\epsilon>0$ and $M\ge 1$, the set $\mathcal{G}_{GD_\beta}(\epsilon,M)$ of pairs $(x,r)\in \pom\times(0,\infty)$ such that there exists a positive constant $c>0$, with 
\[
\abs{D_\beta(X)-c\,G^\infty(X)}\le\epsilon r \quad \text{for }X\in\om\cap B(x,Mr),
\]
is {\it Carleson-prevalent}.
\begin{definition}[Carleson-prevalent]
We say that $\mathcal{G}\subset\pom\times(0,\infty)$ is a Carleson-prevalent set if there exists a constant $C\ge0$ such that for every $x\in\pom$ and $r>0$,
\[
\int_{y\in\pom\cap B(x,r)}\int_{0<t<r}\1_{\mathcal{G}^c}(y,t)\frac{d\sigma(y)dt}{t}\le C\,r^{n-1}.
\]
\end{definition}

One could say that our condition \eqref{Main1b'} is stronger than $G^\infty$ being prevalently close to $D_\beta$, and so the theorem follows from \cite{DM2}. But actually it is not so easy to link the two conditions directly. Nonetheless, we can use Chebyshev's inequality to derive a weak condition from \eqref{Main1b'}, which can be used as a replacement of $G^\infty$ being prevalently close to $D_\beta$ in the proof.

We will soon see that the condition on the operator in Theorem \ref{thm.conv} can be relaxed. 
Again following \cite{DM2}, given an elliptic operator $L=-\div{\A \nabla}$, we say that $L$ is {\it locally sufficiently close to a constant coefficient elliptic operator} if for every choice of $\tau>0$ and $K\ge 1$, $\mathcal{G}_{cc}(\tau,K)$ is a Carleson prevalent set, where $\mathcal{G}_{cc}(\tau,K)$ is the set of pairs $(x,r)\in\pom\times(0,\infty)$ such that there is a constant matrix $\A_0=\A_0(x,r)$ such that 
\[
\iint_{X\in W_K(x,r)}\abs{\A(X)-\A_0}dX\le\tau r^n,
\]
where 
\begin{equation}\label{Wkxr}
    W_K(x,r)=\set{X\in\om\cap B(x, Kr): \dist(X,\pom)\ge K^{-1}r}.
\end{equation}
We will actually prove Theorem \ref{thm.conv} for elliptic operators $L$ that are sufficiently close locally to a constant coefficient elliptic operator.

The first step of deriving weak conditions from the strong conditions on the operator and $G^\infty$ is the observation that for any integrable function $F$, if there is a constant $C\in(0,\infty)$ such that
\[
\iint_{B(x,r)\cap\om}\abs{F(Y)}dY\le C\,r^{n-1} \quad \text{for  }x\in\pom, r>0,
\]
then for any $K\ge1$,
\begin{equation}\label{obs.avg}
   \int_{y\in B(x,r)\cap\pom}\int_{0<t<r}\fiint_{W_K(y,t)}\abs{F(Y)}dY\,dt\,d\sigma(y)\le C\,K^{n-1}r^{n-1}  
\end{equation}
for $x\in\pom$, $r>0$. This follows immediately from Fubini's theorem and the fact that $W_K(x,r)$ defined in \eqref{Wkxr} is a Whitney region which is away from the boundary. 

\begin{lemma}\label{lem.Carlprev}
\begin{enumerate}
    \item Let $L=-\div \A \nabla$ be a uniformly elliptic operator which satisfies the weak DKP condition with constant $M\in(0,\infty)$ on $\om$. Then $L$ is locally sufficiently close to a constant coefficient elliptic operator.
    \item If $G^{X_0}$ satisfies \eqref{Main1b'} for all $B$ centered at the boundary and such that $X_0 \notin 2B$, then for every choice of $\epsilon>0$ and $K\ge1$, the set \begin{multline}\label{def.setG}
        \mathcal{G}^{X_0}(\epsilon,K):=  \Big\{(x,r)\in\pom\times(0,\infty): \, X_0 \notin B(x,2Kr) \text{ and } \\ \iint_{W_K(x,r)}\abs{\nabla\ln\br{\frac{G^{X_0}}{D_\beta}(X)}}^2D_\beta(X)dX\le\epsilon\,r^{n-1} \Big\}
    \end{multline}
    is Carleson-prevalent. 
\end{enumerate}
\end{lemma}

\begin{proof}
Both results follow from the previous observation \eqref{obs.avg} and Chebyshev's inequality. In fact, for (1), we have $\A=\B+\C$ such that for any $x\in\pom$ and $r>0$,
\begin{equation}\label{eq9a1}
   \int_{y\in B(x,r)\cap\om}\int_{0<t<r}\fiint_{W_K(y,t)}\br{\abs{\nabla\B}^2\delta+\abs{\C}^2\frac{1}{\delta}}dY\,dt\,d\sigma(y)\le M\,K^{n-1}r^{n-1}.  
\end{equation}
By the Poincar\'e inequality, the left-hand side is bounded from below by 
\begin{multline*}
     c\int_{y\in B(x,r)\cap\om}\int_{0<t<r}\fiint_{W_K(y,t)}\br{\abs{\B-(\B)_{W_K(y,t)}}^2+\abs{\C}^2}dY(Kt)^{-1}dt\,d\sigma(y)\\
     \ge\frac{c}2\int_{y\in B(x,r)\cap\om}\int_{0<t<r}\fiint_{W_K(y,t)}\br{\abs{\A-(\B)_{W_K(y,t)}}^2}dY(Kt)^{-1}dt\,d\sigma(y)\\
     \ge \frac{c}2\frac{\tau}{K^{n+1}}\int_{y\in B(x,r)\cap\om}\int_{0<t<r}\1_{\mathcal{G}_{cc}(\tau,K)^c}(y,t)\frac{dt\,d\sigma(y)}{t},
\end{multline*}
where we have used the fact that $(\B)_{W_K(y,t)}$ is a constant matrix and the definition of the set $\mathcal{G}_{cc}(\tau,K)$. Combining with \eqref{eq9a1}, we have that
\[
\int_{y\in B(x,r)\cap\om}\int_{0<t<r}\1_{\mathcal{G}_{cc}(\tau,K)^c}(y,t)\frac{dt\,d\sigma(y)}{t}\le \frac{CMK^{2n}}{\tau}r^{n-1},
\]
which proves (1).

Now we justify (2). Let $\epsilon>0$, $K\ge1$ and $X_0\in\om$ be fixed, and let $B$ be a ball of radius $r$ centered at the boundary. 
Our goal is to show that
\[
    \int_{y\in B\cap\pom}\int_{0<t<r}\1_{\mathcal{G}^{X_0}(\epsilon,K)^c}(y,t)\frac{dt\,d\sigma(y)}{t}\le C_{\epsilon,K}\sigma(B\cap \pom).
\]
We discuss two cases. If $X_0 \notin 4KB$, then since $G^{X_0}$ satisfies \eqref{Main1b'} for the ball $2KB$, we have that 
\[
\int_{y\in B\cap\pom}\int_{0<t<r}\fiint_{W_K(y,t)}\abs{\nabla\ln\br{\frac{G^{X_0}}{D_\beta}}(Y)}^2D_\beta(Y)dY dtd\sigma(y)\le CK^{n-1}r^{n-1}.
\]
Notice that the assumption $X_0 \notin 4KB$ guarantees that  $X_0\notin B(y,2Kt)$ for all $y\in B\cap\pom$ and $0<t<r$. Therefore, if $(y,t)\in \mathcal{G}^{X_0}(\epsilon,K)^c$, then  
\[\iint_{W_K(y,t)}\abs{\nabla\ln\br{\frac{G^{X_0}}{D_\beta}}(Y)}^2D_\beta(Y)dY>\epsilon\,r^{n-1}.\]
From this, it follows that  
\begin{equation} \label{case1GeK}
\int_{y\in B\cap\pom}\int_{0<t<r}\1_{\mathcal{G}^{X_0}(\epsilon,K)^c}(y,t)\frac{dt\,d\sigma(y)}{t}\le \frac{CK^{2n-1}}{\epsilon}\sigma(B\cap \pom).
\end{equation}
Now let us deal with the case where $X_0 \in 4KB$. For $x\in B \cap \partial \Omega$, we define $B_x:= B(x,|x-X_0|/20K)$. Since $\{B_x\}_{x\in B}$ covers $B \cap \partial \Omega$, we can find a non-overlapping subcollection $\{B_{i}\}_{i\in I}$ such that $\{5B_{i}\}_{i\in I}$ covers $B \cap \partial \Omega$. We write $r_i>0$ for the radius of $B_i$ and we define
\[S:= (B\cap \partial \Omega) \times (0,r) \setminus \bigcup_{i\in I} (5B_i\cap \partial \Omega) \times (0,5r_i) \]
We have 
\begin{multline*}
    \int_{y\in B\cap\pom}\int_{0<t<r}\1_{\mathcal{G}^{X_0}(\epsilon,K)^c}(y,t)\frac{dt\,d\sigma(y)}{t} \leq \sum_{i\in I} \int_{y\in 5B_i\cap\pom}\int_{0<t<5r_i}\1_{\mathcal{G}^{X_0}(\epsilon,K)^c}(y,t)\frac{dt\,d\sigma(y)}{t} \\
    + \iint_{S} \1_{\mathcal{G}^{X_0}(\epsilon,K)^c}(y,t)\frac{dt\,d\sigma(y)}{t} =: T_1 + T_2.
\end{multline*}
Since $X_0 \notin 20KB_i$, we can apply \eqref{case1GeK}, and we have
\[
    T_1 \leq C_{K,\epsilon} \sum_{i\in I} \sigma(5B_i \cap \partial \Omega) \lesssim \sum_{i\in I} \sigma(B_i \cap \partial \Omega) \leq \sigma(2B \cap \partial \Omega) \lesssim \sigma(B \cap \partial \Omega),
\]
because $\{B_i\}$ is a non-overlapping and included in $2B$. It remains to prove a similar bound on $T_2$. Remark first that 
\[S \subset \{(y,t) \in \partial\Omega \times (0,r): \, |y-X_0|/100K < t\},\]
and therefore
\[
    T_2 \leq \int_{0}^r \int_{y\in B(X_0,100Kt) \cap \partial \Omega}  \frac{d\sigma(y)\, dt}{t} \leq C K^{n-1} r^{n-1} \lesssim \sigma(B \cap \partial \Omega).
\]
The lemma follows.
\end{proof}

Before we continue, we need to adapt Theorem 2.19 in \cite{DM2} to our situation, that is we want to construct a positive solution in a domain which is the limit of a sequence of domain.

\begin{lemma}\label{lem.cvg}
Let $\Omega_k$ be a sequence of 1-sided Chord-Arc domains domains in $\R^n$ with uniform 1-sided CAD constants. Let $\partial \Omega_k$ be its Ahlfors regular boundary equipped with a Ahlfors regular measure $\sigma_k$ (such that the constant in \eqref{defADR} is uniform in $k$).

Assume that $0\in \partial \Omega_k$ and $\diam \Omega_k \geq 2^k$. Moreover, assume that the $\partial \Omega_k$ and $\Omega_k$ converges to $E_\infty$ and $\Omega_\infty$ locally in  the Hausdorff distance, that is, for any $j\in \mathbb N$, we have
\[\lim_{k\to \infty} d_{0,2^j}(E_\infty,\partial \Omega_k) = 0 \text{ and } \lim_{k\to \infty} d_{0,2^j}(\Omega_\infty,\Omega_k) = 0.\]
Here, for a couple of sets $(E,F)$, we define the Hausdorff distance
\[d_{0,2^j}(E,F):= \sup_{x\in E \cap B(0,2^j)} \dist(x,F) + \sup_{y\in F \cap B(0,2^j)} \dist(y,E).\]
Then $E_\infty = \partial \Omega_\infty$, $E_\infty$ is an unbounded $(n-1)$-Ahlfors regular set, $\Omega_\infty$ is a 1-sided Chord-Arc Domain. Moreover, if the Radon measure $\sigma$ is any weak-* limit of the $\sigma_k$, then $\sigma$ is an Ahlfors regular measure on $E_\infty = \partial \Omega_\infty$.

\medskip

Let $Y_0$ be a corkscrew point of $\Omega_\infty$ for the boundary point $0$ at the scale 1. If $L_k = - \diver A_k \nabla$ and $L_\infty = -\diver A_\infty \nabla$ are operators - in $\Omega_k$ and $\Omega_\infty$ respectively - that satisfies
\[\lim_{k\to \infty} \|A_k - A_\infty\|_{L^1(B)} = 0 \qquad \text{ for any ball $B$ such that $2B \subset \Omega_\infty$},\]
and if $u_k$ are positive solutions in $\Omega_k \cap B(0,2^{k+1})$ to $L_k u_k = 0$ with $\Tr u_k=0$ on $\pom_k\cap B(0,2^{k+1})$, then the sequence of functions $v_k:= u_k/u_k(Y_0)$ converges, uniformly on every compact subset of $\om_\infty$, and in $W^{1,2}_{\loc}(\om_\infty)$, to $G^\infty$, the unique Green function with pole at infinity which verifies $G^\infty(Y_0) = 1$. 
\end{lemma}

\begin{proof}
The geometric properties of $E_\infty$ and $\om_\infty$ can be derived verbatim as in the proof of Theorem 2.19 in \cite{DM2}. The uniform convergence of a subsequence of $v_k$ on any compact set $K\Subset\om_\infty$ follows from the standard argument of uniform boundedness of $\set{v_k}$ on $K$, and H\"older continuity of solutions. The Caccioppoli inequality would give the weak convergence of another subsequence of $v_k$ to some $v_\infty$ in $W_
{\loc}^{1,2}(\om_\infty)$. This is enough to show that $v_\infty\in W_{\loc}^{1,2}(\om_\infty)\cap C(\overline{\om_\infty})$ is a weak solution of $L_\infty v_\infty=0$ in $\om_\infty$, as we can write 
\[
\iint_{\om_\infty}A_\infty\nabla v_\infty\cdot\nabla\vp dX=\iint_{\om_\infty}A_\infty(\nabla v_\infty-\nabla v_k)\cdot\nabla\vp \,dX+\iint_{\om_\infty}(A_\infty-A_k)\nabla v_k\cdot\nabla\vp \,dX
\]
for every $\vp\in C_0^\infty(\om_\infty)$ and any $k$ sufficiently big so that $\supp\vp\subset\om_k\cap B(0,2^{k+1})$. Therefore, $v_\infty=G^\infty$ is the Green function with pole at infinity for $L_\infty$ in $\om_\infty$ and normalized so that $G^\infty(Y_0)=1$.

That $v_k$ converges to $G^\infty$ (strongly) in $W_{\loc}^{1,2}(\om_\infty)$ needs more work, but we can directly copy the proof of Lemma 2.29 in \cite{DM2}. Roughly speaking, for any fixed ball $B$ with $4B\subset\om$, we would need to introduce an intermediate function $V_k$, which satisfies $L_kV_k=0$ in $B_\rho$ for some $\rho\in(r,2r)$, and $V_k=v_k$ on the sphere $\partial B_\rho$.  We refer the readers to \cite{DS2} for the details.
\end{proof}

We shall need the following result on compactness of closed sets, which has been proved in \cite{DS3}.
\begin{lemma}[\cite{DS3} Lemma 8.2]{\label{lem.DS3}}
Let $\set{E_j}$ be a sequence of non-empty closed
subsets of $\Rn$, and suppose that there exists an $r>0$ such that $E_j\cap B(0,r)\neq\emptyset$ for all $j$. Then there is a subsequence of $\set{E_j}$ that converges to a nonempty closed subset $E$ of $\Rn$ locally in the Hausdorff distance. 
\end{lemma}

\ms
Now we are ready to prove the main theorem of this section.
\begin{proof}[Proof of Theorem \ref{thm.conv}]
We prove that $\pom$ is uniformly rectifiable by showing that $\om_{\rm ext}$ satisfies the corkscrew condition (see Lemma \ref{lem.UReqv}). Following the proof of Theorem 7.1 in \cite{DM2}, it suffices to show that the set $\mathcal{G}_{CB}(c)$ is Carleson-prevalent for some $c>0$, where $\mathcal{G}_{CB}(c)$ is the set of pairs $(x,r)\in\pom\times(0,\infty)$ such that we can find $Z_1,Z_2\in B(x,r)$, that lie in different connected components of $\Rn\setminus\pom$, and such that $\dist(Z_i,\pom)\ge cr$ for $i=1,2$. To do that, we will rely on the fact that, on 1-sided CAD domains, if the elliptic measure is comparable to the surface measure, then the complement $\om_{\rm ext}$ satisfies the corkscrew condition, which is implied by the main result of \cite{HMMTZ}.

Thanks to Lemma \ref{lem.Carlprev}, for each choice of $\epsilon>0$ and $M\ge1$, the sets $\mathcal{G}^{X_0}(\epsilon,M)$  and $\mathcal{G}_{cc}(\epsilon,M)$ are Carleson-prevalent. So it suffices to show that
\begin{equation}\label{eq9b1}
    \mathcal{G}^{X_0}(\epsilon,M)\cap\mathcal{G}_{cc}(\epsilon,M)\subset\mathcal{G}_{CB}(c) \quad \text{for some }c>0,\epsilon>0, \text{ and }M\ge1.
\end{equation}
We prove by contradiction. Assume that \eqref{eq9b1} is false, then for $c_k=\epsilon_k=M_k^{-1}=2^{-k}$, we can find a 1-sided NTA domain $\om_k$ bounded by an Ahlfors regular set $\pom_k$, a point $X_k\in\om_k$ (or $X_k\in\om_k\cup\set{\infty}$ when $\om$ is unbounded), an elliptic operator $L_k=-\div\A_k\nabla$ that is locally sufficiently close to a constant coefficient elliptic operator, and a pair $(x_k,r_k)\in\pom_k\times(0,\infty)$ for which 
\[
(x_k,r_k)\in  \mathcal{G}^{X_k}(\epsilon_k,M_k)\cap\mathcal{G}_{cc}(\epsilon_k,M_k)\setminus\mathcal{G}_{CB}(c_k).
\]
By translation and dilation invariance, we can assume that $x_k=0$ and $r_k=1$. Notice that $(0,1)\in\mathcal{G}^{X_k}(\epsilon_k,M_k)$ implies that $X_k\notin B(0,2^k)$, and in particular, $\diam(\om_k)\ge 2^k$, and $X_k$ tends to infinity as $k\to\infty$.

By Lemma \ref{lem.DS3}, we can
 extract a subsequence so that $\om_k$ converges to a limit $\om_\infty$. By Lemma \ref{lem.cvg}, $\om_\infty$ is 1-sided NTA, $\pom_k$ converges to $\pom_\infty$ which is Ahlfors regular. Moreover, by Lemma \ref{lem.DS3}, we
can extract a further subsequence so that the Ahlfors regular measure $\sigma_k$ given on $\pom_k$ converges weakly
to an Ahlfors regular measure $\sigma$. 
Since $(0,1)\in\mathcal{G}_{cc}(2^{-k},2^k)$, $\A_k$ converges to some constant matrix $\A_0$ in $L^1_{\loc}(\om_\infty)$.

Choose a corkscrew point $Y_0\in\om_\infty$ for some ball $B_0$ centered on $\pom_\infty$, and let $G_k=G_k^{X_k}$ be the Green function for $L_k$ in $\om_k$, normalized so that $G_k(Y_0)=1$. Since $L_kG_k=0$ in $\om_k\cap B(0,2^k)$, Lemma \ref{lem.cvg} asserts that $G^k$ converges to the Green function $G=G_\infty^\infty$ with pole at infinity for the constant-coefficient operator $L_0=-\div \A_0\nabla$, uniformly on compact sets of $\om_\infty$, and in $W_{\loc}^{1,2}(\om_\infty)$. Since  $\sigma_k\rightharpoonup\sigma$, $D_k=D_{\beta,\sigma_k}$ converges to $D=D_{\beta,\sigma}$ uniformly on compact sets of $\om_\infty$, and so does $\nabla D_k$ to $\nabla D$. Since $(0,1)\in\mathcal{G}^{X_k}(2^{-k},2^k)$,
\begin{equation}\label{eq9b2}
    \iint_{W_{2^k}(0,1)}\abs{\frac{\nabla G_k}{G_k}-\frac{\nabla D_k}{D_k}}^2D_k(X)dX\le 2^{-k} \qquad \text{for all } k\in\mathbb{Z}_+,
\end{equation}
where $W_{2^k}(0,1)$ is the Whitney region defined as in \eqref{Wkxr} for $\om_k$. Fix any compact set $K\Subset\om_\infty$. We claim that
\begin{equation}\label{eq9b3}
    \lim_{k\to\infty} \iint_{K}\abs{\frac{\nabla G_k}{G_k}-\frac{\nabla D_k}{D_k}}^2D_k(X)dX=\iint_{K}\abs{\frac{\nabla G}{G}-\frac{\nabla D}{D}}^2D(X)dX.
\end{equation}
In fact, since $G$ is a positive solution of $L_0 G=0$ in $\om_\infty$ with $G(Y_0)=1$, the Harnack inequality implies that $G\ge c_0$ on $K$ for some $c_0>0$. Then the uniform convergence of $G_k$ to $G$ on $K$ implies that for $k$ large enough, $\set{G_k^{-1}}$ is uniformly bounded on $K$, and so $G_k^{-1}$ converges uniformly to $G^{-1}$ on $K$. Then \eqref{eq9b3} follows from the fact that $\nabla G_k$ converges to $\nabla G$ in $L^2(K)$, the uniform convergence of $G_k^{-1}$ to $G^{-1}$ on $K$, and the uniform convergences of $\nabla D_k$ and $D_k^{-1}$ to $\nabla D$ and $D^{-1}$.

Now by \eqref{eq9b2} and \eqref{eq9b3}, we get that 
\[\iint_K\abs{\nabla\ln\br{\frac{G}{D}}(X)}^2D(X)dX=0,\]
and so $G=CD_{\beta,\sigma}$ in $\om_\infty$. We can copy the proof of Theorem 7.1 of \cite{DM2} verbatim from now on to conclude that this leads to a contradiction. Roughly speaking, $G=CD_{\beta,\sigma}$ would imply that the elliptic measure $\omega^\infty$ for $L_0$, with a pole at $\infty$, is comparable to $\H_{|\pom_\infty}^{n-1}$. Then by \cite{HMMTZ} Theorem 1.6 one can conclude that $\pom_\infty$ is uniformly rectifiable, and hence $\Rn\setminus\overline\om_\infty$ satisfies the corkscrew condition, which contradicts with the assumption that $(0,1)=(x_k,r_k)\notin\mathcal{G}_{CB}(c_k)$. 
\end{proof}

\section{Assuming that $\Omega$ is semi-uniform is not sufficient.} \label{Scount}

In this subsection, we will give an example of domain where the harmonic measure on $\partial \Omega$ is $A_\infty$-absolutely continuous with respect to the $(n-1)$-dimensional Hausdorff measure, but where Theorem \ref{Main1} fails. It is known that the harmonic measure is $A_\infty$-absolute continuous with respect to the surface measure whenever the domain $\Omega$ is semi-uniform and its boundary is $(n-1)$-Ahlfors regular and uniformly rectifiable (see \cite[Theorem III]{Azzam}). The notion of semi-uniform domain is given by the next definition.

\begin{definition}[{\bf Semi-uniform domains}] \label{defSUD}
We say that $\Omega$ is semi-uniform if it satisfies the corkscrew condition and  (see Definition \ref{def1.cork}) if for every $\Lambda\geq 1$, there exists $C_\Lambda>0$ such that for any $\rho>1$ and every pair of points $(X,x) \in \Omega \times \partial \Omega$ such that $\abs{X-x}<\Lambda\rho$, there exists a Harnack chain of length bounded by $C_\Lambda$ linking $X$ to one of the corkscrew points for $x$ at scale $\rho$.
\end{definition}

Semi-uniform domains were first introduced by Aikawa and Hirata in \cite{AiHi} using {\it cigar curves}. The two definitions of semi-uniform domains are known to be equivalent, see for instance, \cite{Azzam} Theorem 2.3.

Our counterexample is constructed in $\R^2$ for simplicity but can easily be extended to any dimension.
\begin{figure}
\begin{center}
\includegraphics[scale=0.8]{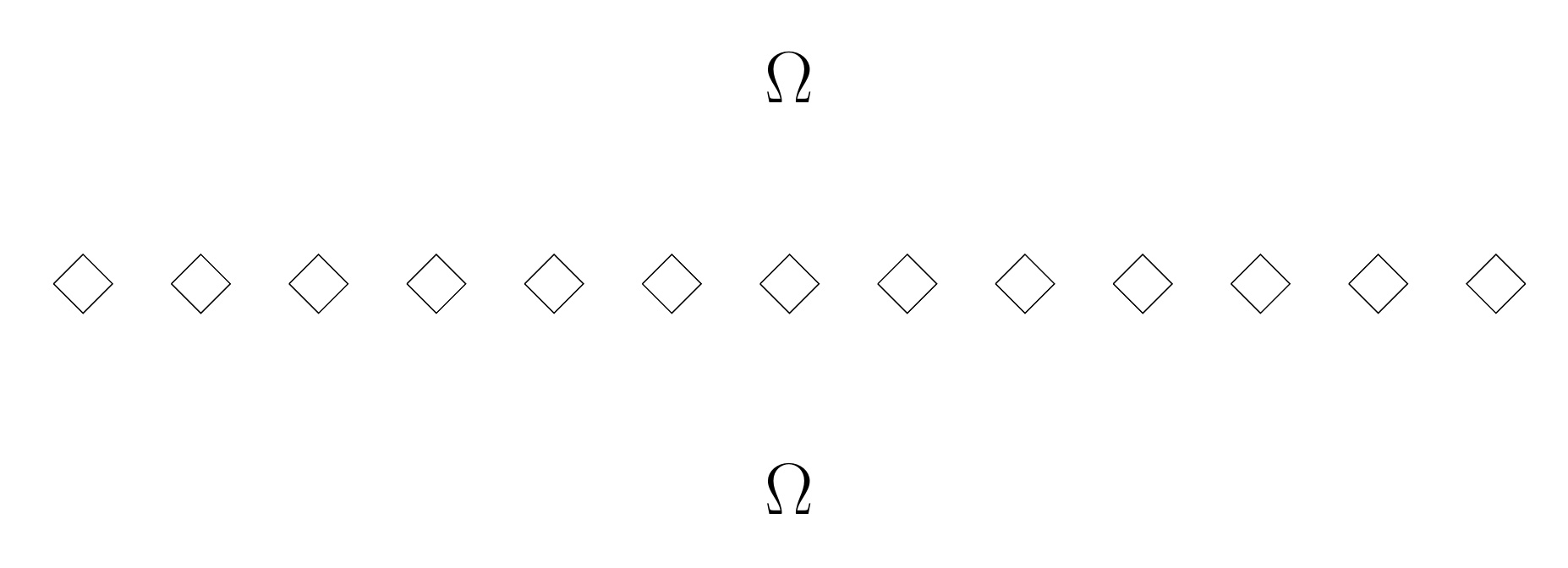}
\end{center}
\caption{The domain $\Omega$} \label{fig1}
\end{figure}
Our domain (see Figure \ref{fig1}) will be 
$$\Omega := \R^2 \setminus \bigcup_{k\in \mathbb Z} \Big\{(x,t)\in \R^2, |x-2k|+|t| < \frac12\Big\}$$
Note that $\partial \Omega$ is uniformly rectifiable, but the domain contains two parts ($\Omega \cap \R^2_+$ and $\Omega \cap \R^2_-$) which are not well connected to each other, that is, this domain does not satisfy the Harnack Chain Condition (see Definition \ref{def1.hc}). We let the reader check that the domain is still semi-uniform.

\medskip

Due to the lack of Harnack chains, the space $\Omega$ does not have a unique - up to constant - Green function with pole at $\infty$. If we take the pole at $t \to -\infty$, then we can construct a positive function $G$ which will be bounded on $\Omega \cap\R^2_+$, and we shall prove that this is incompatible with our estimate \eqref{Main2d} that says that $\frac{\partial_t G}{G}$ is ``close'' to $\frac1t$ when $t$ is large enough.

\subsection{Construction of $G$}

The goal now will be to construct a positive function in $\Omega$, which is morally the Green function with pole at $t=-\infty$. We could have used the usual approach, that is taking the limit when $n$ goes to infinity of - for instance - $G(X,X_n)/G(X_0,X_n)$ in the right sense, where $G$ is the Green function on $\Omega$ for the Laplacian, and $X_n := (1,n)$. However, the authors had difficulty proving the 2-periodicity in $x$ of the limit and didn't know where to find the right properties in the literature (as our domains are unbounded). So we decided to make the construction from scratch.

\medskip

We want to work with the Sobolev space 
\begin{multline*}
W = \Big\{u\in W^{1,2}_{loc}(\overline{\Omega}), \, u(x,t) = u(x+2,t) \text{ and} \, u(x,-t) = u(x,t) \text{ for $(x,t)\in \Omega$}, \\ 
\, \iint_{S_0} |\nabla u(x,t)|^2 dx \, dt < +\infty\Big\}. 
\end{multline*}
Here and in the sequel $S_k$ is the strip $\Omega \cap \Big( [k,k+1) \times \R \Big)$.
Note that due to the 2-periodicity in $x$ and the symmetry, the function $u\in W$ is defined on $\R^2$ as soon as $u$ is defined on any of the sets $S_k$. We will also need 
\[W^+:= \{u_{|\Omega \cap \R^2_+}, \, u \in W\} \text{ and } W_0:= \{u\in W, \, \Tr(u) = 0 \text{ on } \partial \Omega\}.\]
We let the reader check that the quantity
\[\|u\|_W := \left( \iint_{S_0} |\nabla u(x,t)|^2 dx \, dt \right)^\frac12\]
is a norm on the space $W_0$, and the couple $(W_0,\|.\|_W)$ is a Hilbert space. 

The bilinear form 
\[a(u,v):= \iint_{S_0} \nabla u \cdot \nabla v \, dt \, dx\]
is continuous and coercive on $W_0$, so for any $k\in \mathbb N$, there exists $\wt G_k \in W_0$ such that 
\begin{equation} \label{defwtGk}
a(\wt G_k, v ) = \iint_{S_0} \nabla \wt G_k \cdot \nabla v \, dx \, dt = 2^{-k} \int_0^1 \int_{-2^{k+1}}^{-2^k} v(x,t) \, dt\, dx \qquad \text{ for } v\in W_0.
\end{equation}
The first key observation is:

\begin{proposition} \label{prwtGk}
$\wt G_k \in W_0$ is a positive weak solution to $-\Delta u = 0$ in $\Omega \cap \{t>-2^k\}$.
\end{proposition}

\bp The fact that $\wt G_k$ is nonnegative is a classical result that relies on the fact that $u\in W_0 \implies |u|\in W_0$ and the bilinear form $a(u,v)$ is coercive. See for instance \cite{DFMprelim}, (10.18)--(10.20). 

In order to prove that $\wt G_k$ is a solution in $\Omega \cap \{t>-2^k\}$, take $\phi \in C^\infty_0(\Omega \cap \{t>-2^k\})$. For $j\in\mathbb{Z}$, let $\phi_j$ be the only symmetric and 2-periodic function in $x$ such that $\phi_j = \phi$ on $S_j$. Observe that $\phi_j$ is necessary continuous, and so $\phi_j$ lies in $W_0$. Thus
\begin{multline*}
\iint_\Omega \nabla \wt G_k  \cdot \nabla \phi \, dx\, dt 
= \sum_{j\in \mathbb Z} \iint_{S_j} \nabla \wt G_k \cdot \nabla \phi \, dx\, dt 
= \sum_{j\in \mathbb Z}  \iint_{S_j} \nabla \wt G_k \cdot \nabla \phi_j \, dx\, dt \\
= \sum_{j\in \mathbb Z} \iint_{S_0} \nabla \wt G_k \cdot \nabla \phi_j \, dx\, dt = 0
\end{multline*}
by \eqref{defwtGk}, since $\phi_j =\phi \equiv 0$ on $\{t\leq -2^k\}$ for all $j\in\mathbb Z$. 

Since $\wt G_k$ is a solution, which is nonnegative and not identically equal to 0 (otherwise \eqref{defwtGk} would be false), the Harnack inequality (Lemma \ref{Harnack}) entails that $\wt G_k$ is positive. The proposition follows.
\ep

Let $X_0:= (1,0) \in \Omega$. From the above proposition, $\wt G_k(X_0) >0$ so we can define
\begin{equation} \label{defGk}
G_k(X) := \frac{\wt G_k(X)}{\wt G_k(X_0)}.
\end{equation}

\begin{proposition} \label{prGk}
For each $k\in \mathbb N$, the function $G_k(X) \in W_0$ is a positive weak solution to $-\Delta u = 0$ in $\Omega \cap \{t>-2^k\}$. Moreover, we have the following properties:
\begin{enumerate}[(i)]
\item for any compact set $K\Subset \overline{\Omega}$, there exists $k:=k(K)$ and $C:=C(K)$ such that $G_j(X) \leq C_K$ for all $j\geq k$ and $X\in K$ and $\{G_j\}_{j\geq k}$ is equicontinuous on $K$;
\item  there exists $C>0$ such that
\[ \iint_{\Omega \cap ([-2,2] \times [-1,1])} |\nabla G_k(x,t)|^2 dx\, dt \leq C \qquad \text{ for all } k\in \mathbb N;\]
\item there exists $C>0$ such that 
\[\|G_k\|_{W^+}^2 := \iint_{S_0 \cap \R^2_+} |\nabla G_k|^2 dx\, dt \leq C \qquad \text{ for all } k\in \mathbb N.\] 
\end{enumerate}
\end{proposition}

\bp
The fact that $G_k$ is a positive weak solution is given by Proposition \ref{prwtGk}. So it remains to prove (i), (ii) and (iii).

We start with (i). Since $G_k$ is a weak solution in $\Omega_0 := \Omega \cap [(-4,4) \times (-2,2)]$ when $k\geq 1$, and since $\Omega_0$ is a Chord Arc Domain, we can invoke the classical elliptic theory and we can show that there exists $C>0$ such that 
\[ \sup_{\Omega \cap ([-2,2] \times [-1,1])} G_k \leq C G_k(1,0) = C \quad\text{for all }k\ge 1,\]
see for instance Lemma 15.14 in \cite{DFMprelim2}. By the 2-periodicity of $G_k$, it means that 
\[ \sup_{k\geq 1} \sup_{\Omega \cap ( \R \times [-1,1])} G_k \leq C,\]
and then since we can link any point of a compact $K \Subset \overline{\Omega}$ back to $\Omega \cap ( \R \times [-1,1])$ with a Harnack chain (the length of the chain depends on $K$), we have
\[ \sup_{j\geq k} \sup_{K} G_j \leq C_K,\]
whenever $G_j$ is a solution in the interior of $K$, which is bound to happen if $j\geq k(K)$ is large enough. 

The functions $G_k$ are also H\"older continuous up to the boundary in the areas where they are solutions, so $\{G_j\}_{j\geq k}$ is equicontinuous on $K$ as long as $k$ is large enough so that $K \subset \overline{\Omega} \cap \set{t>-2^k}$.

\medskip

Point (ii) is a consequence of the Caccioppoli inequality at the boundary. We only need to prove the bound when $k\geq 2$, since all the $G_k$ are already in $W_0$ by construction. We have by the Caccioppoli inequality at the boundary (see for instance Lemma 11.15 in \cite{DFMprelim2}) that
\begin{multline*}
\iint_{\Omega \cap ([-2,2] \times [-1,1])} |\nabla G_k(x,t)|^2 dx\, dt \lesssim \iint_{\Omega \cap ([-4,4] \times [-2,2])} |G_k(x,t)|^2 dx\, dt \\
\lesssim \sup_{\Omega \cap ([-4,4] \times [-2,2])} |G_k(x,t)|^2 \lesssim 1.
\end{multline*}

\medskip

Point (iii) is one of our key arguments. We define $W_0^+$ as the subspace of $W^+$ that contained the functions with zero trace on $\partial (\Omega \cap \R^2_+)$. 

Since $G_k \in W_0$, its restriction $(G_k)_{|\Omega \cap \R^2_+}$ is of course in $W^+$. Moreover, $G_k$ is a solution to $-\Delta u = 0$ in $\Omega \cap \R^2_+$. We can invoke the uniqueness in Lax-Milgram theorem (see Lemma 12.2 in \cite{DFMprelim2}, but adapted to our periodic function spaces $W_0^+$ and $W^+$) to get that $G_k$ is the only weak solution to $-\Delta u = 0$ in $\Omega \cap \R^2_+$ for which the trace on $\partial (\Omega \cap \R^2_+)$ is $(G_k)_{|\partial (\Omega \cap \R^2_+)}$. Moreover,
\[\|G_k\|_{W^+} \leq C \|(G_k)_{|\partial (\Omega \cap \R^2_+)}\|_{H^{1/2}_{\partial\Omega_+}},\]
where $H^{1/2}_{\partial\Omega_+}$ is the space of traces on $\partial \Omega_+:= \partial (\Omega \cap \R^2_+)$ for the symmetric 2-periodic functions defined as
\begin{multline*}
H^{1/2}_{\partial\Omega_+} := \Big\{ f: \, \partial \Omega_+ \mapsto \R \text{ measurable such that $f$ is symmetric and $2$-periodic in $x$,}  \\
 \text{ and } \|f\|_{H^{1/2}_{\partial\Omega_+}}:= \left( \int_{\partial\Omega_+ \cap S_0} \int_{\partial\Omega_+ \cap S_0} \frac{|f(x)-f(y)|^2}{|x-y|^{3/2}} d\mathcal H^1(x) \, d\mathcal H^1(y)\right)^\frac12 < +\infty \Big\}.
\end{multline*}
So in particular, we have by a classical argument that
\[\|(G_k)_{|\partial (\Omega \cap \R^2_+)}\|^2_{H^{1/2}_{\partial\Omega_+}} \leq C \iint_{\Omega \cap ([-2,2] \times [-1,1])} |\nabla G_k(x,t)|^2 dx\, dt.\]
We conclude that 
\[ \iint_{S_0 \cap \R^2_+} |\nabla G_k|^2 dx\, dt \lesssim  \iint_{\Omega \cap ([-2,2] \times [-1,1])} |\nabla G_k(x,t)|^2 dx\, dt \lesssim 1\]
by (ii). Point (iii) follows.
\ep

\begin{proposition}
There exists a symmetric (in $x$), 2-periodic (in $x$), positive weak solution $G\in W_{\loc}^{1,2}(\om)\cap C(\overline\om)$ to $-\Delta G = 0$ in $\om$ such that $G=0$ on $\pom$ and $G(X_0) = 1$ and
\begin{equation}  \label{GinW+} \iint_{S_0 \cap \R^2_+} |\nabla G|^2 dx\, dt < +\infty.
\end{equation}
\end{proposition}

\bp We invoke the Arzel\`a-Ascoli theorem - whose conditions are satisfied thanks to Proposition \ref{prGk} (i) - to extract a subsequence of $G_k$ that converges uniformly on any compact to a continuous function $G$. The fact $G$ is non-negative, symmetric, 2-periodic, and satisfies $G(X_0) =1$ is immediate from the fact that all the $G_k$ are already like this. The functions $G_k$ converges to $G$ in $W^{1,2}_{loc}(\overline{\Omega})$ thanks to the Caccioppoli inequality, and then by using the weak convergence of $G_k$ to $G$ in  $W^{1,2}_{loc}(\overline{\Omega})$, we can easily prove that $G$ is a solution to $-\Delta u = 0$ in $\Omega$ (hence $G$ is positive by the Harnack inequality, since it was already non-negative). The convergence of $G_k$ to $G$ in $W^{1,2}_{loc}(\overline{\Omega})$ also allow the uniform bound on $\|G_k\|_{W^+}$ given by Proposition \ref{prGk} (iii) to be transmitted to $G$, hence \eqref{GinW+} holds. The proposition follows.
\ep

\subsection{$G$ fails the estimate given in Theorem \ref{Main1}} 

\begin{lemma} \label{lemdtGinW}
$\partial_t G$ is harmonic in $\om$, that is, it is a solution of $-\Delta u=0$ in $\om$, and 
we have
\[\int_1^\infty \int_{0}^1 |\nabla \partial_t G|^2 dx\, dt < +\infty.\]
\end{lemma}

\bp
Morally, we want to prove that if $G$ is a solution (to $-\Delta u = 0$), then $\nabla G \in W^{1,2}$, which is a fairly classical regularity result. The difficulty in our case is that the domain in consideration is unbounded.

Since $G$ is a harmonic function (solution of the Laplacian), the function $g(x):= G(x,1)$ is smooth. We can prove the bound 
\[\int_1^\infty \int_{0}^1  |\nabla \partial_x G|^2 dx\, dt \lesssim \int_0^1 |g'(x)|^2 dx + \int_0^1 |g''(x)|^2 dx + \int_1^\infty \int_{0}^1  |\nabla G|^2 dx\, dt < +\infty \]
by adapting the proof of Proposition 7.3 in \cite{DFMpert-reg} to our simpler context (and invoking \eqref{GinW+} and $g \in C^\infty(\R)$ to have the finiteness of the considered quantities). In order to have the derivative on the $t$-derivative, it is then enough to observe
\begin{multline*} \int_1^\infty \int_{0}^1  |\nabla \partial_t G|^2 dx\, dt \lesssim \int_1^\infty \int_{0}^1  |\partial_x \partial_t G|^2 dx\, dt  +  \int_1^\infty \int_{0}^1 |\partial_t \partial_t G|^2 dx\, dt \\
 =  \int_1^\infty \int_{0}^1 |\partial_t \partial_x G|^2 dx\, dt  +  \int_1^\infty \int_{0}^1 |\partial_x \partial_x G|^2 dx\, dt \\
 \lesssim  \int_1^\infty \int_{0}^1 |\nabla \partial_x G|^2 dx\, dt < +\infty,
\end{multline*}
where we use the fact that $G$ is a solution to $-\Delta u = 0$ - i.e. $\partial_t \partial_t G = - \partial_x \partial_x G$ - for the second line. The lemma follows.
\ep

We will also need a maximum principle, given by

\begin{lemma} \label{maxprinciple} 
If $u$ is a symmetric (in $x$), 2-periodic (in $x$) harmonic function in $\R \times (t_0,\infty)$ that satisfies
\begin{equation} \label{uinWmp} 
\int_{t_0}^\infty \int_{0}^1 |\nabla u|^2 dx\, dt < +\infty,
\end{equation}
then $u$ has a trace - denoted by $\Tr_{t_0} u$ - on $\R\times \{t_0\}$ and
\[\inf_{y\in (0,1)} (\Tr_{t_0} u)(y) \leq u(x,t) \leq \sup_{y\in (0,1)} (\Tr_{t_0} u)(y) \qquad \text{ for all } x\in \R, \, t>t_0.\]
\end{lemma}

\bp
The existence if the trace - in the space $W^{2,\frac12}(\R \times \{t_0\})$ - is common knowledge. The proof of Lemma 12.8 in \cite{DFMprelim2} (for instance) can be easily adapted to prove our case.
\ep

\begin{lemma} \label{lemGeq1}
There exists $C\geq 1$ such that 
\[C^{-1} \leq G(x,t) \leq C \qquad \text{ for } x\in \R, \, t\geq1.\]
\end{lemma}

\bp
Since $G(1,0) = G(X_0) = 1$ and $G$ is a positive solution, the Harnack inequality implies that $C^{-1} \leq G(x,1) \leq C$ for $x\in [0,1]$. Since $G$ is symmetric and 2-periodic in $x$, we have $C^{-1} \leq G(x,1) \leq C$ for $x\in \R$. We conclude with the maximum principle (Lemma \ref{maxprinciple}), since the bound \eqref{uinWmp} is given by \eqref{GinW+}.
\ep

\begin{lemma}  \label{lemdtG<c/t}
For every $c>0$, there exists $t_0\geq 1$ such that 
\[\partial_t G(x,t) \leq \frac{c}{t} \text{ for all $x\in \R$, $t\geq t_0$}.\]
\end{lemma}

\bp Let $x$ be fixed. Since $G$ is symmetric and 2-periodic in $x$, we can assume without loss of generality that $x\in (0,1)$. Then recall that $\partial_t G$ is a weak solution in $\Omega$, so in particular, we have the Moser estimate and the Caccioppoli inequality, which give
\begin{multline} \label{eqbdddtG}
\sup_{y\in \R, \, s>4} s|\partial_t G(y,s)| \lesssim \sup_{y\in \R, \, s>4} s \left(\fint_{s/2}^{2s} \fint_{x-s}^{x+s} |\nabla G(z,r)|^2 dz \, dr\right)^\frac12 \lesssim \sup_{y\in \R, \, s>1} G \lesssim 1.
\end{multline}
by Lemma \ref{lemGeq1}. Moreover, $\partial_t G$ is H\"older continuous, that is,
\begin{multline} \label{eqHolderdtG}
\sup_{y\in (0,1)} |\partial_t G(x,t) - \partial_t G(y,t)| 
\leq Ct^{-\alpha}\br{\fint_{(1-t)/2}^{(1+t)/2}\fint_{t/2}^{3t/2}\abs{\partial_t G(y,s)}^2dsdy}^{1/2}\\
\leq C t^{-\alpha} \sup_{y\in \R, \, s>t/2} |\partial_t G(y,s)| \leq C' t^{-\alpha - 1} \qquad \text{ for } t\geq 8
\end{multline}
by \eqref{eqbdddtG}. 

We pick $t_0 \geq 8$ such that $2C'(t_0)^{-\alpha} \leq c$. Assume by contradiction that there exist $x\in (0,1)$ and $t\geq t_0$ are such that $\partial_t G(x,t) \geq c/t$, then 
\begin{multline*} 
\inf_{y\in \R} \partial_t G(y,t) = \inf_{y\in (0,1)} \partial_t G(y,t) \geq \partial_t G(x,t) - \sup_{y\in (0,1)} |\partial_t G(x,t) - \partial_t G(y,t)| \\
\geq \frac{c - C't^{-\alpha}}{t} \geq \frac{c}{2t} 
\end{multline*}
by our choice of $t_0$. Since $\partial_t G$ is a solution that satisfies \eqref{uinWmp} - see Lemma \ref{lemdtGinW} - the maximum principle given by Lemma \ref{maxprinciple} entails that
\[\partial_t G(y,s) \geq \frac{c}{2t} \qquad \text{ for } y \in \R, \, s>t,\]
which implies
\[\int_0^1 \int_t^{\infty} |\nabla G(y,s)|^2 ds \, dy = +\infty,\]
which is in contradiction with \eqref{GinW+}. We conclude that for every $x\in (0,1)$ and $t\geq t_0$, we necessary have $\partial_t G \leq c/t$. The lemma follows.\ep

\begin{lemma} \label{lemt>t0}
For any $\beta >0$, there exists a $t_0 \geq 1$ and $\epsilon>0$ such that 
\begin{equation} \label{dtG/G-dtD/D>e/t}
\left| \frac{ \partial_t G(x,t)}{G(x,t)} - \frac{\partial_t D_\beta(x,t)}{D_\beta(x,t)} \right| \geq \frac{\epsilon}{t} \qquad \text{ for } x\in \R, \, t\geq t_0.
\end{equation}
\end{lemma}

\bp
The set $\partial \Omega$ is $(n-1)$-Ahlfors regular, so \eqref{equivD} gives the equivalence $D_\beta(X) \approx \dist(X,\partial \Omega)$ for $X\in \Omega$, and hence the existence of $C_1>0$ (depending on $\beta$ and $n$) such that 
\begin{equation} \label{DbDb+2=t}
(C_1)^{-1} t \leq D_\beta(x,t) \leq C_1 D_{\beta+2}(x,t) \leq (C_1)^2 t \qquad \text{ for } x\in \R, \, t\geq 1.
\end{equation}
Check then that 
\[\partial_t D_\beta(x,t) = \frac{d+\beta}{\beta} D_\beta^{1+\beta}(x,t) \int_{(y,s) \in \partial \Omega} |(x,t)-(y,s)|^{-d-\beta-2} (t-s) \, d\sigma(y,s) \]
In particular, since $s\leq \frac12$ whenever $(y,s) \in \partial \Omega$, we have, for $(x,t) \in \R \times [1,\infty)$, that
\begin{multline*}
\partial_t D_\beta(x,t) \geq \Big( t - \frac12 \Big) \frac{n+\beta-1}{\beta} D^{1+\beta}_\beta(x,t) \int_{(y,s) \in \partial \Omega} |(x,t)-(y,s)|^{-n-\beta-1}  \, d\sigma(y,s) \\
 \geq \frac t2  \frac{n+\beta-1}{\beta} D^{1+\beta}_\beta(x,t) D^{-\beta-2}_{\beta+2}(x,t) \geq c_{\beta,n} 
 \end{multline*}
 for some $c_{\beta,n}>0$, by \eqref{DbDb+2=t}. In conclusion, using \eqref{DbDb+2=t} again, we have the existence of $c_1>0$ such that 
 \begin{equation} \label{NDb/D>c/t}
 \frac{\partial_t D_\beta(x,t)}{D_{\beta}(x,t)} \geq \frac{c_1}{t} \qquad \text{ for } x\in \R, \, t\geq 1.
 \end{equation}
 
Let $C_2$ be the constant in Lemma \ref{lemGeq1}. Thanks to Lemma \ref{lemdtG<c/t}, there exists $t_0\geq 1$ such that $\partial_t G(x,t) \leq c_1/(2C_2t)$ for any $x\in \R$ and $t\geq t_0$, which means that 
 \begin{equation} \label{NG/G<c/t}
 \frac{\partial_t G(x,t)}{G(x,t)} \leq \frac{c_1}{2t}  \qquad \text{ for } x\in \R, \, t\geq t_0.
 \end{equation}
 The combination of \eqref{NDb/D>c/t} and \eqref{NG/G<c/t} gives \eqref{dtG/G-dtD/D>e/t} for $\epsilon = c_1/2$.
\ep

\begin{lemma} \label{counterexample}
The positive solution $G$ does not satisfies \eqref{Main1b}, proving that assuming that $\Omega$ is semi-uniform is not sufficient for Theorem \ref{Main1}.
\end{lemma}

\bp
Let $B_r$ be the ball of radius $r$ centered at $(0,\frac12) \in \partial \Omega$, and take $r\geq 2t_0$, where $t_0\geq 1$ is the value from Lemma \ref{lemt>t0}. We have
\begin{multline*}
\iint_{\Omega \cap B_r} \left| \frac{\nabla G}{G} - \frac{\nabla D_\beta}{D_\beta} \right|^2 D_\beta \, dx\, dt \geq \iint_{B_r \cap \{t\geq t_0\}} \left| \frac{\nabla G}{G} - \frac{\nabla D_\beta}{D_\beta} \right|^2 D_\beta \, dx\, dt \\
\geq C^{-1}\epsilon^2\iint_{B_r \cap \{t\geq t_0\}}    \frac{dx\,dt}{t}
\end{multline*}
by \eqref{dtG/G-dtD/D>e/t} and \eqref{equivD}. We conclude that
\[\frac{1}{\sigma(B_r)} \iint_{\Omega \cap B_r} \left| \frac{\nabla G}{G} - \frac{\nabla D_\beta}{D_\beta} \right|^2 D_\beta \, dx\, dt \gtrsim \ln\Big(\frac{r}{t_0} \Big) \rightarrow +\infty \text{ as } r\to \infty,\]
which means that $G$ does not satisfies \eqref{Main1b}. The lemma follows.
\ep


\begin{thebibliography}{AAA}
\bibitem[Azz1]{A} J. Azzam. {\em Harmonic measure and the analyst’s traveling salesman theorem}. Preprint, arXiv: 1905.09057.

\bibitem[Azz2]{Azzam} J. Azzam. {\em Semi-uniform domains and the $A_\infty$ property of harmonic measure}. Int. Math. Res. Not. (2021) no. 9, 6717--6771.

\bibitem[AGMT]{AGMT} J. Azzam, J. Garnett, M. Mourgoglou, and X. Tolsa. {\em Uniform rectifiability, elliptic measure, square functions, $\epsilon$-approximability via an ACF monotonicity formula}. Int. Math. Res. Not. (to appear).


\bibitem[AHMMT]{AHMMT} J.~Azzam, S.~Hofmann, J.~M. Martell, M.~Mourgoglou, and X.~Tolsa. {\em Harmonic measure and quantitative connectivity: geometric characterization of the {$L^p$}-solvability of the {D}irichlet problem.} Invent. Math. {\bf 222} (2020), no. 3, 881--993. 

\bibitem[AHMNT]{AHMNT} J. Azzam, S. Hofmann, J.M.  Martell, K. Nystr{\"o}m,  T. Toro. 
{\em A new characterization of chord-arc domains}, J. Eur. Math. Soc. {\bf 19} (2017), no. 4, 967--981. 

\bibitem[AHM3TV]{AHM3TV} J. Azzam, S. Hoffman, M. Mourgoglou,  J. M. Martell, S. Mayboroda, X.  Tolsa, A. Volberg. {\em Rectifiability of harmonic measure}, Geom. Funct. Anal., {\bf 26} (2016), no. 3, 703--728.

\bibitem[AiHi]{AiHi} Aikawa, Hiroaki, and Kentaro Hirata.
{\em Doubling conditions for harmonic measure in John domains}, In Annales de l'institut Fourier, {\bf 58} (2008) no. 2,  429--445.



\bibitem[BH]{BH} S. Bortz, S. Hofmann.
{\it Harmonic measure and approximation of uniformly rectifiable sets.} Rev. Mat. Iberoam. \textbf{33} (2017), no. 1, 351–373.

\bibitem[BHHLN]{BHHLN} S. Bortz, J. Hoffman, S. Hofmann, J.-L. Luna-Garcia, and K. Nyström. {\em Coronizations and big pieces in metric spaces.} Ann. Inst. Fourier (Grenoble) \textbf{72} (2022), no. 5, 2037–2078.



\bibitem[CFK]{CFK}  L. Caffarelli, 
E. Fabes, 
C. Kenig. 
{\it Completely singular elliptic-harmonic measures.} Indiana Univ. Math. J., {\bf 30} (1981), no. 6, 917--924. 

\bibitem[CFMS]{CFMS} L. Caffarelli, E. Fabes, S. Mortola, S. Salsa.
{\em Boundary behavior of nonnegative solutions of elliptic operators in divergence form.}
Indiana Univ. Math. J., 30 (1981), no. 4, 621--640.


\bibitem[CHM]{CHM} M. Cao, P. Hidalgo-Palencia, J. M. Martell. {\em Carleson measure estimates, corona decompositions, and perturbation of elliptic operators without connectivity.} Preprint, arXiv:2202.06363.


\bibitem[CHMT]{CHMT} J.~Cavero, S.~Hofmann, J.M.~Martell, and T.~Toro. {\em Perturbations of elliptic operators in 1-sided chord-arc domains. Part II: Non-symmetric operators and Carleson measure estimates.} Trans. of the AMS. (11) \textbf{373} (2020), 7901--7935.

\bibitem[Chr]{Ch} M. Christ,  {\em A $T(b)$ theorem with remarks on analytic
capacity and the Cauchy integral.} Colloq. Math., {\bf LX/LXI} (1990), 601--628.







\bibitem[D1]{David88}  G.\,David. {\em Morceaux de graphes lipschitziens et int\'egrales singuli\`eres sur une surface. (French) [Pieces of Lipschitz graphs and singular integrals on a surface]}. Rev. Mat. Iberoamericana \textbf{4} (1988), no. 1, 73--114.

\bibitem[D2]{David91} G.\,David. 
{\em Wavelets and singular integrals on curves and surfaces.} 
Lecture Notes in Mathematics, \textbf{1465}. Springer-Verlag, Berlin, 1991.

\bibitem[DJ]{DJ} G. David, D. Jerison.  
{\it Lipschitz approximation to hypersurfaces, harmonic measure, and singular integrals.} 
Indiana Univ. Math. J., {\bf 39} (1990), no. 3, 831--845.

\bibitem[DS1]{DS1} G. David, S. Semmes.
{\em Singular integrals and rectifiable sets in $\R^n$: Beyond Lipschitz graphs.}
Asterisque, {\bf 193} (1991).

\bibitem[DS2]{DS2} G. David, S. Semmes.
{\em Analysis of and on uniformly rectifiable sets.} 
Mathematical Surveys and Monographs, {\bf 38}. American Mathematical Society, Providence, RI, 1993.

\bibitem[DS3]{DS3} G. David, S. Semmes.
{\em Fractured fractals and broken dreams: self-similar geometry through metric and measure.} Vol. 7. Oxford University Press, 1997.

\bibitem[DEM]{DEM} G. David, M. Engelstein, S. Mayboroda. {\em Square functions, non-tangential limits and harmonic measure in co-dimensions larger than 1}. Duke Math. J. {\bf 170} (2021), no 3, 455--501.


\bibitem[DFM1]{DFMprelim} G. David, J.  Feneuil, S. Mayboroda. {\em Elliptic theory for sets with higher co-dimensional boundaries}.  Mem. Am. Math. Soc. \textbf{273} (2021), no 1346, iii+123 pp.

\bibitem[DFM2]{DFM3}
G. David, J. Feneuil, S. Mayboroda.  {\it Dahlberg's theorem in higher co-dimension.} Journal of Functional Analysis, {\bf 276} (2019), no. 9 2731--2820.

\bibitem[DFM3]{DFMprelim2} G. David, J.  Feneuil, S. Mayboroda. {\em Elliptic theory in domains with boundaries of mixed dimension}.  Preprint, arXiv:2003.09037.

\bibitem[DFM4]{DFMGinfty}
G. David, J. Feneuil, S. Mayboroda.  {\em Green function estimates on complements of low-dimensional uniformly rectifiable sets.} Math. Ann. (accepted).

\bibitem[DFM5]{DFMpert-reg}
Z. Dai, J. Feneuil, S. Mayboroda.  {\em The regularity problem in domains with lower dimensional boundaries.} Preprint, arXiv:2208.00628.

\bibitem[DLM1]{DLM1}
G. David, L. Li, S. Mayboroda.  {\em Carleson measure estimates for the Green function.} Arch. Rational. Mech. Anal. {\bf 243} (2022), 1525–1563.

\bibitem[DLM2]{DLM2}
G. David, L. Li, S. Mayboroda.  {\em Carleson estimates for the Green function on domains with lower dimensional boundaries.} J. Funct. Anal. {\bf 283} (2022), no. 5, \url{https://doi.org/10.1016/j.jfa.2022.109553}.

\bibitem[DM1]{DM1}
G. David, S. Mayboroda.  {\em Harmonic measure is absolutely continuous with respect to the Hausdorff measure on all low-dimensional uniformly rectifiable sets.} Int. Math. Res. Not. (2022), rnac109, \url{https://doi.org/10.1093/imrn/rnac109}.

\bibitem[DM2]{DM2}
G. David, S. Mayboroda.  {\em Approximation of Green functions and domains with uniformly rectifiable boundaries of all dimensions.} Adv. Math. {\bf 410} (to appear).











\bibitem[Fen1]{FenUR} J. Feneuil.
{\em Absolute continuity of the harmonic measure on low dimensional rectifiable sets.} J. Geom. Anal. \textbf{32} (2022), no. 10, Paper No. 247, 36 pp.

\bibitem[Fen2]{FenDKP} J. Feneuil.
{\em A change of variable for Dahlberg-Kenig-Pipher operators.} Proc. Am. Math. Soc. \textbf{150} (2022), no. 8, 3565--3579.


\bibitem[GMT]{GMT} J. Garnett, M. Mourgoglou, X. Tolsa. {\em Uniform rectifiability from Carleson measure estimates and $\epsilon$-approximability of bounded harmonic functions.}
Duke Math. J. \textbf{167} (2018), no. 8, 1473--1524. 



\bibitem[HM1]{HM1} S. Hofmann, J.M. Martell.  {\em Uniform rectifiability and harmonic measure {I}: uniform rectifiability implies {P}oisson kernels in {$L^p$}}. Ann. Sci. \'Ec. Norm. Sup\'er. (4), {\bf 47} (2014), no. 3, 577--654.

\bibitem[HM2]{HM2} S. Hofmann, J.M. Martell. {\em Uniform Rectifiability and harmonic measure IV: Ahlfors regularity plus Poisson kernels in Lp implies uniform rectifiability.} Preprint, arXiv:1505.06499.

\bibitem[HMM1]{HMMDuke}
S. Hofmann, J.M. Martell, S. Mayboroda. 
{\em Uniform rectifiability, Carleson measure estimates, and approximation of harmonic functions}. Duke Math. J. {\bf 165} (2016), no. 12, 2331--2389.

\bibitem[HMM2]{HMMtrans}
S. Hofmann, J.M. Martell, S. Mayboroda. 
{\em Transference of scale-invariant estimates from Lipschitz to Non-tangentially accessible to Uniformly rectifiable domains}. Preprint, arXiv:1904.13116.

\bibitem[HMMTZ]{HMMTZ}
S. Hofmann, J.M. Martell, S. Mayboroda, T. Toro, Z. Zihui. 
{\em Uniform rectifiability and elliptic operators satisfying a Carleson measure condition.} Geom. Funct. Anal. {\bf 31} (2021), no. 2, 325--401.
  
\bibitem[HMU]{HMU} S. Hofmann, J.M. Martell, I. Uriarte-Tuero. {\em Uniform rectifiability and harmonic measure, {II}: {P}oisson kernels in {$L^p$} imply uniform rectifiability.} Duke Math. J., {\bf 163} (2014), no. 8, 1601--1654.


\bibitem[HMT1]{HMT} S. Hofmann, J.M. Martell, T. Toro. {\em $A_\infty$ implies NTA for a class of variable coefficient elliptic operators.} J. Differential Equations {\bf 263} (2017), no. 10, 6147--6188.

\bibitem[HMT2]{HMT2} S. Hofmann, J.M. Martell and T. Toro,
{\em General divergence form elliptic operators on domains with ADR boundaries,
and on 1-sided NTA domains.} In progress.

\bibitem[JK]{JK} D. Jerison and C. Kenig. {\em The Dirichlet problem in nonsmooth domains.} Ann. of Math. (2) {\bf 113}  (1981), no. 2, 367--382.





\bibitem[Ken]{KenigB} C. E. Kenig.
{\em Harmonic analysis techniques for second order elliptic boundary value problems.} 
CBMS Regional Conference Series in Mathematics, {\bf 83}. Amer. Math. Soc., Providence, RI, 1994.



\bibitem[KP]{KePiDrift} C. Kenig, J. Pipher.  
{\it The Dirichlet problem for elliptic equations with drift terms.} 
Publ. Mat., {\bf 45} (2001), no. 1, 199--217. 





\bibitem[MM]{MM} L. Modica, S. Mortola. {\em Construction of a singular elliptic-harmonic measure.} Manuscripta Math. {\bf 33} (1980/81), no. 1, 81--98.


\bibitem[MT]{MT} M. Mourgoglou, X. Tolsa. {\em The regularity problem for the Laplace equation in rough domains}. Preprint, arXiv:2110.02205.

\bibitem[MPT]{MPT} M. Mourgoglou, B. Poggi, X. Tolsa. {\em $L^p$-solvability of the Poisson-Dirichlet problem and its applications to the regularity problem}. Preprint, arXiv:2207.10554.

\bibitem[NTV]{NTV} F. Nazarov, X. Tolsa and A. Volberg, \emph{On the uniform rectifiability of {AD}-regular measures with bounded {R}iesz transform operator: the case of codimension 1}. Acta Math. \textbf{213} (2014), no.~2, 237--321. 



\bibitem[Ste]{Stein93} E. M. Stein.{\em Harmonic analysis: real-variable methods, orthogonality, and oscillatory integrals}. Princeton Mathematical Series, {\bf 43}. Princeton University Press, Princeton, N.J., 1993.

\bibitem[Tol]{Tolsa09} X. Tolsa.
{\em Uniform rectifiability, {C}alder\'on-{Z}ygmund operators with odd kernel, and quasiorthogonality.}
{Proc. Lond. Math. Soc. (3)}, {\bf 98} (2009), no. 2, 393--426.








\end{thebibliography}
\end{document}